\documentclass[11pt,reqno]{article}
\usepackage{amsmath}
\usepackage{amsthm}
\usepackage{amssymb}
\usepackage{amsfonts}
\usepackage{cases}
\usepackage{graphicx}
\usepackage{xcolor}
\usepackage[margin=1in]{geometry}
\usepackage[utf8]{inputenc}
\usepackage[margin=1in]{geometry}
\usepackage{verbatim}
\usepackage{caption,subcaption}
\usepackage[colorlinks,citecolor=blue,urlcolor=blue]{hyperref}

\DeclareRobustCommand{\vect}[1]{\bm{#1}}
\usepackage{cleveref}

\newcounter{sentenceI}
\renewcommand{\thesentenceI}{I}

\newcommand{\sentenceI}[1]{%
	\refstepcounter{sentenceI}%
	\label{#1}%
	\thesentenceI%
}

\newcounter{sentenceII}
\renewcommand{\thesentenceII}{II}

\newcommand{\sentenceII}[1]{%
	\refstepcounter{sentenceII}%
	\label{#1}%
	\thesentenceII%
}

\crefname{sentenceI}{sentence}{sentences}
\crefname{sentenceII}{sentence}{sentences}

\usepackage{bm}
\usepackage{upgreek}

\usepackage{soul}

\makeatletter

\numberwithin{equation}{section}
\newtheorem{Definition}{Definition}[section]
\newtheorem{Remark}{Remark}[section]
\newtheorem{Theorem}{Theorem}[section]
\newtheorem{Lemma}{Lemma}[section]
\newtheorem{Proposition}{Proposition}[section]
\newtheorem{Corollary}{Corollary}[section]
\newtheorem{Assumption}{Assumption}[section]
\newtheorem{Example}{Example}[section]
\newcommand{\be}{\begin{equation}}
\newcommand{\ee}{\end{equation}}
\newcommand{\bee}{\begin{equation*}}
\newcommand{\eee}{\end{equation*}}
\newcommand{\bi}{\begin{itemize}}
\newcommand{\ei}{\end{itemize}}

\DeclareMathOperator*{\argmax}{arg\,max}

\def \E{\mathbb{E}}

\def \N{\mathbb{N}}
\def \P{\mathbb{P}}

\def \R{\mathbb{R}}

\def \Pc{{\mathcal P}}

\def \Wc{{\mathcal W}}

\def \eps{\varepsilon}
\def \supp{\operatorname{\texttt{supp}}}

\def \pitildebm{\bm{\pi}}

\def \T{{\mathbb T}}

\def \bmpi{\bm{\pi}}

\def \pitilde{{\tilde \pi}}

\def \Pitilde{{\widetilde \Pi}}
\def \pitildebm{{\bm{\tilde \pi}}}

\def \pihat{{\hat \pi}}
\def \muhat{{\hat \mu}}

\newcommand{\Vast}{\bBigg@{3}}

\title{On Time-Inconsistency in Mean Field Games}

\author{Erhan Bayraktar\thanks{
		Department of Mathematics, University of Michigan, email: \texttt{erhan@umich.edu}. E. Bayraktar is partially supported by the National Science Foundation under grant DMS2106556 and by the Susan M. Smith chair.}	
	\and Zhenhua Wang\thanks{
		Department of Mathematics, Iowa State University, email: \texttt{zhenhuaw@iastate.edu}. }
}	

\date{}

\begin{document}

\maketitle

\begin{abstract}
We investigate an infinite-horizon time-inconsistent mean-field game (MFG) in a discrete time setting. We first present a classic equilibrium for the MFG and its associated existence result. This classic equilibrium aligns with the conventional equilibrium concept studied in MFG literature when the context is time-consistent. Then we demonstrate that while this equilibrium produces an approximate optimal strategy when applied to the related $N$-agent games, it does so solely in a precommitment sense. Therefore, it cannot function as a genuinely approximate equilibrium strategy from the perspective of a sophisticated agent within the $N$-agent game. To address this limitation, we propose a new {\it consistent} equilibrium concept in both the MFG and the $N$-agent game.
We show that a consistent equilibrium in the MFG can indeed function as an approximate consistent equilibrium in the $N$-agent game. Additionally, we analyze the convergence of consistent equilibria for $N$-agent games toward a consistent MFG equilibrium as $N$ tends to infinity.
\end{abstract}

{
	\hypersetup{linkcolor=black}
	\tableofcontents
}

\section{Introduction}
Mean field game (MFG) theory has attracted a lot of attention during the last decade. This field studies limiting models for weakly interacting $N$-agent stochastic games. The primary focus of existing studies on MFG in both discrete-time and continuous-time frameworks lies within the time-consistent context, and research in this area encompasses the analysis of two key topics. First, it has been shown that an equilibrium in the (time-consistent) MFG can serve as an approximate equilibrium in the (time-consistent) $N$-agent game when $N$ is large enough; see, e.g. \cite{carmona2013probabilistic, saldi2018markov, anahtarci2023q}. Secondly, efforts have also been made to show that equilibria in (time-consistent) $N$-agent games can approach an equilibrium in the (time-consistent) MFG as $N$ tends to infinity; see, e.g., \cite{gomes2013continuous, lacker2016general, lacker2017limit, fischer2017connection, bayraktar2018analysis, cardaliaguet2019master, bayraktar2022finite, bayraktar2023mean}. 

Given the flow of the population distribution, the MFG reduces to a stochastic control problem for a single agent. It is widely acknowledged that a non-exponentially discounted reward function induces time inconsistency in a stochastic control problem: an optimal control policy derived by an agent today will not be optimal from the eyes of the same agent tomorrow. 
As proposed in Strotz \cite{strotz1955myopia}, a {\it sophisticated} agent treats a time-inconsistent control problem as an intertemporal game among current self and future selves and looks for a subgame perfect Nash equilibrium: a consistent strategy that, given the future selves follow this strategy, the current self has no incentive to deviate from it. Nash equilibrium in time-inconsistent stochastic controls has been widely explored; for recent developments, see \cite{yong2012time, ekeland2012time, MR4288523, bjork2021time, bayraktar2023relaxed, bayraktar2023existence} and the references therein.

In this paper, we consider an infinite horizon discrete time MFG over a finite state space and its large (but finite) population counterparts. Each agent endeavors to maximize the total expected reward along the time horizon. The reward function is presumed to be in general discounted form, inherently conferring time inconsistency to the context. Consequently, a sophisticated agent shall treat the game as a competition not only with other agents but also with all future selves and try to look for an equilibrium strategy from this perspective. 
Through this paper, we aim to establish suitable equilibrium concepts for the time-inconsistent MFG and its $N$-agent counterparts, and we seek to address the following questions in the context of time-inconsistency.
\bi 
\setlength{\itemindent}{-.37in}
\item[] Question (\sentenceI{sentence:I}): Whether or not an equilibrium for the MFG serves as an approximate equilibrium in $N$-agent games for sufficiently large $N$?
\item[] Question (\sentenceII{sentence:II}): Whether or not a sequence of equilibria in $N$-agent games approaches an equilibrium for the MFG as $N$ tends to infinity? 
\ei
The main contributions of this paper are three-fold:
\bi 
\item We introduce a classic equilibrium concept in Definition \ref{def.nonstationary.relax} for the MFG in the context of time-inconsistency and establish its general existence through Theorem \ref{thm.exist.nonstatoinary}. This equilibrium concept is a truly Nash equilibrium from the perspective of a sophisticated agent (see Remark \ref{rm.explain.def}) within the MFG. It is classic in the sense that it aligns with the established equilibrium concepts in the mainstream of (time-consistent) MFG, as elucidated in Remarks \ref{rm.nonstationary}, \ref{rm:equi.stationary}.

\item We show in Proposition \ref{prop.precommittedN} 
that a classic MFG equilibrium in Definition \ref{def.nonstationary.relax} serves an approximate optimal policy within the $N$-agent game,
provided that the following condition hold:
	\begin{flushleft}
The initial states of N agents, $\left( s^{N,i}_0 \right)_{1\leq i\leq N}$, are independently simulated by the initial population distribution $\mu_0$ in the classic MFG equilibrium.
	\end{flushleft}
This assumption of initial states aligns with the conventions prevalent in the (time-consistent) MFG literature when applying an MFG equilibrium in an $N$-agent game. \footnote{Another condition assumed in the literature of time-consistent MFG is that the initial empirical $N$-agent distribution $\mu^N_0$ converges to $\mu_0$, which follows from the above i.i.d. assumption.} However, our exploration in Section \ref{subsec:explain} exposes that such an assumption compromises the outcome in Proposition \ref{prop.precommittedN} in a precommitment manner, thereby disqualifying an equilibrium in Definition \ref{def.nonstationary.relax} for the MFG as a truly approximate equilibrium from the standpoint of a sophisticated agent within the N-agent game.

\item We propose a novel equilibrium concept called {\it consistent equilibrium} in Definition \ref{def:sophisticated.ne} for the MFG and in Definition \ref{def:N.sophisticated} for the N-agent games, respectively. The consistent MFG equilibrium concept in Definition \ref{def:sophisticated.ne} surpasses the classic MFG equilibrium concept defined in Definition \ref{def.nonstationary.relax}, as highlighted in Remark \ref{rm:sophisticated.give.classic}. Furthermore, we prove in Theorem \ref{thm.MFG.Neps} that a consistent MFG equilibrium indeed functions as a genuine approximate consistent equilibrium from the perspective of a sophisticated agent within the $N$-agent game. Additionally, we establish in Theorem \ref{thm.Nequi.converge} that a sequence of consistent equilibria in $N$-agent games can converge to a consistent MFG equilibrium. Through these analyses, we provide answers to Questions \eqref{sentence:I} \& \eqref{sentence:II}.
\ei 

Literature on the time-inconsistent MFG is rather limited and focuses mainly on linear-quadratic games; see \cite{ni2017time, moon2020linear, wang2023time}. We want to emphasize that these papers define equilibrium in a similar manner as Definition \ref{def.nonstationary.relax}, thus having the same limitation as shown in Section \ref{subsec:explain} when applied within the $N$-agent games. Additionally, it's worth noting three recent articles \cite{djehiche2023time, yu2023time, wang2024sharp} on time-inconsistent mean field stopping problems. In \cite{djehiche2023time}, precommitted strategies are explored under recursive utilities, while \cite{yu2023time} contributes to equilibrium strategies adapted by a social controller across the entire population. {\color{black}\cite{wang2024sharp} studies the so-called sharp equilibrium strategies for mean-field stopping games and shows the existence of a sharp equilibrium in two classes of examples.}

The rest of the paper is organized as follows. Section \ref{sec:MFG.classicequilibrium} introduces the MFG model in the context of time-inconsistency. We present
a classic MFG equilibrium concept (see Definition \ref{def.nonstationary.relax}) in Section \ref{subsec:MFG.equilibrium}, while the proof of its general existence (see Theorem \ref{thm.exist.nonstatoinary}) is provided in Section \ref{subsec:proof.nonstationary}. 
Then, we connect the MFG to its $N$-agent game counterpart in Section \ref{sec:Nagent}. Section \ref{subsec:Ngame.notations} describes the N-agent game model in the context of time-inconsistency and demonstrates an approximate optimality result (see Proposition \ref{prop.precommittedN}) when applying a classic MFG equilibrium in Definition \ref{def.nonstationary.relax} to the N-agent games. In Section \ref{subsec:explain}, we clarify that this approximate optimality result applies only in a precommitment manner, indicating that an MFG equilibrium, as defined in Definition \ref{def.nonstationary.relax}, is not a truly approximate equilibrium in the eyes of a sophisticated agent within the $N$-agent game. To address this limitation, we propose new equilibrium concepts, consistent equilibrium, for both the MFG and the $N$-agent game (see Definition \ref{def:sophisticated.ne} and Definition \ref{def:N.sophisticated}, respectively) in Section \ref{sec:sophisticated}. Then we address Question \eqref{sentence:I} and Question \eqref{sentence:II} in Theorem \ref{thm.MFG.Neps} 
and in Theorem \ref{thm.Nequi.converge}, respectively, which are the main results in Section \ref{sec:sophisticated}.
We also show in an example study the existence of consistent MFG equilibria that are continuous on population distribution. Section \ref{subsec:proof.sophisticated} summarizes the proofs for Theorems \ref{thm.MFG.Neps}, \ref{thm.Nequi.converge}, {\color{black} and Section \ref{sec:conclusion} concludes the paper.}

\section{Mean field games under the classic equilibrium notion}\label{sec:MFG.classicequilibrium}
Denote by $\T:=\N\cup\{0\}$ the infinite discrete time horizon and {\color{black}$[d]:=\{1,2,...,d\}$ a general state space of size $d$ with discrete metric}. Denote by $a\cdot b$ the product of matrices $a,b$, and we treat a $d$-dimensional row (resp. column) vector as a matrix with dimension equal to $d\times 1$ (resp. $1\times d$). Given $a\in \R^d$, denote by $a(x)$ the $x$-th component of $a$ for $x\in[d]$. For a multivariable function $w(x,\nu)$, denote by $w(\cdot,\nu)$ (resp. $w(x,\cdot)$) the function $x\mapsto w(x,\nu)$ (resp. $\nu\mapsto w(x,\nu)$) when $\nu$ is fixed (resp. $x$ is fixed).

Denote by $\Pc(X)$ all probability measures on a given Borel space $X$.  Then 
$$\Pc([d])= \left\{ \nu=(\nu(1),\cdots,\nu(d)): \nu(x)\geq 0 \;\ \forall x\in[d] \text{ and } \sum_{x\in[d]} \nu(x)=1 \right\},$$ 
is the $d$-simplex, which we treat as a subset in $\R^{d-1}$ and is endowed with the Euclidean metric $|\cdot|$. We further define $\Delta:= [d]\times \Pc([d])$, which is endowed with the product topology.

We take a bounded domain $U\subset\R^\ell$ with positive Lebesgue measure, which plays the role of the action value space during the paper. Through the paper, we endow $\Pc(U)$ with the Wasserstein 1-distance $\Wc_1$. Since $U$ is a bounded domain in $\R^\ell$, the Wasserstein 1-distance on $\Pc(U)$ can be expressed, for $\varpi, \bar\varpi\in \Pc(U)$ that
\be\label{eq.cha.W1}  
\Wc_1 (\varpi, \bar\varpi) =\sup\left\{ \int_U h(u)(\varpi-\bar\varpi)(du): \sup_{u\in U} |h(u)|\leq 1, \sup_{u,\bar u\in U}|h(u)-h(\bar u)|\leq |u-\bar u| \right\},
\ee 
and $\Wc_1$ generates the weak topology on $\Pc(U)$ (for a reference, see \cite[Theorem 8.3.2, Theorem 8.10.45]{bogachev2007measure}). We also denote by $\supp(\varpi)$ the support of a probability measure $\varpi\in \Pc(U)$.

\subsection{Mean field games with time-inconsistency}\label{subsec:MFG.equilibrium}
We now consider an MFG with state space $[d]$ over time horizon $\T$. Let $f(t,x,\nu,u):  \T\times \Delta\times U\mapsto \R$ denote the reward function, which is not necessarily exponentially discounted and thus introduces time-inconsistency in the game. We denote by $P(x,\nu, y, u): \Delta\times [d]\times U \to [0,1]$ the (time-homogeneous) transition function such that $P(x,\nu, y, u)$ represents the probability that a single agent visits the state $y$ from the state $x$ by applying the control value $u$ under population distribution $\nu$.
Denote by $P(x,\nu, u)$ (resp. $P(\nu,u)$)
the row vector $(P(x,\nu,y, u))_{y\in[d]}$ 
(resp. the matrix $(P(x,\nu,y, u))_{x,y\in[d]}$). Furthermore, for any $t\in \T, x,y\in [d]$, the functions $f(t,x,\nu,u), P(x,\nu,y,u)$ are assumed to be {\color{black}Borel measurable on $\nu$ and continuous on $u$.}

Throughout this paper, we focus on relaxed type policies, represented as $\Pc(U)$-valued controls, rather than $U$-valued controls. 
Indeed, \cite{bayraktar2023relaxed} shows that a $U$-valued equilibrium may not even exist for a time-inconsistent stochastic control problem unless certain strict concavity structures are assumed. Therefore, a $U$-valued equilibrium may also not exist for the MFG in the context of time-inconsistency, and it is natural to consider the larger family of $\Pc(U)$-valued policies, which contains all $U$-valued policies by treating a control value $u$ as a Dirac measure in $\Pc(U)$. 

For arbitrary $\varpi\in \Pc(U)$ and a mapping $\upvarpi =\upvarpi(x): [d]\mapsto \Pc(U)$, we further define
\be\label{eq.P} 
\begin{aligned}
	& P^\varpi(x,\nu,y):=\int_U P(x, \nu,y, u)\varpi(du),\quad \forall x, y\in[d], \nu\in \Pc([d]), t\in \T;\\
	&P^\varpi(x,\nu):=\left( P^\varpi(x, \nu, y)\right)_{y\in[d]},
	\quad 
	 f^\varpi(t,x,\nu):=\int_U f(t,x, \nu,u)\varpi(du),\quad \forall (x,\nu)\in \Delta, t\in \T;\\
	 & P^{\upvarpi}(\nu):=\left(\int_U P(x, \nu, y, u)\upvarpi(x)(du) \right)_{x,y\in[d]}\quad \forall \nu\in \Pc([d]).
\end{aligned}
\ee
That is, $P^{\upvarpi}(\nu)$ represents the transition matrix with the $x$-th row equal to $P^{\upvarpi(x)}(x,\nu)$ for each $x\in[d]$.

A relaxed type policy $\pi$ dependent on both time and state takes the form $\pi=(\pi_t(x))_{x\in[d], t\in \T}: \T\times [d]\mapsto \Pc(U)$. We further write $\pi_t$ for the vector $(\pi_t(x))_{x\in [d]}$ and denote by $\Pi$ the set that contains all such time-state dependent relaxed policies. A population flow can be treated as a mapping $\mu=(\mu_t)_{t\in \T}: \T\mapsto \Pc([d])$, and we denote by $\Lambda$ the set that contains all such mappings. 
Given a policy $\pi\in \Pi$ (resp. a population flow $\mu\in \Lambda$) and two time values $0\leq t_0<t_1\leq \infty$, {\color{black}we shall apply the short notation $\pi_{[t_0:t_1)}$ (resp. $\mu_{[t_0:t_1)}$) for the segment of flow $(\pi_t)_{t\in \T}$ from $t_0$ to $t_1-1$ (resp. for the segment of flow $(\mu_t)_{t\in \T}$ from $t_0$ and $t_1-1$),  and write $\pi=\pi_{[0,\infty)}=(\pi_t)_{t\in \T}$ (resp. $\mu=\mu_{[0,\infty)}=(\mu_t)_{t\in \T}$). Moreover, given $\pi,\pi'\in \Pi$ and two integers $0\leq t_0<t_1$, we denote by $\pi'_{[t_0: t_1)}\otimes \pi_{[t_1:\infty)}$ the new pasting policy for applying $\pi'$ through steps $t_0$ to $t_1-1$ and switching to $\pi$ afterward.} 

We denote by $(s^{\pi,\mu}_t)_{t\in \T}$ the dynamic of a single agent for applying a policy $\pi\in \Pi$ under a given population flow $\mu\in \Lambda$. At each step $t$, $s^{\pi,\mu}_{t+1}$ is generated according to $P \left( s^{\pi,\mu}_{t},\mu_t, \alpha_t \right)$, where $\alpha_t$ is a random variable  valued in $U$ and is simulated independently by $\pi_t\left(s^{\pi,\mu}_t \right)$.
For such $(\pi,\mu)\in \Pi \times \Lambda$, we further define, for each $t\in \T$ that
\be\label{eq.def.J} 
J^{\pi_{[t:\infty)}}\left(x, \nu,\mu_{[t+1:\infty)}\right):= \E\left[ \sum_{l\in \T} f^{\pi_{t+l}\left( s^{\pi,\mu}_{t+l} \right)}\left(l, s_{t+l}^{\pi,\mu}, \mu_{t+l} \right)  \bigg | s^{\pi,\mu}_t=x, \mu_t=\nu \right].
\ee 
Notice that  $\mu_t=\nu$ and $\mu_{[t+1:\infty)}$ together consist the tail of the population flow after step $t-1$. Then $J^{\pi_{[t:\infty)}}\left(x, \nu,\mu_{[t+1:\infty)}\right)$ represents the expected total pay-off after $t-1$  for a single agent (conditional on that $s^{\pi,\mu}_t=x$ and $\mu_t=\nu$). And we write $J^{\pi}\left(x, \nu,\mu_{[1:\infty)}\right)$ for short when $t=0$ in \eqref{eq.def.J}, which represents the expected future pay-off at time 0 for starting at state $x$ with $\mu_0=\nu$.


Given $\pi\in\Pi$ and $\nu\in \Pc([d])$, denote by $\mu^{\pi,\nu}:=\mu^{\pi, \nu}_{[0:\infty)}$ the population flow starting at $\nu$ for all agents applying the same policy $\pi$. That is, $\mu^{\pi, \nu}_{[0:\infty)}: \T\mapsto \Pc([d])$ satisfies
\be\label{eq:mu.pi} 
\mu_{t+1}^{\pi,\nu} = \mu_t^{\pi, \nu} \cdot P^{\pi_t}\left( \mu_t^{\pi, \nu} \right) \quad  t\in \T,\quad \text{and }\mu_0^{\pi, \nu}=\nu. 
\ee
Notice that the flow $\mu^{\pi, \nu}_{[0:\infty)}$ is fully determined by $\pi\in \Pi$ and initial value $\nu$.

We now provide the main assumption in the paper, which ensures that all value functions throughout the paper are well-defined. 
\begin{Assumption}\label{assum.bound}
	The function $t\mapsto \underset{(x,\nu)\in \Delta, u\in U}{\sup}|f(t,x,\nu,u)|$ is decreasing with 
	$$  
	\sum_{t\in \T} \sup_{(x,\nu)\in \Delta, u\in U} |f(t,x,\nu,u)|<\infty.
	$$
\end{Assumption}

\begin{Definition}\label{def.nonstationary.relax}
	Take an initial population distribution $\nu\in \Pc([d])$. Consider a pair $(\pi,\mu)\in \Pi\times\Lambda$ with $\mu_0=\nu$. We call $(\pi,\mu)$ is an equilibrium for the MFG with initial population $\nu$ 
	if 
	\bi 
	\item[A.] $\mu$ is the deterministic flow for all agents applying $\pi$, that is,
	\be\label{eq.def.nonstationary0}  
	\mu_{t} = \mu^{\pi, \nu}_{t}
	\quad \forall t\geq 1.
	\ee 
	\item[B.] 
	For each $t\in \T$,
	\be\label{eq.def.nonstationary1}  
	J^{\pi_{[t:\infty)}}\left(x, \mu_t, \mu_{[t+1:\infty)}\right) =\sup_{\pi'\in \Pi}	J^{\pi'_{[t:t+1)} \otimes\pi_{[t+1:\infty)}}\left( x,\mu_t, \mu_{[t+1:\infty)}\right)\quad \forall x\in [d].
	\ee
	\ei 
\end{Definition}

Due to the discrete time setting, we can replace $\sup_{\pi'\in \Pi}	J^{\pi'_{[t:t+1)} \otimes\pi_{[t+1:\infty)}}\left( x,\mu_t, \mu_{[t+1:\infty)}\right)$ on the right-hand side of \eqref{eq.def.nonstationary1} by 
\be\label{eq:mfg.varpi} 
\sup_{\varpi\in \Pc(U)}	J^{\varpi \otimes_1 \pi_{[t+1:\infty)}}\left( x,\mu_t, \mu_{[t+1:\infty)}\right),
\ee 
where $\varpi \otimes_1\pi_{[t+1:\infty)}$ means applying $\varpi$ at $t$ for one step then back to $\pi$ afterward.
\begin{Remark}\label{rm.explain.def}

	Condition B in Definition \ref{def.nonstationary.relax} tells that a deviation from the policy $\pi$ at any moment $t$ will result in a no-better payoff for a single agent. Thus, Conditions A and B in Definition \ref{def.nonstationary.relax} together imply that a single agent shall keep following $\pi$ at any $t$, provided all other agents also follow $\pi$ along the time. This illustrates that $\pi$ is a time-consistent strategy, that is, a subgame perfect Nash equilibrium, from the perspective of a sophisticated agent within the MFG. Also, notice that the flow $\mu$ in the MFG is deterministic and will not be influenced by changes in the strategy of a single agent. 
	\end{Remark}
	
	\begin{Remark}\label{rm.nonstationary}
	When the context of the MFG is time-consistent, that is, 
	\be\label{eq.deltaexp} 
	f(t,x,\nu,u)=\rho^{t}g(x,\nu,u)
	\ee
	for some function $g:\Delta\times U\mapsto \R$ and $0<\rho<1$, \eqref{eq.def.nonstationary1} implies that 
\bee\label{eq:response.1}
J^{\pi_{[t:\infty)}}\left(x, \mu_t, \mu_{[t+1:\infty)}\right) =\sup_{\pi'\in \Pi}	J^{\pi'_{[t:t+k)} \otimes \pi_{[t+k:\infty)}}\left( x,\mu_t, \mu_{[t+1:\infty)} \right)\quad \forall k\in \N,\; \forall (t,x)\in \T\times [d].
\eee
	As a consequence, 
	\be\label{eq.def.nonstationary1'}  
	J^{\pi}\left(x, \mu_0, \mu_{[1:\infty)}\right) =\sup_{\pi'\in \Pi}	J^{\pi'}\left( x,\mu_0, \mu_{[1:\infty)} \right)\quad \forall x\in [d].
	\ee
	Hence, Definition \ref{def.nonstationary.relax} becomes the conventional equilibrium concept in the (time-consistent) MFG literature under the infinite horizon setting, which has been extensively studied; see \cite{guo2019learning, saldi2018markov}, among many others.
	
\end{Remark}

\begin{Definition}\label{def.stationary.relax}
	Consider a time-independent pair $(\pihat,\muhat)$ with $\pihat=(\pi(x))_{x\in[d]}: [d]\mapsto \Pc(U)$ and $\muhat\in \Pc([d])$. We call $(\pihat,\muhat)$ a stationary equilibrium for the MFG
	if the couple $(\pi, \mu)\in \Pi\times \Lambda$ with $\pi_{[0: \infty)}\equiv\pihat$ and $\mu_{[0:\infty)}\equiv \muhat$, is an equilibrium in Definition \ref{def.nonstationary.relax} for the MFG with initial population distribution $\muhat$.
\end{Definition}

\begin{Remark}\label{rm:equi.stationary}
	Notice that an equilibrium in Definition \ref{def.nonstationary.relax} depends on the initial population $\mu_0=\nu$, which is arbitrarily chosen, while the initial value in a stationary equilibrium in Definition \ref{def.stationary.relax} is a part of the equilibrium. 
	
	If $(\pihat, \muhat)$ constitutes a stationary equilibrium as defined in Definition \ref{def.stationary.relax}, then $\muhat= \muhat\cdot P^{\pihat}(\muhat)$. That is, $\muhat$ is an invariant measure for $P^{\pihat}(\muhat)$, and the population flow is constantly equal to $\muhat$, provided that all agents applying $\pihat$ for each step $t$ and the initial population starts with $\muhat$.
	
	Similar to Remark \ref{rm.nonstationary}, when \eqref{eq.deltaexp} holds, the stationary equilibrium concept in Definition \ref{def.stationary.relax} aligns with the conventional stationary equilibrium in time-consistent MFGs under the infinite horizon setting. This concept has been widely studied in the related literature, e.g., \cite{guo2023general, adlakha2015equilibria, anahtarci2023q, yardim2023policy}. 
\end{Remark}


Now we characterize an MFG equilibrium in Definition \ref{def.nonstationary.relax} as a fixed point of the operator $\Phi^\nu$ below. To begin with, for $\pi\in \Pi$ and initial population distribution $\nu\in \Pc([d])$, we further define an auxiliary function $v^{\pi, \nu}(t+1,x): \T\times[d]\mapsto \R$ as
\be\label{eq:def.auxiv} 
v^{\pi, \nu}(t+1,x):=  \E\left[ \sum_{l\geq 1} f^{\pi_{t+l}\left( s^{\pi,\mu^{\pi,\nu}}_{t+l} \right)}\left(l, s_{t+l}^{\pi,\mu^{\pi,\nu}}, \mu_{t+l}^{\pi,\nu} \right)  \bigg | s^{\pi,\mu^{\pi,\nu}}_{t+1}=x\right].
\ee 
Then for any $\varpi\in \Pc(U)$,
\be\label{eq:J.auxiv}   
J^{\varpi \otimes_1 \pi_{[t+1:\infty)}}\left( x, \mu^{\pi,\nu}_t, \mu^{\pi,\nu}_{[t+1:\infty)} \right)=f^{\varpi}\left( 0,x, \mu^{\pi,\nu}_t \right) +P^{\varpi}\left( x, \mu^{\pi,\nu}_t \right) \cdot v^{\pi,\nu}(t+1, \cdot).
\ee 
Drawing from \eqref{eq:J.auxiv} and Condition B in Definition \ref{def.nonstationary.relax}, we define the mapping $\Phi^\nu(\pi): \Pi\mapsto 2^{\Pi}$ as
\be\label{eq.Phi1} 
\begin{aligned}
	\Phi^\nu(\pi):=
	\bigg\{ & \pi'\in \Pi:\supp\left( \pi'_t(x) \right)\\
	& \subset \argmax_{u\in U}\left\{ f\left(0, x, \mu^{\pi,\nu}_t, u \right)+P\left(x,\mu^{\pi, \nu}_t, \cdot, u \right)\cdot v^{\pi, \nu}(t+1, \cdot)\right\}, \forall (t,x)\in \T\times [d] \bigg\}.
\end{aligned}
\ee 
\begin{Remark}\label{rm.Phinonempty}
{\color{black} For any $\pi\in \Pi$, Assumption \ref{assum.bound} together with the continuity of $f, P$ on variable $u$ guarantees that the function $u\mapsto f(0,x,\mu^{\pi,\nu}_t,u)+P(x,\mu^{\pi,\nu}_t,\cdot, u)\cdot v^{\pi,\nu}(t+1,\cdot)$, defined on the compact set $U$, is bounded and continuous for any $(x,\nu)\in \Delta$ and $t\in \T$, meaning the argmax set in \eqref{eq.Phi1} is non-empty and Borel.} Therefore, $\Phi^\nu(\pi)$ is non-empty for any $\pi\in\Pi$.
\end{Remark}

\begin{Proposition}\label{prop.cha.nonstationary}
	Suppose Assumption \ref{assum.bound} holds. Fix $\nu\in \Pc([d])$. Consider $(\pi,\mu)\in \Pi\times \Lambda$ with $\mu_0=\nu$. Then $(\pi,\mu)\in \Pi\times \Lambda$ is an equilibrium in Definition \ref{def.nonstationary.relax} for the MFG 
	with initial population distribution $\nu$ {\color{black} if and only if $\pi$ is a fixed point of $\Phi^\nu$ and $\mu=\mu^{\pi,\nu}$}.
\end{Proposition}

\begin{proof}
	Suppose $(\pi,\mu)\in \Pi\times \Lambda$ is an equilibrium in Definition \ref{def.nonstationary.relax} for the MFG with initial population $\mu_0=\nu$. 
	Condition A in Definition \ref{def.nonstationary.relax} tells that
	$$
	\mu_t=\mu^{\pi,\nu}_t \quad \forall t\in \T.
	$$
	This together with \eqref{eq:J.auxiv} and Condition B in Definition \ref{def.nonstationary.relax} provides, for all $(t,x)\in \T\times [d]$ that
	$$
	\pi_t(x)\in \argmax_{\varpi\in \Pc(U)} \left\{  \int_U \left[ f \left( t,x,\mu^{\pi,\nu}_t, u \right) + P \left( x,\mu^{\pi,\nu}_t,\cdot ,u \right)\cdot v^{\pi, \nu}(t+1,\cdot) \right] \varpi(du)\right\},
	$$
	which is equivalent to 
	$$
	\supp(\pi_t(x))\subset  \argmax_{u\in U} \left\{ f \left( 0,x,\mu^{\pi,\nu}_t, u \right) + P \left( x,\mu^{\pi,\nu}_t,\cdot ,u \right)\cdot v^{\pi, \nu}(t+1,\cdot) \right\}\;\quad \forall (t,x)\in \T\times [d].
	$$
	Hence, $\pi$ is a fixed point of $\Phi^\nu$.
	
	{\color{black}Conversely, take $(\pi,\mu)\in \Pi\times \Lambda$ such that $\pi$ is a fixed point of $\Phi^\nu$ and $\mu=\mu^{\pi,\nu}$. 
	 Then $\left( \pi,\mu \right)$ automatically satisfies Condition A in Definition \ref{def.nonstationary.relax} with initial population distribution equal to $\nu$.} Moreover, \eqref{eq:J.auxiv} and \eqref{eq.Phi1} together imply, for any $(t,x)\in \T\times [d]$ that
	$$
	J^{\pi_{[t:\infty)}}\left( x,\mu^{\pi,\nu}_t, \mu^{\pi,\nu}_{[t+1:\infty)}\right) =\sup_{\varpi\in \Pc(U)} J^{\varpi\otimes_1 \pi_{[t+1:\infty)}}\left( x,\mu^{\pi,\nu}_t, \mu^{\pi,\nu}_{[t+1:\infty)} \right)\quad \forall (t,x)\in \T\times [d],
	$$
	thus Condition B in Definition \ref{def.nonstationary.relax} holds. Hence, $(\pi, \mu)$ is an equilibrium in Definition \ref{def.nonstationary.relax} for the MFG with initial population distribution $\nu$.
\end{proof}

\begin{Assumption}\label{assum:continuous}
	For each $t\in \T$ and $x,y\in  [d]$, the functions $(\mu, u)\mapsto f(t,x,\mu, u), P(x,\mu,y,u)$ are continuous on $\Pc([d])\times U$.
\end{Assumption}

Since $\Pc([d])\times U$ is compact, Assumption \ref{assum:continuous} implies uniform continuity on $ \Pc([d])\times U$. That is, given any fixed $t, x,y$, for each $\eps>0$, there exits $\delta>0$ such that 
\be\label{eq:fP.uniformcts}  
|f(t,x,\nu,u)-f(t,x,\bar\nu, \bar u)|+|P(x, \nu,y, u)-P(x, \bar\nu, y, \bar u)|<\eps, \; \text{whenever } \left| (\nu, u)- (\bar\nu, \bar u) \right|<\delta
\ee 
Our first result concerns the existence of an MFG equilibrium as defined in Definition \ref{def.nonstationary.relax}.

\begin{Theorem}\label{thm.exist.nonstatoinary}
	Suppose Assumptions \ref{assum.bound}, \ref{assum:continuous} hold. Then for any $\nu \in \Pc([d])$, $\Phi^\nu(\pi)$ has a fixed point. {\color{black}As a consequence,} an equilibrium in Definition \ref{def.nonstationary.relax} for the MFG with initial population $\nu$ exists.
\end{Theorem}

The proof for Theorem \ref{thm.exist.nonstatoinary} is carried out in Section \ref{subsec:proof.nonstationary}.

\subsection{Proof for Theorem \ref{thm.exist.nonstatoinary}}\label{subsec:proof.nonstationary}


We first define notations that will be used during the arguments within this subsection. Recall that $\Pc(U)$ is endowed with the weak topology. Denote by the measure $m(t,x)$ on $\T\times [d]$ as the product of discrete measure $p$ on $\T$ such that $p({t})=2^{-t-1}$ and  the uniform distribution measure on $[d]$. Define $\|h\|_{L^1(\T\times [d])}:= \int_{\T\times [d]}h(t,x)dm(t,x)$ and $\|h\|_{L^\infty(\T\times [d])}:= \sup_{(t,x)\in \T\times [d]}|h(t,x)|$ when the norms are finite. We further endow $L^\infty(\T\times [d])$ with the weak star topology. Then, $h^n\to^* h$ in $L^\infty(\T\times [d])$, that is, $h^n$ converges to $h$ in the weak star topology in $L^\infty(\T\times [d])$, if and only if $\int_{\T\times [d]} h^n(t,x)g(t,x)dm(t,x)\to \int_{\T\times [d]} h(t,x) g(t,x) dm(t,x)$ for all $g\in L^1(\T\times [d])$.

A Carath\'eodory function ($C$-function) $c(t,x,u)$ can be defined as a function that is Borel on $T\times [d]$ such that $u\mapsto c(t,x,u)$ is continuous for each $(t,x)\in \T\times[d]$ and 
\bee  
\int_{T\times [d]} \sup_{u\in U}|c(t,x,u)| dm(t,x)<\infty.
\eee 
Since $U$ is compact, $\Pi$ endowed with the weak topology is a compact convex subset of a locally convex linear topological space. Moreover, for $(\pi^n)_{n\in \N}\subset\Pi$ and $\pi\in \Pi$, $\pi^n\to ^* \pi$ if and only if for any $C$-function $c(t,x,u)$, 
$$
\lim_{n\to\infty} \int_{\T\times [d]}\left( \int_U c(t,x,u) \pi^n_t(x)(du)\right) dm(t,x) \to   \int_{\T\times[d]}\left( \int_U c(t,x,u) \pi_t(x)(du)\right) dm(t,x).
$$

Next, we carry out arguments to show that $\Phi^\nu$ has a closed graph. Then a direct application of the Kakutani-Fan fixed point theorem (see \cite[Corollary 17.55]{guide2006infinite}) establishes the existence of a fixed point for $\Phi^\nu$.


\begin{Lemma}\label{lm.vmu.converge}
	Suppose Assumptions \ref{assum.bound}, \ref{assum:continuous} hold. Let $(\pi^n)_{n\in \N}\subset\Pi$ such that $\pi^n\to ^* \pi^\infty$ in $\Pi$. Take $\nu\in \Pc([d])$, then
	\begin{align}
		&\lim_{n\to\infty}\mu^{\pi^n,\nu}_t= \mu^{\pi^\infty,\nu}_t\quad \forall t\in \T, \label{eq:lmvmu0}\\
		&\lim_{n\to\infty}v^{\pi^n, \nu}(t+1,x) = v^{\pi^\infty, \nu}(t+1,x)\quad  \forall t\in \T, x\in [d]. \label{eq:lmvmu1}
	\end{align}
\end{Lemma}

\begin{proof}
	Take $\nu\in \Pc([d])$. During the proof, we will write $\mu^n$ (resp. $\mu^n_t$) short for $\mu^{\pi^n,\nu}$ for all $n\in \N\cup\{\infty\}$ (resp. $\mu^{\pi^n, \nu}_t$ for all $t\in \T$ and all $n\in \N\cup\{\infty\}$).
	
	
	We first prove \eqref{eq:lmvmu0} for all $t\in \T$ by induction.
	Notice that $\mu^{n}_0= \nu =\mu^{\infty}_0$. Now suppose \eqref{eq:lmvmu0} holds for $t$. 
	Then a direct calculation gives that 
	\be\label{eq.vmuconverge.1} 
	\begin{aligned}
		&\left| \mu^{n}_{t+1}-\mu^{\infty}_{t+1}\right|\\
		&\leq  \left| \left(\mu^{n}_{t}-\mu^{\infty}_{t}  \right)\cdot P^{\pi^n_t} \left( \mu^{n}_{t}\right) \right|
		+ \left|\mu^{\infty}_{t}\cdot P^{\pi^n_t} \left( \mu^{n}_{t}\right)- \mu^{\infty}_{t}\cdot P^{\pi^\infty_t} \left(\mu^{\infty}_{t}\right) \right|\\
		&\leq  \sqrt{d} \left| \mu^{n}_{t}-\mu^{\infty}_{t} \right|
		+ \sqrt{d} \sup_{x\in  [d]}\int_U \left| P\left( x,\mu^{n}_{t},u \right)- P \left( x,\mu^{\infty}_{t},u \right)\right|  \pi^n_t(x)(du) \\
		&\quad + \sqrt{d}\sup_{x\in [d]}\left| \int_U P\left( x,\mu^{\infty}_{t},u \right) \pi^n_t(x)(du)- \int_U P\left( x,\mu^{\infty}_{t},u \right) \pi^\infty_t(x)(du) \right|
	\end{aligned}
	\ee 
	The first term on the third line of \eqref{eq.vmuconverge.1} converges to zero based on the induction assumption. \eqref{eq:fP.uniformcts} together with the finiteness of $[d]$ tells that the second term on the third line of \eqref{eq.vmuconverge.1} also converges to zero.
	As for the third term in last line of \eqref{eq.vmuconverge.1}, for arbitrarily fixed $x_0,y_0\in [d]$, we construct a $C$-function as
	\be\label{eq.vmuconverge.0} 
	c(l,x,u):=d 2^{t+1}P\left( x_0,\mu^{\infty}_{t}, y_0, u \right), \text{ when }x=x_0 \text{ and } l= t;\quad  c(l,x,u):=0, \text{ otherwise }.
	\ee 
	Then $\pi^n\to^*\pi^\infty$ implies  that 
	\be\label{eq.vmuconverge.3}  
	\begin{aligned}
		&\lim_{n\to\infty} P^{\pi^n_t(x_0)} \left( x_0,\mu^{\infty}_{t},y_0 \right)=\lim_{n\to\infty} \int_{\T\times [d]}\left( \int_U c(l,x,u) \pi^n_l(x)(du)\right) dm(l,x)\\
		&= \int_{\T\times[d]}\left( \int_U c(l,x,u) \pi^\infty_l(x)(du)\right) dm(l,x)=P^{\pi^\infty_t(x_0)}\left(x_0,\mu^{\infty}_{t},y_0 \right).
	\end{aligned}
	\ee 
	Since $[d]$ is finite, \eqref{eq.vmuconverge.3} shows the convergence to zero of the last term in \eqref{eq.vmuconverge.1}. As a consequence, \eqref{eq:lmvmu0} holds.
	
	Now we prove \eqref{eq:lmvmu1}. By Assumption \ref{assum.bound}, for any $\eps>0$, there exists $T$ such that 
	\begin{align}
		&\left| v^{\pi^n, \nu}(t+1,x) - v^{\pi^\infty, \nu}(t+1,x) \right| \notag\\
		&\leq \Vast| \E \left[ \sum_{1\leq l\leq T} f^{\pi^n_{t+l} \left( s^{\pi^n, \mu^{n}}_{t+l}\right)} \left( l,s^{\pi^n, \mu^{n}}_{t+l},\mu^{n}_{t+l} \right) \bigg| s^{\pi^n, \mu^{n}}_{t+1} =x \right] \label{eq.vmuconverge.5}\\
		&\qquad -\E \left[\sum_{1\leq l\leq T} f^{\pi^\infty_{t+l}\left(s^{\pi^\infty, \mu^{\infty}}_{t+l}\right)}(l,s^{\pi^\infty, \mu^\infty}_{t+l},\mu^{\infty}_{t+l}) \bigg| s^{\pi^\infty, \mu^{\infty}}_{t+1} =x \right]  \Vast|+\eps. \notag
	\end{align}
	Notice that the expectations in above formula are equal to 
	\be\label{eq.vmuconverge.7} 
	\E\left[ \sum_{1 \leq l\leq T} g^{n}(l,z^{n}_l,\mu^{n}_{t+l}) \right],\quad \E \left[ \sum_{1\leq l\leq T} g^{\infty}(l,z^{\infty}_l,\mu^{\infty}_{t+l}) \right] \text{ respectively},
	\ee 
	where, for each $n\in \N\cup\{\infty\}$, $g^n(l,y,\nu):= f^{\pi^n_{t+l} (y)}(l, y,\nu)$ for $1\leq l\leq T$ and $(y,\nu)\in \Delta$, and $(z^n_l)_{l\geq 1}$ represents the time-inhomogeneous Markov chain with $z^n_1=x$ and transition matrix at step $l$ as
	$$
Q^n_l:= \left(P^{\pi^n_{t+l}(y)} \left( y,\mu^{n}_{t+l},z \right)\right)_{y,z\in[d]}, \text{ for } 1\leq l\leq T.
	$$
	By \eqref{eq:lmvmu0} and $\pi^n\to^*\pi^\infty$, we have $z^n_l$ converges in distribution to $z^\infty_l$ for all $1\leq l\leq T$. For each $y\in[d], 1\leq l\leq T$,
	\begin{align*}
		&\left| g^n\left( l,y,\mu^{n}_{t+l} \right) - g^\infty \left( l,y,\mu^{\infty}_{t+l} \right)\right|\\
		&\leq \left| f^{\pi^n_{t+l}(y)} \left( l,y,\mu^{n}_{t+l} \right) - f^{\pi^n_{t+l}(y)} \left( l,y,\mu^{\infty}_{t+l} \right)\right|
		 +\left| f^{\pi^n_{t+l}(y)} \left( l,y,\mu^{\infty}_{t+l} \right) - f^{\pi^\infty_{t+l}(y)} \left( l,y,\mu^{\infty}_{t+l} \right)\right|.
	\end{align*}
	Then \eqref{eq:fP.uniformcts} tells the first term on the right-hand side tends to zero as $n\to\infty$. By a similar argument as in \eqref{eq.vmuconverge.3} using a proper $C$-function, the second term on the right-hand side above converges to zero as well. Then we reach to
	\begin{align*}
		&\sup_{1\leq l\leq T, x\in [d]}\left| g^n\left( l,x,\mu^{n}_{t+l} \right) - g^\infty \left( l,x,\mu^{\infty}_{t+l} \right)\right|
		\to0, \; \text{as }n\to\infty.
	\end{align*}
	Therefore, the first expectation term converges to the second expectation term in \eqref{eq.vmuconverge.7}. As a consequence, the right-hand side of \eqref{eq.vmuconverge.5} less or equal to $\eps$ as $n\to\infty$. Then \eqref{eq:lmvmu1} follows from the arbitrariness of $\eps$.
\end{proof}

\begin{Lemma}\label{lm.closedgraph}
	Suppose Assumptions \ref{assum.bound}, \ref{assum:continuous} hold. Take $\nu\in \Pc([d])$. Then $\Phi^\nu:\Pi\mapsto 2^{\Pi}$ defined in \eqref{eq.Phi1} has a closed graph.
\end{Lemma}

\begin{proof}
	Take $(\pi^n, \bar\pi^n)_{n\in\N}\subset\Pi$ such that $\bar{\pi}^n\in \Phi^\nu(\pi^n)$ for each $n\in \N$ and 
	$$\pi^n\to^* \pi^\infty, \bar\pi^n\to^* \bar\pi^{\infty} \; \text{as $n\to\infty$}, \quad \text{for some $\pi^\infty, \bar\pi^\infty\in\Pi$.}$$
	Notice that $\Phi^\nu(\pi^n)$ is non-empty for each $n\in \N\cup\{ \infty\}$ as stated in Remark \ref{rm.Phinonempty}.
	 We now prove $\bar\pi^\infty\in \Phi^\nu(\pi^\infty)$. And it suffices to prove, for each $(t,x)\in \T\times [d]$ that
	\be\label{eq.prop.closedgraph1} 
	\supp(\bar\pi^\infty (t,x))\subset \argmax_{u\in U}\left\{ f \left( 0, x, \mu^{\pi^\infty,\nu}_t, u \right)+P\left( x,\mu^{\pi^\infty,\nu}_t, \cdot, u \right)\cdot v^{\pi^\infty, \nu}(t+1,\cdot) \right\}.
	\ee 
	
	Fix an arbitrary $(t_0, x_0)\in \T\times [d]$. 
{\color{black}	Denote
	\begin{align*}
	\begin{cases}
	E^n:= &  \argmax_{u\in U}\left\{ f\left(0, x_0, \mu^{\pi^n, \nu}_{t_0}, u \right)+P\left( x,\mu^{\pi^n, \nu}_{t_0}, \cdot, u \right) \cdot v^{\pi^n,\nu}(t_0+1,\cdot)\right\}\\
    A^n := & \max_{u\in U}  \left\{  f\left(0, x_0, \mu^{\pi^n, \nu}_{t_0}, u \right)+P\left( x,\mu^{\pi^n, \nu}_{t_0}, \cdot, u \right) \cdot v^{\pi^n,\nu}(t_0+1,\cdot) \right\}
	\end{cases} \quad \forall n\in \N\cup \{\infty\}.
	\end{align*}
Then \eqref{eq.prop.closedgraph1} follows by proving
$ 
\left(\bar\pi^\infty_{t_0}(x_0) \right)(B_r(u_0))=0$ for all $u_0\in U$ and $r>0$ such that 
\be\label{eq.prop.closedgraph3}  
\overline{B_r(u_0)}\cap E^\infty=\emptyset \text{ and } B_r(u_0)\subset U.
\ee
By Assumption \ref{assum:continuous} and the compactness of $U$, $E^n$ are closed for each $n\in \N\cup\{\infty\}$. By Lemma \ref{lm.vmu.converge}, $\mu^{\pi^n, \nu}_{t_0}\to \mu^{\pi^\infty,\nu}_{t_0}$ in $\Pc([d])$ and $v^{\pi^n,\nu}(t_0+1, \cdot )\to v^{\pi^\infty, \nu}(t_0+1,\cdot)$ in $\R^d$.  This together with the uniform continuity of $f, P$ on the compact set $\Pc([d])\times U$ in \eqref{eq:fP.uniformcts} shows that 	\footnote{ {\color{black}Indeed, if $\limsup_{n\to\infty} A^n >A^\infty$, then there exists a subsequence $u_{n_j}\to \hat u$ for some $\hat u\in U$ such that 
	\begin{align*}
		&\lim_{j\to\infty} \left[f\left(0, x_0, \mu^{\pi^{n_j}, \nu}_{t_0}, u_{n_j} \right)+P\left( x,\mu^{\pi^{n_j}, \nu}_{t_0}, \cdot, u_{n_j} \right) \cdot v^{\pi^{n_j},\nu}(t_0+1,\cdot)\right] \\
		&= \limsup_{n\to\infty} A^n >A^\infty\geq  f\left(0, x_0, \mu^{\pi^{\infty}, \nu}_{t_0}, \hat u \right)+P\left( x,\mu^{\pi^{\infty}, \nu}_{t_0}, \cdot, \hat u \right) \cdot v^{\pi^{\infty},\nu}(t_0+1,\cdot),
	\end{align*} 
	a contradiction. Thus, $\limsup_{n\to\infty} A^n \leq A^\infty$. Similarly, $\liminf_{n\to\infty} A^n \geq A^\infty$. }}
\be\label{eq:prop.closedgraph.new1} 
	\lim_{n\to\infty} A^n = A^\infty. 
\ee 
Pick an arbitrary ball $B_r(u_0)$ that satisfies \eqref{eq.prop.closedgraph3}, let {\color{black}$D$} be the distance between $\overline{B_r(u_0)}$ and $E^\infty$, then ${\color{black}D}>0$. So $\overline{B_{r+{\color{black}D}/2}(u_0)}\cap  E^\infty=\emptyset$, and we can conclude that an integer $N$ exists such that 
	\be\label{eq.prop.closedgraph5} 
	E^n\cap \overline{B_{r+{\color{black}D}/2}(u_0)}=\emptyset, \quad \forall N\leq n<\infty.
	\ee
Otherwise, there exists a subsequence $(u^{n_k})_{k\in \N}$ such that $u^{n_k}\in E^{n_k}\cap\overline{B_{r+{\color{black}D}/2}(u_0)} $ and $u_{n_k}\to \bar u$ for some $\bar u\in \overline{B_{r+{\color{black}D}/2}(u_0)}$. Then, again by the uniform continuity of $f, P$ in \eqref{eq:fP.uniformcts},
\begin{align*}
	&f\left(0, x_0, \mu^{\pi^{\infty}, \nu}_{t_0}, \bar u \right)+P\left( x,\mu^{\pi^{\infty}, \nu}_{t_0}, \cdot, \bar u \right) \cdot v^{\pi^{\infty},\nu}(t_0+1,\cdot)\\
	=&\lim_{k\to\infty} \left[f\left(0, x_0, \mu^{\pi^{n_k}, \nu}_{t_0}, u_{n_k} \right)+P\left( x,\mu^{\pi^{n_k}, \nu}_{t_0}, \cdot, u_{n_k} \right) \cdot v^{\pi^{n_k},\nu}(t_0+1,\cdot)\right] \\
	=& \lim_{k\to\infty}A^{n_k}=A^\infty,
\end{align*}
	where the  last equality follows from \eqref{eq:prop.closedgraph.new1}. Thus, $\bar u\in E^\infty$, a contradiction.} Then consider a non-negative $C$-function $c(t,x,u): \T\times [d]\times U\mapsto \R$ such that $c(t,x,u)\equiv 0$ on $U$ if $(t,x)\neq (t_0,x_0)$, and 
	$$
	c(t_0, x_0,u)=1\; \text{for }u\in B_r(u_0),\quad   c(t_0,x_0,u)\equiv 0 \; \text{for } u\notin B_{r+{\color{black}D}/2}(u_0).
	$$
	By $\bar\pi^n\to^* \bar\pi^\infty$, 
	\bee
	\begin{aligned}
		&\int_U c(t_0,x_0, u)\bar\pi^\infty_{t_0}(x_0)(du)=d 2^{t_0+1}\int_{\T\times [d]}\left( \int_U c(t,x,u) \bar\pi^\infty_t(x)(du) \right) dm(t,x)\\
		&= \lim_{n\to\infty} d 2^{t_0+1}\int_{\T\times [d]}\left( \int_U c(t,x,u)\bar\pi^n_t(x)(du) \right) dm(t,x)= \lim_{n\to\infty}   \int_U c(t_0,x_0,u)\bar\pi^n_{t_0}(x_0)(du) =0,
	\end{aligned}
	\eee 
	where the last equality follows from \eqref{eq.prop.closedgraph5} and the fact that $\supp \left(\bar\pi^n_{t_0}(x_0) \right)\subset E^n$ for all $n\in \N$. Hence, $\left(\bar\pi^\infty_{t_0}(x_0) \right)(B_r(u_0))=0$. By the arbitrariness of $u_0$, \eqref{eq.prop.closedgraph1} is verified for the arbitrarily picked $(t_0, x_0)\in \T\times [d]$ and the proof is complete.
\end{proof}

Now we are ready to finish the proof for Theorem \ref{thm.exist.nonstatoinary}. Since $\Pi$ is a compact convex subset of a locally convex linear topological space. For an arbitrary $\nu\in \Pc([d])$, Lemma \ref{lm.closedgraph} together with Remark \ref{rm.Phinonempty} shows that the mapping $\Phi^\nu(\pi):\Pi\mapsto 2^{\Pi}$ is nonempty for all $\pi\in\Pi$ and has a closed graph. Then by Kakutani-Fan fixed point theorem (see \cite[Corollary 17.55]{guide2006infinite}), $\Phi^\nu$ has a fixed point. {\color{black}Take $\pi\in \Pi$ as a fixed point of $\Phi^\nu$, then by Proposition \ref{prop.cha.nonstationary}, $(\pi, \mu^{\pi,\nu})$ is an equilibrium in Definition \ref{def.nonstationary.relax} for the MFG with initial population distribution $\nu$.}

\section{Precommitment equilibrium in the $N$-agent game }\label{sec:Nagent}
\subsection{$N$-agent games with time-inconsistency}\label{subsec:Ngame.notations}
Now we set up the finite-agent counterpart. Let $\{e_1,\cdots, e_d \}$ be the standard basis of $\R^d$. By taking an integer $N\geq 2$ as the number of agents, we define
\be\label{eq.def.DeltaN} 
\Delta_N:=\left\{(x,\nu)\in \Delta: N\nu\in \N^d\text{ and } N\nu(x)>0 \right\}.
\ee 
And the vector $\left( s^{N,1}_t,...,s^{N,N}_t \right)$ represents the states of all $N$ agents at time $t$.

Let $\bmpi^N:= (\pi^{N,1},...,\pi^{N,N})$ be an $N$-tuple where the $i$-th agent carries out the relaxed control $\pi^{N,i}\in \Pi$ for each $1\leq i\leq N$.
For $t\in \T$, denote by $\mu^{N}_t:= \frac{1}{N}\sum_{1\leq j\leq N} e_{s^{N,j}_t}$ the empirical distribution of the $N$ agents under $\bmpi^N$ at time $t$. 
For agent $i$, $s^{N,i}_{t+1}$ is determined by
$
P\left(s^{N,i}_t, \mu^{N}_t, \alpha^{N,i}_t \right)
$, where $\alpha^{N,i}_t, 1\leq i\leq N$, are mutually independent random variables following distributions $\pi^{N,i}_t \left( s_t^{N,i} \right), 1\leq i\leq N$, respectively. Then the law of the flow $\left(\mu^{N}_t\right)_{t\geq 1}$ is determined by $\bmpi^N$ and the initial population empirical distribution $\mu^{N}_0$. {\color{black}Take $(x,\nu)\in \Delta_N$, we introduce the following total pay-off for agent $i$ starting at state $x$ with 
$\mu^{N}_0=\nu$ as 
\be\label{eq.def.JNagent}  
\begin{aligned}
	J^{N,i, \bmpi^N}(x, \nu):= \E\left[ \sum_{t=0}^\infty  f^{\pi^{N,i}_t\left( s^{N,1}_t \right) }\left(t, s^{N,i}_t, \mu^{N}_t \right) \bigg | s^{N,i}_0=x, \mu^{N}_0=\nu \right]
\end{aligned}
\ee}
Since the dynamic of the flow $\mu^{N}_{[0:\infty)}$ is governed by $\bmpi^N$ and the initial value $\nu$, the value function $J^{N,1, \bmpi^N}(x, \nu)$ above is also determined by the $N$-tuple policy $\bm\pi^N$ and the initial values $(x,\nu)$. Hence, there is no need to explicitly indicate the tail $\mu^{N}_{[1:\infty)}$ in the notation $J^{N,1, \bmpi^N}$. {\color{black}When all $N$ agents adapts the same policy $\pi\in \Pi$, i.e., $\bmpi^N=(\pi,\cdots, \pi)$, we always take  agent 1 as the representative agent in an $N$-agent game.}

Given an $N$-tuple $\bmpi^N=\left( \pi^{N,1},\cdots, \pi^{N,N} \right)$ such that $\pi^{N,i}\in \Pi$ for each $1\leq i\leq N$, we denote by $(\bmpi^N_{-i}, \pi')$ the new $N$-tuple by replacing the $i$-th component $\pi^{N,i}$ with another policy $\pi'\in \Pi$. Given two $N$-tuples $\bmpi^N=(\pi^{N,1},...,\pi^{N,N})$ and ${\bar \bmpi}^N=({\bar \pi}^{N,1},...,{\bar \pi}^{N,N})$ and two non-negative integers $t_0<t_1<\infty$, we further denote by ${\bar \bmpi}^N_{[t_0: t_1)}\otimes \bmpi^N_{[t_1: \infty)}$ as the new $N$-tuple such that agent $i$ applies ${\bar \pi}^{N,i}$ during steps $t_0$ to $t_1-1$ then applies $\pi^{N,i}$ from $t_1$ onward, for all $1\leq i\leq N$. 
{\color{black}
\begin{Remark}\label{rm:initial.iid}
One common condition for a $N$-agent game in the MFG literature states the following:
\be\label{eq:precommitted.iid0}
\text{	$\left( s^{N,i}_0 \right)_{1\leq i\leq N}$ are independently simulated by a common distribution $\nu\in \Pc([d])$.}
\ee
For the rest part of this paper, whenever a formula within an $N$-agent game involves a random variable rooted from the above condition  \eqref{eq:precommitted.iid0},
we shall use the subscript notation $\E_\nu$ when taking an expectation, where $\nu$ is the common distribution generating all initial values for the $N$ agents.
\end{Remark}}

Now we make the first attempt to address Question \eqref{sentence:I}. The following proposition answers what happens when employing an MFG equilibrium as defined in Definition \ref{def.nonstationary.relax} within an $N$-agent game.

\begin{Proposition}\label{prop.precommittedN}
	{\color{black}Suppose Assumptions \ref{assum.bound}, \ref{assum:continuous} hold.} Let $(\pi, \mu)\in \Pi\times \Lambda$ be an equilibrium as defined in Definition \ref{def.stationary.relax} for the MFG with initial population distribution $\nu\in \Pc([d])$. For any $\eps>0$, there exists $N_0$ such that for all $N\geq N_0$, suppose $(s^{N,i}_0)_{1\leq i\leq N}$ are independently simulated by $\nu$, then the replicating $N$-tuple $\bm{\pi}^N=(\pi,...,\pi)$ satisfies that 
	{\color{black}\be\label{eq.prop.precommittedN} 
	\E_{\nu}\left[J^{N,i,\bmpi^N}\left(s^{N,i}_0, \mu^{N}_0\right) \right] \geq  \sup_{\pi' \in \Pi}\E_{\nu}\left[J^{N,i,\left( \bmpi^N_{-i}, \pi' \right)_{[0:1)}\otimes\bm\pi^N_{[1:\infty)}} \left(s^{N,i}_0, \mu^{N}_0 \right) \right] -\eps,\quad \text{for }  1\leq i\leq N.
	\ee}
\end{Proposition}

\begin{proof}
	{\color{black}In the formulation of right-hand side of \eqref{eq.prop.precommittedN}, the $N$-tuple that agent $i$ takes deviation from is replicating.
	Thus, we simply take agent 1 as the representative agent and prove the result only for $i=1$, which is
	\be\label{eq.prop.precommittedN1} 
{\color{black}
	\E_{\nu}\left[J^{N,1,\bmpi^N}\left(s^{N,1}_0, \mu^{N}_0\right) \right] \geq  \sup_{\pi' \in \Pi}\E_{\nu}\left[J^{N,1,\left( \bmpi^N_{-1}, \pi' \right)_{[0:1)}\otimes\bm\pi^N_{[1:\infty)}} \left(s^{N,1}_0, \mu^{N}_0 \right) \right] -\eps,\; \text{for $N$ big enough.}}
\ee	
}
	
	Fix an arbitrary $\nu\in \Pc([d])$. Let $(\pi,\mu)\in \Pi\times \Lambda$ be an equilibrium in Definition \ref{def.nonstationary.relax} for the MFG with initial population distribution $\nu$. For an arbitrary $\pi'\in \Pi$, let $\mu^N_{[0:\infty)}$ be the empirical population flow under the $N$-tuple $\left( \bmpi^N_{-1}, \pi' \right)_{[0:1)}\otimes\bm\pi^N_{[1:\infty)}$ with $ \bmpi^N=(\pi,\cdots, \pi)$ and $(s^{N,i}_t)_{t\in \T}$ be the dynamic for agent $i$.
	
We first prove by induction, for each $t\in \T$ that 
	\be\label{eq.precommit.Nmu0}  
	\sup_{\pi'\in \Pi} \E_\nu \left[\left|\mu^N_t-\mu^{\pi,\nu}_t\right| \right] \to 0, \text{ as $N\to\infty$}.
	\ee 
	When $t=0$, \eqref{eq.precommit.Nmu0} holds by the Strong Law of Large Numbers and the condition that $\left( s^{N,i}_0 \right)_{1\leq i\leq N}$ are i.i.d. and simulated by $\nu$. Suppose \eqref{eq.precommit.Nmu0} holds for $t$, now we prove the case for $t+1$. For $i=1,...,N$, let $\alpha^{N,i}_t$ be mutually independent random action values that are simulated by $\pi_t\left(s^{N,i}_t \right)$ respectively.
	A direct calculation gives that 
	\begin{align}
		\E_\nu \left[ \left| \mu^{N}_{t+1}- \mu^{\pi,\nu}_{t+1} \right|  \right]
		&\leq \E_\nu \left[ \left| \frac{1}{N}\sum_{i=1}^N  e_{s^{N,i}_t}\cdot P\left(\mu^{N}_t,\alpha^{N,i}_t \right)-  \frac{1}{N}\sum_{i=1}^N  e_{s^{N,i}_t}\cdot P\left(\mu^{\pi, \nu}_t,\alpha^{N,i}_t \right)\right|  \right]\notag\\
		&\qquad +\E_\nu \left[\left|  \frac{1}{N}\sum_{i=1}^N  e_{s^{N,i}_t}\cdot P\left(\mu^{\pi,\nu}_t,\alpha^{N,i}_t \right) - \mu^{N}_t\cdot P^\pi(\mu^{\pi,\nu}_t) \right| \right] \notag \\
		&\qquad +\E_\nu \left[ \left|  \mu^{N}_t\cdot P^\pi(\mu^{\pi,\nu}_t) - \mu^{\pi,\nu}_t\cdot P^\pi(\mu^{\pi,\nu}_t) \right| \right], \quad \text{when } t\geq 1. \label{eq:prop.precommit.Nmu1}
	\end{align}
	The convergence of the first and third terms of the right-hand side of \eqref{eq:prop.precommit.Nmu1} follows from the induction assumption and the uniform continuity of $P(x,\mu,u)$ on $(\mu,u)\in \Pc([d])\times U$ (see \eqref{eq:fP.uniformcts}). Meanwhile, \cite[Lemma A.2]{budhiraja2015long} gives the convergence of the second term. Since the contribution of $\alpha^{N,1}_0$ to the term $\frac{1}{N}\sum_{i=1}^N  e_{s^{N,i}_0}\cdot P\left(\mu^{N}_0,\alpha^{N,i}_0 \right)$ in \eqref{eq:prop.precommit.Nmu1} vanishes as $N\to\infty$, the above argument is also valid for $t=0$, even though $\alpha^{N,1}_0$ is simulated by $\pi'_0$ instead of $\pi_0$.
	In sum, \eqref{eq.precommit.Nmu0} holds. 
	
{\color{black} Next, for any $\pi'\in\Pi$, consider the following auxiliary Markov chain $(z^{\pi'\otimes_1\pi}_t)_{0\leq t\leq T}$ with $z^{\pi'\otimes_1\pi}_0\sim \nu$ that is independent of $(s^{N,i}_t)_{t\in \T}$. The dynamic of $(z^{\pi'\otimes_1\pi}_t)_{0\leq t\leq T}$ is guided by the deterministic transition matrices $$
\begin{aligned}
&Q_0=(Q_0(x,y))_{x,y\in[d]}:=\left( P^{\pi'_0(x)}(x,\nu,y) \right)_{x,y\in[d]},\\
&Q_t=(Q_t(x,y))_{x,y\in[d]}:= \left( P^{\pi_{t}(x)}\left(x, \mu^{\pi,\nu}_{t}, y\right)\right)_{x,y\in[d]} \; \text{for } 1\leq t\leq T-1.
\end{aligned}
$$
That is, the dynamic of $(z^{\pi'\otimes_1\pi}_t)_{0\leq t\leq T}$ is first guided by $\pi_0$ at step zero then by $\pi_t$ for each step $1\leq t\leq T$.
By \eqref{eq.precommit.Nmu0} and the uniform continuity of $P$ in \eqref{eq:fP.uniformcts}, 
\be\label{eq:prop.sN0new}   
\begin{aligned}
&\lim_{N\to\infty}\sup_{\pi'\in \Pi, x, y\in[d]}\E_\nu\left[ \left|P^{\pi'_0}\left(x, \mu^{N}_0, y \right)-Q_0(x, y) \right| \right] = 0,\\
&\lim_{N\to\infty}\sup_{\pi'\in \Pi, x, y\in[d]}\E_\nu\left[ \left|P^{\pi_t}\left(x, \mu^{N}_t, y \right)-Q_t(x, y) \right| \right] = 0 \quad  \text{for } 1\leq t\leq T-1.
\end{aligned}
\ee 
Now we prove by induction, for each $0\leq t\leq T$ that
\be\label{eq:prop.sNnew}  
\text{$\left(\mu^N_t, e_{s^{N,1}_t} \right) \to \left(\mu^{\pi, \nu}_t, e_{z^{\pi'\otimes_1\pi}_t}\right)$ in distribution uniformly over $\pi'\in \Pi$.}
\ee
By \eqref{eq.precommit.Nmu0}, $\mu^N_0\to \nu$ in probability uniformly over $\pi'$, where $\nu$ is deterministic, and both $z^{\pi'\otimes_1\pi}_0$ and $s^{N,1}_0$ have distribution equal to $\nu$. Then \eqref{eq:prop.sNnew} holds for $t=0$. Suppose \eqref{eq:prop.sNnew} holds for $t$ and we prove the case for $t+1$. Due to \eqref{eq.precommit.Nmu0} and the fact that $\mu^{\pi,\nu}_{t+1}$ is deterministic, to prove the joint distribution of $(\mu^N_{t+1}, e_{s^{N,1}_{t+1}})$ converges to the joint distribution of $(\mu^{\pi, \nu}_{t+1}, e_{z^{\pi'\otimes_1\pi}_{t+1}})$ uniformly over $\pi'\in \Pi$, it suffices to prove $e_{s^{N,1}_{t+1}}$ converges to $e_{z^{\pi'\otimes_1\pi}_{t+1}}$in distribution uniformly over $\pi'\in \Pi$. Notice that
\be\label{eq:prop.sNnew0}    
e_{s^{N,1}_{t+1}} = e_{s^{N,1}_{t}}\cdot P^{\pi_t}\left(\mu^{N,1}_t \right) \text{ in distribution}.
\ee 
Then by \eqref{eq:prop.sN0new} \footnote{{\color{black} When $t=0$, the first equality in \eqref{eq:prop.sN0new} should be applied and $P^{\pi_t}$ should be replaced by $P^{\pi'_0}$ in \eqref{eq:prop.sNnew0} and \eqref{eq:prop.sNnew1}. For $t\geq 1$, the second equality in \eqref{eq:prop.sN0new} should be used.}}, the finiteness of $[d]$, the uniform continuity of $P$ (see \eqref{eq:fP.uniformcts}) and the induction hypothesis, we have 
\be\label{eq:prop.sNnew1}    
e_{s^{N,1}_{t}}\cdot P^{\pi_t}(\mu^{N,1}_t) \to e_{z^{\pi'\otimes_1\pi}_t}\cdot P^{\pi_t}(\mu^{\pi,\nu}_t) = e_{z^{\pi'\otimes_1\pi}_{t+1}}\text{ in distribution uniformly over $\pi'\in \Pi$},
\ee 
and the induction proof is complete.

Next, we prove that
	\be\label{eq.precommit.Nv}  
	\begin{aligned}
		\lim_{N\to\infty} \sup_{\pi'\in \Pi}\Bigg| & \E_\nu\left[ J^{N,1,\left( \bmpi^N_{-1}, \pi' \right)_{[0:1)}\otimes\bm\pi^N_{[1:\infty)}} \left(s^{N,1}_0, \mu^{N}_0 \right)  \right]\\
		& - \E_\nu\left[ J^{\pi'_{[0:1)}\otimes \pi_{[1:\infty)}} \left(s^{N,1}_0,\nu, \mu^{\pi,\nu}_{[1:\infty)} \right)  \right]  \Bigg| =0 .
	\end{aligned}
	\ee 		
	Notice that for an arbitrary $\pi'\in\Pi$,
	\be\label{eq.precommit.Nv0} 
	\begin{aligned}
		&\E_\nu\left[ J^{N,1,\left( \bmpi^N_{-1}, \pi' \right)_{[0:1)}\otimes\bm\pi^N_{[1:\infty)}} \left(s^{N,1}_0, \mu^{N}_0 \right)  \right]\\
		&=\E_\nu\left[ f^{\pi'_{0} \left(s^{N,1}_{0}\right)}\left(0, s^{N,1}_{0}, \mu^N_{0}\right) +\sum_{l=1}^{\infty} f^{\pi_{l} \left(s^{N,1}_{l}\right)}\left(l, s^{N,1}_{l}, \mu^N_{l}\right)  \right],\\
		&\E_\nu\left[ J^{\pi'_{[0:1)}\otimes \pi_{[1:\infty)}} \left(s^{N,1}_0,\nu, \mu^{\pi,\nu}_{[1:\infty)} \right)  \right] \\
		&= \E\left[f^{\pi'_{0} \left(z^{\pi'\otimes_1\pi}_{0}\right)}\left(0, z^{\pi'\otimes_1\pi}_{0}, \nu \right)+ \sum_{l=1}^{\infty} f^{\pi_{l} \left(z^{\pi'\otimes_1\pi}_{l}\right)}\left(l, z^{\pi'\otimes_1\pi}_{l}, \mu^{\pi,\nu}_{l}\right)\right].
	\end{aligned}
\ee
	For any $\eps>0$, from Assumption \ref{assum.bound} we conclude that the tails for $l>T$ in the two sums of \eqref{eq.precommit.Nv0} are less than $\eps$ for some $T\in \T$. 
	For each $1\leq l\leq T$, by \eqref{eq:prop.sNnew}, \eqref{eq:fP.uniformcts} and the finiteness of $[d]$,
	\begin{align}
		&\sup_{\pi'\in \Pi} \left| \E_\nu\left[f^{\pi_{l} \left(s^{N,1}_{l}\right)}\left(l, s^{N,1}_{l}, \mu^N_{l}\right) \right] - \E\left[ f^{\pi_{l} \left( z^{\pi'\otimes_1\pi}_{l}\right)}\left(l, z^{\pi'\otimes_1\pi}_{l}, \mu^{\pi,\nu}_{l}\right) \right] \right| \notag\label{eq.precommit.Nv'} 
		\to 0\quad \text{as $N\to\infty$}, \notag
	\end{align}
and
	\bee  
\sup_{\pi'\in \Pi} \left| \E_\nu\left[f^{\pi'_0 \left(s^{N,1}_0 \right)}\left(0, s^{N,1}_0 , \mu^N_0 \right) \right] - \E\left[ f^{\pi'_0 \left( z^{\pi'\otimes_1\pi}_0 \right)}\left(0, z^{\pi'\otimes_1\pi}_0 , \mu^{\pi,\nu}_0 \right) \right] \right|	\to 0\quad \text{as $N\to\infty$}.
	\eee
Therefore, \eqref{eq.precommit.Nv} holds by the arbitrariness of $\eps$.
	
Now we finish the proof for the desired result. For any $\eps>0$, by applying \eqref{eq.precommit.Nv}, there exits $N_0$ such that, for all $N\geq N_0$,
\bee
\begin{cases}
 \sup_{\pi'\in \Pi}  \left| \E_\nu\left[ J^{N,1,\left( \bmpi^N_{-1}, \pi' \right)_{[0:1)}\otimes\bm\pi^N_{[1:\infty)}} \left(s^{N,1}_0, \mu^{N}_0 \right)  \right]
- \E_\nu\left[ J^{\pi'_{[0:1)}\otimes \pi_{[1:\infty)}} \left(s^{N,1}_0,\nu, \mu^{\pi,\nu}_{[1:\infty)} \right)  \right]  \right| <\eps,\\
\left| \E_\nu\left[ J^{N,1,\bm\pi^N} \left(s^{N,1}_0, \mu^{N}_0 \right)  \right]
- \E_\nu\left[ J^{ \pi} \left(s^{N,1}_0,\nu, \mu^{\pi,\nu}_{[1:\infty)} \right)  \right]  \right| <\eps.
\end{cases}
\eee
Then, combining the above inequalities and \eqref{eq.def.nonstationary1} with $t=0$ yields 
\begin{align*}
&\E_{\nu}\left[J^{N,1,\bmpi^N}\left(s^{N,1}_0, \mu^{N}_0\right) \right] - \sup_{\pi' \in \Pi}\E_{\nu}\left[J^{N,1,\left( \bmpi^N_{-1}, \pi' \right)_{[0:1)}\otimes\bm\pi^N_{[1:\infty)}} \left(s^{N,1}_0, \mu^{N}_0 \right) \right]\\
&\geq -\left| \E_{\nu}\left[J^{N,1,\bm\pi^N} \left(s^{N,1}_0, \mu^{N}_0 \right) - J^{\pi} \left(s^{N,1}_0,\nu, \mu^{\pi,\nu}_{[1:\infty)} \right)  \right] \right|\\
&\quad - \sup_{\pi' \in \Pi}\left| \E_{\nu}\left[  J^{\pi} \left(s^{N,1}_0,\nu, \mu^{\pi,\nu}_{[1:\infty)} \right)  - J^{\pi'_{[0:1)}\otimes \pi_{[1:\infty)}} \left(s^{N,1}_0,\nu, \mu^{\pi,\nu}_{[1:\infty)} \right)  \right]\right|\\
&\quad - \sup_{\pi' \in \Pi}\left| \E_{\nu}\left[ J^{\pi'_{[0:1)}\otimes \pi_{[1:\infty)}} \left(s^{N,1}_0,\nu, \mu^{\pi,\nu}_{[1:\infty)} \right) - J^{N,1,\left( \bmpi^N_{-1}, \pi' \right)_{[0:1)}\otimes\bm\pi^N_{[1:\infty)}} \left(s^{N,1}_0, \mu^{N}_0 \right)  \right] \right|\\
&\geq -\eps-0-\eps\geq -2\eps,
\end{align*}
and \eqref{eq.prop.precommittedN1} follows.}
\end{proof}

\begin{Remark}\label{rm.timeconNstationary}
	When the context is time-consistent, i.e., \eqref{eq.deltaexp} holds, by assuming the same conditions in Proposition \ref{prop.precommittedN}, a similar result holds by replacing \eqref{eq.prop.precommittedN} with
	\be\label{eq.prop.precommittedN'} 
	{\color{black}\E_{\nu}\left[J^{N,i,\bmpi^N}\left(s^{N,i}_0, \mu^{N}_0\right) \right] \geq  \sup_{\pi'\in \Pi}\E_{\nu}\left[J^{N,i,\left( \bmpi^N_{-i}, \pi'\right)} \left(s^{N,i}_0, \mu^{N}_0 \right) \right] -\eps,\quad \text{for }  1\leq i\leq N.}
	\ee 
	That is, Proposition \ref{prop.precommittedN} becomes the classical $\eps$-optimality result in the MFG literature when applying an equilibrium in Definition \ref{def.nonstationary.relax} within an $N$-agent game; see \cite{saldi2018markov} for similar results. Similar findings are also provided for continuous time finite-horizon setting, e.g., \cite{carmona2013probabilistic}.
\end{Remark}

\begin{Remark}\label{rm.nonNstationary}
	 A similar result as in Proposition \ref{prop.precommittedN} for a stationary MFG equilibrium in Definition \ref{def.stationary.relax} can be stated as follows: 
	Let $(\pihat, \muhat)\in (\Pc(U))^d\times \Pc([d])$ be a stationary equilibrium as in Definition \ref{def.stationary.relax}. For each $N\geq 2$, consider the replicating $N$-tuple $\bm{\pi}^N=(\pi^{N,1},...,\pi^{N,N})$ such that $\pi^{N,i}_{[0:\infty)}\equiv \pihat$ for all $1\leq i\leq N$. Then for any $\eps>0$, there exists $N_0$ such that for all $N\geq N_0$, if $( s^{N,i}_0 )_{1\leq i\leq N}$ are independently simulated by $\muhat$, then
	\be\label{eq.prop.precommittedNnonstationary} 
	{\color{black}\E_{\muhat}\left[J^{N,i,\bmpi^N}\left(s^{N,i}_0, \mu^{N}_0\right) \right] \geq  \sup_{\pi' \in \Pi}\E_{\muhat}\left[J^{N,i,\left( \bmpi^N_{-i}, \pi' \right)_{[0:1)}\otimes\bm\pi^N_{[1:\infty)}} \left(s^{N,i}_0, \mu^{N}_0 \right) \right] -\eps,\;\; \text{for } 1\leq i\leq N}.
	\ee 

	When the context is time-consistent, i.e., when \eqref{eq.deltaexp} is assumed, a similar result holds for a stationary equilibrium by rewriting 
	\eqref{eq.prop.precommittedNnonstationary} as 
	$$
	{\color{black}
	\E_{\muhat}\left[J^{N,i,\bmpi^N}\left(s^{N,i}_0, \mu^{N}_0\right) \right] \geq  \sup_{\pi'\in \Pi}\E_{\muhat}\left[J^{N,i,\left( \bmpi^N_{-i}, \pi'\right)} \left(s^{N,i}_0, \mu^{N}_0 \right) \right] -\eps,\quad\text{for } 1\leq i\leq N,}
    $$
 to mention a few references, see \cite{adlakha2015equilibria} and \cite{anahtarci2023q}. 
\end{Remark}

\subsection{Equilibria in Definition \ref{def.nonstationary.relax} lead to precommitment approximate Nash equilibria for N-agent games}\label{subsec:explain}


As mentioned in Remark \ref{rm:initial.iid}, the results in Proposition \ref{prop.precommittedN} and Remark \ref{rm.timeconNstationary} (resp. the results in Remark \ref{rm.nonNstationary}) require an important condition:
\be\label{eq:precommitted.iid}
\text{	$\left( s^{N,i}_0 \right)_{1\leq i\leq N}$ are independently simulated by $\nu$ (resp. by $\muhat$),}
\ee
which plays a key role to ensure that $\mu^N_t$ converges to $\mu^{\pi,\nu}_t$ (resp. $\mu^{\pi,\hat \mu}_t$). 
{\color{black}Given a $N$ tuple $\bmpi^N=(\pi^{N,1},\cdots, \pi^{N,N})$ and $\varpi\in \Pc(U)$, denote by $((\bmpi^N_{t})_{-i}, \varpi)\otimes_1 \bm\pi^N_{[t+1:\infty)}$ the pasting $N$-tuple where agent $i$ applies $\varpi$ at step $t$ and then switches to $\pi^{N,i}$ afterward, while any other agent $j$ following strategy $\pi^{N,j}$ for $j\neq i$. Then we can further extend \eqref{eq.prop.precommittedN} to the following.}
\be\label{eq.prop.precommittedNt} 
{\color{black}\begin{aligned}
&\E_\nu\left[ J^{N,i,\bmpi^N_{[t:\infty)}}\left(s^{N,i}_t, \mu^N_t\right)\right] \\
&\geq \sup_{\varpi \in \Pc(U)}\E_\nu\left[ J^{N,i,\left(\left(\bmpi^N_{t}\right)_{-i}, \varpi\right)\otimes_1 \bm\pi^N_{[t+1:\infty)}} \left(s^{N,i}_t, \mu^N_t \right)\right] -\eps \;\;\text{for } 1\leq i\leq N,\quad \forall t\in \T.
\end{aligned}}
\ee
Here $\E_\nu$ still means the expectation is taken at time 0 before that initial values $( s^{N,i}_0 )_{1\leq i\leq N}$ are i.i.d. simulated from $\nu$. That is, \eqref{eq:precommitted.iid} is still required.

Condition \eqref{eq:precommitted.iid} is also commonly assumed among the references mentioned above, as well as in the literature, when applying an MFG equilibrium in an $N$-agent game to achieve $\eps$-optimality. And because of Condition \eqref{eq:precommitted.iid}, Proposition \ref{prop.precommittedN} and the results outlined in Remarks \ref{rm.timeconNstationary}, \ref{rm.nonNstationary} 
indicate that an MFG equilibrium in Definition \ref{def.nonstationary.relax} or Definition \ref{def.stationary.relax}, in no matter time-inconsistent or time-consistent contexts, serves as an approximate equilibrium ($\eps$-equilibrium) for the associated $N$-agent games  but only in a {\it precommitment} sense. That is to say, all agents in the N-agent game must collectively agree to adhere to $\pi$ and will not reconsider {\it after} their initial states have been simulated.

However, from the perspective of a sophisticated agent, a time-inconsistent problem is an inter-temporal game among current self and future selves. Hence, a sophisticated agent shall treat the $N$-agent game as a competition not only with other agents but also with future selves. This means that such an agent is inclined to reconsider and switch to another policy  at any time if the current policy is no longer optimal at that moment. In sum, a sophisticated agent shall look for  subgame-perfect Nash equilibrium as a time-consistent plan such that, given all future selves following this plan, there is no incentive to deviate from it in the current step. Therefore, a {\it truly} $\eps$-equilibrium for {\color{black}a sophisticated agent $i$} in the $N$-agent game shall satisfy, for $N$ big enough that
{\color{black}\be\label{eq.equi.strict0}    
\begin{aligned}
	&\E \left[ J^{N,i,\bmpi^N_{[t:\infty)}}\left(s^{N,i}_t, \mu^N_t\right) \bigg| \mu^N_t\right] \\
	&\geq  \E \left[ \sup_{\varpi \in \Pc(U)}J^{N,i,\left(\left(\bmpi^N_{t}\right)_{-i}, \varpi\right)\otimes_1 \bm\pi^N_{[t+1:\infty)}} \left(s^{N,i}_t, \mu^N_t \right) \bigg| \mu^N_t\right] -\eps \quad \forall t\in \T. 
\end{aligned}
\ee 
And a {\it truly} $\eps$-equilibrium in the $N$-agent game where all $N$ agents are sophisticated, shall satisfy for $N$ big enough that {\footnote{\color{black} In \eqref{eq.equi.strict0} and \eqref{eq.equi.strict}, we do not assume condition \eqref{eq:precommitted.iid} to be in force, so we use $\E$ instead of $\E_\nu$. And the left-hand side should be bigger or equal to the right-hand side regardless of the value of $\mu^N_t$.}}
\be\label{eq.equi.strict}    
\begin{aligned}
		&\E \left[ J^{N,i,\bmpi^N_{[t:\infty)}}\left(s^{N,i}_t, \mu^N_t\right) \bigg| \mu^N_t\right] \\
		&\geq  \E \left[ \sup_{\varpi \in \Pc(U)}J^{N,i,\left(\left(\bmpi^N_{t}\right)_{-i}, \varpi\right)\otimes_1 \bm\pi^N_{[t+1:\infty)}} \left(s^{N,i}_t, \mu^N_t \right) \bigg| \mu^N_t\right] -\eps\quad \text{for } 1\leq i\leq N, \quad \forall t\in \T,
\end{aligned}
\ee}
which is strictly stronger than both \eqref{eq.prop.precommittedN} and \eqref{eq.prop.precommittedNt}, along with the conclusions stated in Remarks \ref{rm.timeconNstationary}, \ref{rm.nonNstationary}. Regrettably, as demonstrated in the following Example \ref{eg:N.precommit}, \eqref{eq.equi.strict} is generally not satisfied by an MFG equilibrium in Definition \ref{def.nonstationary.relax} or in Definition \ref{def.stationary.relax}. 
As a consequence, the results in Proposition \ref{prop.precommittedN} and in Remarks \ref{rm.timeconNstationary}, \ref{rm.nonNstationary} 
 {\it do not} tell that an MFG equilibrium in Definition \ref{def.nonstationary.relax} or in Definition \ref{def.stationary.relax}, when viewed through the eyes of a sophisticated agent, is a truly $\eps$-equilibrium in the $N$-agent game. 

We would like to emphasize that, as highlighted in Remark \ref{rm.explain.def}, an MFG equilibrium in Definition \ref{def.nonstationary.relax} still constitutes a subgame-perfect Nash equilibrium from the perspective of a sophisticated agent within the MFG. 
Indeed, compare with \eqref{eq.equi.strict}, \eqref{eq.def.nonstationary1} tells that an equilibrium policy $(\pi,\mu)$ in Definition \ref{def.nonstationary.relax} for the MFG with initial population distribution $\nu$ satisfies 
$$
\E\left[J^{\pi_{[t:\infty)}}\left(x, \mu_t, \mu_{[t+1:\infty)}\right) \bigg| \mu_t \right]=	\E\left[\sup_{\varpi\in \Pc(U)}J^{\varpi \otimes_1\pi_{[t+1:\infty)}}\left( x,\mu_t, \mu_{[t+1:\infty)}\right)\bigg| \mu_t \right]\quad \forall (t,x)\in \T\times[d].
$$

\begin{Example}\label{eg:N.precommit}
	Let $d=2$ with $[2]=\{1,2\}$ and $U=[0,1]$, we set the transition probabilities as
	\be\label{eq:eg.P}  
	\begin{aligned}
		P(x,\nu,u)=\frac{1}{2}\nu+\left( \frac{1}{4}+(-1)^x\left( \frac{1}{4}-\frac{u}{2} \right), \frac{1}{4}-(-1)^x\left( \frac{1}{4}-\frac{u}{2} \right) \right), \quad \forall (x,\nu)\in \Delta, u\in U.
	\end{aligned}
	\ee 
	Set the reward function $f(t,x,\nu,u):=\delta(t)g(x,\nu,u)$ with 
	\be\label{eq:eg.g}	  
	g(x,\nu,u):=1-\left| \nu(2)u-1/4 \right|,\quad \delta(t)=\frac{1}{2}3^{-t}+\frac{1}{2}4^{-t},
	\ee 
	where $\nu(2)$ represents the second component of $\nu$. {\color{black} Then due to the non-exponential structure of $\delta(t)$, the MFG in this example is time-inconsistent. We further denote by $\bm\delta_a$ the Dirac measure concentrated at a point $a$.} Notice that the reward function is invariant of the state variable $x$, then $v^{\pi,\nu}(1,1)=v^{\pi,\nu}(1,2)$ for any $\pi\in \Pi$ and $\nu\in\Pc([d])$. Then direct calculations verify that $(\pi^{(1)}, \mu^{(1)})\in \Pi\times \Lambda$ with $\pi^{(1)}_{t}\equiv (\bm\delta_{1/2}, \bm\delta_{1/2})$ and $\mu^{(1)}_{t}\equiv (1/2,1/2)=: \nu^{(1)}$ (resp. $(\pi^{(2)}, \mu^{(2)})\in \Pi\times \Lambda$ with $\pi^{(2)}_{t}\equiv (\bm\delta_{1}, \bm\delta_{1})$ and $\mu^{(2)}_{t}\equiv (3/4,1/4)=: \nu^{(2)}$) is an equilibrium in Definition \ref{def.nonstationary.relax} for the MFG with initial distribution $\nu^{(1)}$ (resp. $\nu^{(2)}$). 
	
	{\color{black}Now we apply $(\pi^{(1)}, \mu^{(1)})$ to the corresponding $N$-agent game for big $N=4k$ with $k\in \N$. Let $\bm\pi^N=(\pi^{(1)},\cdots \pi^{(1)} )$.
	Suppose $s^{N,i}_0$ are independently generated from $\nu^{(1)}$, and all agents in the N-agent game plan to follow $\pi^{(1)}$ once their initial states are simulated. 
	Then, with probability $
	\frac{1}{2^{4k}} 
	\left(  
	\begin{matrix}
	4k\\k
	\end{matrix}
	\right)$ the rare event $\{ \mu^N_0=\nu^{(2)} \}$ would happen. 
This suggests that, conditional on $
	\mu^N_0=\nu^{(2)}$ and assuming continued adherence of all other agents to $\pi^{(1)}$, 
	a sophisticated agent 1, instead of following $\pi^{(1)}$ without reconsideration at step 0,  would choose policy $\varpi$ by maximizing
	$$
	\E_\nu\left[ J^{N,1,\left(\left(\bmpi^N_{1}\right)_{-1}, \varpi\right)\otimes_1 \bm\pi^N_{[1:\infty)}} \left(s^{N,1}_0, \mu^N_0 \right) \bigg| \mu^N_0= \nu^{(2)}\right].
	$$
If $s^{N,1}_0=1$, for each $\varpi\in \Pc(U)$, the above pay-off equals
\be\label{eq:eg.1} 
\begin{aligned}
&\E_\nu\left[ J^{N,1,\left(\left(\bmpi^N_{1}\right)_{-1}, \varpi\right)\otimes_1 \bm\pi^N_{[1:\infty)}} \left(1, \mu^N_0 \right) \bigg| \mu^N_0= \nu^{(2)}\right]\\
&=
\int_U \left(g(1, \nu^{(2)}, u )+\sum_{x=1,2}P(1, \nu^{(2)}, x,u )V^{N}(x) \right)\varpi(du) \\
&= 
\int_U \left(1-\left| \frac{u}{4}- \frac14 \right|+\frac u2\left( V^N(1)-V^N(2) \right)+\frac{3}{8}V^N(1)+\frac58V^N(2) \right)\varpi(du)
\end{aligned}
\ee 
where the last line follows from \eqref{eq:eg.P} and \eqref{eq:eg.g}, and
$$
V^N(x):=\E_\nu\left[\sum_{t=1}^\infty \delta(t) g\left(s^{N,1}_t, \mu^{N}_t, 1 \right) \bigg| s^{N,1}_1=x, s^{N,1}_0=1, \mu^N_0=\nu^{(2)}\right] \text{ for $x=1,2$}.
$$ 
Similar to \eqref{eq.precommit.Nmu0},
we have 
$\E_\nu\left[\big|\mu^{N}_t\to \mu^{\pi^{(1)}, \nu^{(2)}}_t\big|\big| \mu^{N}_0=\nu^{(2)}\right] \to 0$ 
for each $t\geq 1$ uniformly over $\varpi\in \Pc(U)$. Notice again that $g$ is independent of the state variable $x$. Thus, we have 
\be  
\lim_{N\to\infty} \sup_{\varpi\in \Pc(U)} |V^N(x) -M| =0 \text{ for $x=1,2$ with } M:=\sum_{t=1}^\infty \delta(t) \left(1-\left|\frac{\mu^{\pi^{(1)},\nu^{(2)}}_t(2)}{2}-\frac14 \right| \right).
\ee 
Then there exists $N_0$ such that $|V^N(1)-V^N(2)|<0.01$ and $|\frac{3}{8}V^N(1)+\frac{5}{8}V^N(2)-M|\leq 0.01$ uniformly over $\varpi\in \Pc(U)$ for all $N\geq N_0$. This together with \eqref{eq:eg.1} tells, for $N\geq N_0$, that 
\be\label{eq:eg}  
\begin{aligned}
&\sup_{\varpi\in \Pc(U)}\E_\nu\left[ J^{N,1,\left(\left(\bmpi^N_{1}\right)_{-1}, \varpi\right)\otimes_1 \bm\pi^N_{[1:\infty)}} \left(1, \mu^N_0 \right) \bigg| \mu^N_0= \nu^{(2)}\right]\\
&=\sup_{\varpi\in\Pc(U)}\int_U \left(1-\left| \frac{u}{4}- \frac14 \right|+\frac u2\left( V^N(1)-V^N(2) \right)+\frac{3}{8}V^N(1)+\frac58V^N(2) \right)d\varpi(u)\\
&\geq \int_U \left(1-\left| \frac{u}{4}- \frac14 \right|+\frac u2\left( V^N(1)-V^N(2) \right)+\frac{3}{8}V^N(1)+\frac58V^N(2) \right)\pi^{(2)}(du)\\
&\geq 1-|1/4-1/4|-0.005+M-0.01\\
&>1-|1/4-1/4|-0.005+M-0.01-0.0975=1-|1/8-1/4|+0.0025+M+0.01\\
&\geq \int_U \left(1-\left| \frac{u}{4}- \frac14 \right|+\frac u2\left( V^N(1)-V^N(2) \right)+\frac{3}{8}V^N(1)+\frac58V^N(2) \right)\pi^{(1)}(du)\\
&=\E_\nu\left[ J^{N,1, \bm\pi^N_{[0:\infty)}} \left(1, \mu^N_0 \right) \bigg| \mu^N_0= \nu^{(2)}\right].
\end{aligned}
\ee 
 If $s^{N,1}_0=2$, a similar argument also yields an integer $N_1$ exists such that
\be\label{eq:eg'}    
\begin{aligned}
	&\sup_{\varpi\in \Pc(U)}\E_\nu\left[ J^{N,1,\left(\left(\bmpi^N_{1}\right)_{-1}, \varpi\right)\otimes_1 \bm\pi^N_{[1:\infty)}} \left(2, \mu^N_0 \right) \bigg| \mu^N_0= \nu^{(2)}\right]\\
	&=\sup_{\varpi\in\Pc(U)}\int_U \left(1-\left| \frac{u}{4}- \frac14 \right|+\frac u2\left( \hat V^N(2)-\hat V^N(1) \right)+\frac{7}{8}\hat V^N(1)+\frac18\hat V^N(2) \right)d\varpi(u)\\
	&\geq \int_U \left(1-\left| \frac{u}{4}- \frac14 \right|+\frac u2\left( \hat V^N(2)-\hat V^N(1) \right)+\frac{7}{8}\hat V^N(1)+\frac18\hat V^N(2) \right)\pi^{(2)}(du)\\
	&\geq 1-|1/4-1/4|-0.005+M-0.01\\
	&>1-|1/4-1/4|-0.005+M-0.01-0.0975=1-|1/8-1/4|+0.0025+M+0.01\\
	&\geq \int_U \left(1-\left| \frac{u}{4}- \frac14 \right|+\frac u2\left( \hat V^N(2)-\hat V^N(1) \right)+\frac{7}{8}\hat V^N(1)+\frac18\hat V^N(2) \right)\pi^{(1)}(du)\\
	&= \E_\nu\left[ J^{N,1, \bm\pi^N_{[0:\infty)}} \left(2, \mu^N_0 \right) \bigg| \mu^N_0= \nu^{(2)}\right] \quad \forall N\geq N_1,
\end{aligned}
\ee 
where $\hat V^N(x):=\E_\nu\left[\sum_{t=1}^\infty \delta(t) g\left(s^{N,1}_t, \mu^{N}_t, 1 \right) \bigg| s^{N,1}_1=x, s^{N,1}_0=2, \mu^N_0=\nu^{(2)}\right]$ for $x=1,2$. By combining \eqref{eq:eg} and \eqref{eq:eg'}, we conclude that, for all $N=4k\geq N_0\vee N_1$, with probability $
\frac{1}{2^{4k}} 
\left(  
\begin{matrix}
	4k\\k
\end{matrix}
\right)$ the rare event $\{ \mu^N_0=\nu^{(2)}\}$ occurs and 
$$
\begin{aligned}
&\sup_{\varpi\in \Pc(U)} \E_\nu\left[ J^{N,1,\left(\left(\bmpi^N_{1}\right)_{-1}, \varpi\right)\otimes_1 \bm\pi^N_{[1:\infty)}} \left(s^{N,1}_0, \mu^N_0 \right) \bigg| \mu^N_0= \nu^{(2)}\right] - 0.0975\\
&\geq \E_\nu\left[ J^{N,1, \bm\pi^N_{[0:\infty)}} \left(s^{N,1}_0, \mu^N_0 \right) \bigg| \mu^N_0= \nu^{(2)}\right] . 
\end{aligned}
$$
That is, the pay-off from applying $\pi^{(1)}$ at time zero is strictly less than the maximum gain from applying other possible strategies at that moment. As a result, the sophisticated agent 1 will deviate from $\pi^{(1)}$ under this circumstance. Consequently, \eqref{eq.equi.strict} does not holds for $\left( \pi^{(1)}, \mu^{(1)} \right)$ with $t=0$ and $i=1$.}
\end{Example}


\section{Consistent Equilibria}\label{sec:sophisticated}
In this section, our objective is to explore a new type of MFG equilibrium, distinct from the equilibrium concept defined in Definition \ref{def.nonstationary.relax}, that can better serve as a truly consistent approximate equilibrium for a sophisticated agent in the $N$-agent game. To achieve this goal, we shall focus on strategies telling an agent how to react to any population distribution $\nu\in \Pc([d])$.
With this in mind, we define $\Pitilde$ the set containing all relaxed type (time-homogeneous) feedback controls $\pitilde$ that are Borel mappings from $\Delta$ to $\Pc(U)$.

We introduce a new equilibrium concept termed {\it consistent equilibrium} among $\Pitilde$ (see Definition \ref{def:sophisticated.ne} below) for the MFG. This type of equilibrium serves as a truly approximate equilibrium from the viewpoint of a sophisticated agent within the N-agent game (see Definition \ref{def:N.sophisticated} and Theorem \ref{thm.MFG.Neps}), thereby addressing Question \eqref{sentence:I}. Furthermore, we provide in Theorem \ref{thm.Nequi.converge} an answer to Question  \eqref{sentence:II}, stating that under suitable conditions a sequence of consistent equilibria in $N$-agent games converges up to a subsequence to a consistent equilibrium in the MFG. 
While the general existence of a consistent MFG equilibrium remains an open question and we aim to address in future studies, we offer an example study in Section \ref{subsec:example}. And we demonstrate in the example the existence of a consistent equilibrium in the MFG.
The proofs of the main results in this section, Theorem \ref{thm.MFG.Neps} and Theorem \ref{thm.Nequi.converge}, are carried out in Section \ref{subsec:proof.sophisticated}. 

\subsection{The concept of consistent MFG equilibria}
Take $\pitilde=(\pitilde(x,\nu))_{(x,\nu)\in \Delta}\in \Pitilde$, 
we first describe the structure of the MFG under $\pitilde$ as follows. Define $P^{\pitilde(\nu)}(\nu):= (P^{\pitilde(x,\nu)}(x,\nu))_{x\in [d]}$ as the transition matrix under $\pitilde$ and $\nu\in \Pc([d])$.
Provided a population distribution flow $\mu=\mu_{[0:\infty)}\in \Lambda$ with starting value $\mu_0=\nu\in \Pc([d])$, we denote by $s^{\pitilde, \mu}=\left(s^{\pitilde, \mu}_t \right)_{t\in \T}$ the dynamic of a single agent for applying a strategy $\pitilde\in \Pitilde$ under $\mu$ and is determined as follows. For each time $t$, an action value $\alpha_t$ is simulated based on $\pitilde \left( s^{\pitilde, \mu}_t, \mu_t \right)$, and then $s^{\pitilde,\mu}_{t+1}$ is determined by the transition probability vector $P \left( s_t^{\pitilde, \mu}, \mu_t, \alpha_t \right)$.
Similar to \eqref{eq.def.J}, we define, for the control $\pitilde\in \Pitilde$ and the flow $\mu\in \Lambda$ above, the expected pay-off when starting at state $x\in [d]$ as
\be\label{eq.J.relax} 
J^{\pitilde}\left(x, \nu, \mu_{[1:\infty)}\right):= \E\left[ \sum_{t=0}^\infty  f^{\pitilde\left( s^{\pitilde, \mu}_t, \mu_t \right)}\left(t, s_t^{\pitilde, \mu}, \mu_t \right) \bigg| s^{\pitilde, \mu}_0=x \right].
\ee 
Notice again that the initial value $\mu_0=\nu$ and the tail $\mu_{[1:\infty)}$ together forms the whole flow $\mu_{[0:\infty)}$.

Similar to \eqref{eq:mu.pi}, given an arbitrary $\nu\in \Pc([d])$, we denote by $\mu^{\pitilde, \nu}=\mu^{\pitilde, \nu}_{[0:\infty)}$ the population flow for all agents applying the same policy $\pitilde\in \Pitilde$, which evolves as follows.
\be\label{eq.mu.population} 
\mu^{\pitilde, \nu}_{t+1}=\sum_{x\in [d]}  \mu^{\pitilde,  \nu}_t(x)\int_U P\left( x, \mu^{\pitilde,  \nu}_t,u \right)\pitilde \left( x,\mu^{\pitilde,  \nu}_t \right)(du)=\mu^{\pitilde,  \nu}_t \cdot P^{\pitilde \left( \mu^{\pitilde,  \nu}_t \right) } ( \mu^{\pitilde,  \nu}_t ),\quad \mu^{\pitilde,  \nu}_0=\nu.
\ee 
And when all agents in the MFG apply the same relaxed policy $\pitilde\in \Pitilde$, we simply write $J^\pitilde(x,\nu)$ instead of $J^{\pitilde}\left(x, \nu,  \mu^{\pitilde,\nu}_{[1:\infty)} \right)$, since the tail $\mu^{\pitilde,\nu}_{[1:\infty)}$ described in \eqref{eq.mu.population} is fully determined by $\mu^\pitilde_0=\nu$ and $\pitilde$. 
For arbitrary $\pitilde\in \Pitilde$ and $\pi\in \Pi$, we denote by $\pi\otimes_k \pitilde$ the strategy of applying $\pi$ for the first $k$ steps then switching to $\pitilde$ afterward, {\color{black} Notice that, the deviation can be taken from $\Pc(U)$ when $k=1$, due to the discrete time setting.}

\begin{Definition}\label{def:sophisticated.ne}
	We call a policy $\pitilde\in \Pitilde$ a consistent equilibrium in the MFG if
	\be\label{eq.def.nesingle} 
	J^{\pitilde}(x,\nu)={\color{black}\sup_{\varpi\in \Pc(U)}J^{ \varpi\otimes_1\pitilde} } \left( x,\nu, \mu^{\pitilde,\nu}_{[1:\infty)}\right)\quad \forall (x,\nu)\in \Delta. 
	\ee	
\end{Definition}

\begin{Remark}
{\color{black}Notice that the population flow on the right-hand side of \eqref{eq.def.nesingle}  is $\mu^{\tilde \pi,\nu} = \mu^{\tilde\pi, \nu}_{[0:\infty)}$. This is because only the representative agent deviates from $\pitilde$ in the MFG, and the population flow is deterministic and is not influenced by changes in the strategy of the representative agent.}
	
We also would like to emphasize that, compared with Condition B for the classic MFG equilibrium concept defined in Definition \ref{def.nonstationary.relax}, the condition \eqref{eq.def.nesingle} above for a consistent MFG equilibrium  is universally required across all population distributions $\nu\in \Pc([d])$. This criterion ensures that, through the application of such $\pitilde$, a single agent always knows how to react to any population distribution.
\end{Remark}

\begin{Remark}\label{rm:sophisticated.give.classic}
{\color{black} For a consistent equilibrium $\pitilde$, let $\mu_{[0:\infty)}$ represent the population flow under $\pitilde$ with an arbitrary initial value. Since the transition function $P(x,\nu,y,u)$ and $\pitilde$ are both time-homogeneous, \eqref{eq.def.nesingle} implies, for each $t\in \T$ that 
	\be\label{eq.rm4.2}  
	J^{\pitilde}\left(s_t, \mu_t, \mu_{[t+1:\infty)}\right) =\sup_{\varpi\in \Pc(U)}	J^{\varpi \otimes_1\pitilde}\left( s_t,\mu_t, \mu_{[t+1:\infty)}\right),
	\ee
	regardless of the values of $s_t$ and $\mu_t$. Here $J^{\varpi \otimes_1\pitilde}\left( s_t,\mu_t, \mu_{[t+1:\infty)}\right)$ represents the pay-off by applying $\varpi$ at $t$ and switching to $\pitilde$ thereafter. Thus, \eqref{eq.rm4.2} means that there is no incentive for the representative to deviate from $\pitilde$ at any moment $t$, given that all other agents and future selves adhere to $\pitilde$.}	
	
Moreover, the consistent equilibrium concept in Definition \ref{def:sophisticated.ne} strictly strengthens the equilibrium concept defined in Definition \ref{def.nonstationary.relax}. Indeed, suppose $\pitilde\in \Pitilde$ is a consistent equilibrium satisfying Definition \ref{def:sophisticated.ne}, then $\pitilde$ generates equilibria that satisfy Definition \ref{def.nonstationary.relax} as follows: For an arbitrary $\nu\in \Pc([d])$, define a coupled flow $\pi=\pi_{[0:\infty)}\in \Pi, \mu=\mu_{[0:\infty)}\in \Lambda$ by iterating on $t$ as
$$
\begin{cases}
	\mu_0:=\nu,\quad \pi_0(x):=\pitilde(x,\nu) \;\; \forall x\in[d];\\
	\mu_{t+1}:=\mu_t\cdot P^{\pitilde(\mu_t)}(\mu_t),\quad \pi_{t+1}(x):= \pitilde(x,\mu_{t+1}) \;\; \forall  x\in [d],\quad  \forall t\in \T. 
\end{cases}
$$
Then $(\pi, \mu)$ is an equilibrium satisfying Definition \ref{def.nonstationary.relax} for the MFG with the arbitrarily chosen initial population distribution $\nu$.
\end{Remark}

\begin{Remark}
When the context is time-consistent, i.e., \eqref{eq.deltaexp} holds, \eqref{eq.def.nesingle} implies, for all $k\in \N$ and $x\in [d]$ that 
$$\sup_{\pi\in \Pi}J^{\pi\otimes_{k+1} \pitilde}\left( x,\nu, \mu^{\pitilde, \nu}_{[1:\infty)}\right)-\sup_{\pi\in \Pi} J^{\pi\otimes_{k}\pitilde}\left( x,\nu, \mu^{\pitilde,\nu}_{[1:\infty)} \right)\leq 0. $$
As a consequence,
	$$
J^{\pitilde} \left( x,\nu, \mu^{\pitilde,\nu}_{[1:\infty)}  \right)=\sup_{\pi\in \Pi} J^{\pi}\left( x,\nu, \mu^{\pitilde,\nu}_{[1:\infty)}  \right)\quad \forall (x,\nu)\in \Delta,
	$$
	which tells  that a consistent MFG equilibrium is the time-homogeneous optimal strategy for a single agent (under population flow $\mu^{\pitilde,\nu}$) universally over all $(x,\nu)\in \Delta$.
\end{Remark}

\subsection{Revisiting the $N$-agent game}

Now we check whether or not a consistent MFG equilibrium in Definition \ref{def:sophisticated.ne} can bring satisfactory results in the $N$-agent games. We start by defining a sequential notations for policies belonging to $\Pitilde$ in the $N$-agent game. 
Recall \eqref{eq.def.DeltaN}.
Let $\pitildebm^N:= (\pitilde^{N,1},...,\pitilde^{N,N})$ be an $N$-tuple policy with $\pitilde^{N,i}\in \Pitilde$ for all $1\leq i\leq N$.
Denote by $\mu^{N,\pitildebm^N, \nu}_t:= \frac{1}{N}\sum_{1\leq j\leq N} e_{s^{N,j}_t}$, for $t\in \T$, the empirical distribution under $\pitildebm^N$ with starting value $\mu^N_0=\nu$, {\color{black}where $\left( s^{N,i}_0, \nu \right)\in \Delta_N$ for each $1\leq i\leq N$}. For $t\in \T$, $s^{N,i}_{t+1}$ is determined by
$
P\left(s^{N,i}_t, \mu^{N,\pitildebm^N,\nu}_t, \alpha^{N,i}_t\right)
$, where $\alpha^{N,i}_t$ are mutually independent random variables generated from $\pitilde^{N,i} \left( s_t^{N,i}, \mu_t^{N,\pitildebm^N,\nu} \right)$. Then the law of the stochastic flow $\left(\mu^{N, \bm\pitilde^N,\nu}_t\right)_{t\geq 1}$ is determined by $\pitildebm^N$ and initial empirical distribution $\nu$. Take $(x,\nu)\in \Delta_N$, for above $\bm \pitilde^N$ {\color{black}we introduce the pay-off for agent $i$ 
as}
\be\label{eq.def.JNagentpitilde}  
{\color{black}\begin{aligned}
	J^{N,i, \pitildebm^N}(x, \nu):=& \E\left[ \sum_{t=0}^\infty  f^{\pitilde^{N,i}\left( s^{N,i}_t, \mu_t^{N,\pitildebm^N,\nu} \right) }\left(t, s^{N,i}_t, \mu^{N,\pitildebm^N,\nu}_t \right) \bigg| s^{N,i}_0 = x, \mu^N_0=\nu \right].\\
\end{aligned}}
\ee
{\color{black}For an $N$-tuple $\pitildebm^N$ and $\varpi\in\Pc(U)$, we denote by $\left( \pitildebm^N_{-i}, \varpi \right) \otimes_{1}\pitildebm^N$ the new $N$-tuple where agent $i$ applies $\varpi$ for one step then switches to $\pitilde^{N,i}$ afterward while each other agent $j$ follows $\pitilde^{N,j}$ for $j\neq i$.}
Now we introduce the \textit{consistent} $\eps$-equilibrium for time inconsistent $N$-agent games as follows.

\begin{Definition}\label{def:N.sophisticated}
	Given $\eps\geq 0$, we call an $N$-tuple $\pitildebm^N=(\pitilde^{N,1},...,\pitilde^{N,N})$ with $\pitilde^{N,i}\in \Pitilde$ for each $1\leq i\leq N$ a consistent $\eps$-equilibrium in the $N$-agent game if 
{\color{black}	\be\label{eq:def.Nsophisticated}  
	J^{N,i,\pitildebm^N}(x, \nu)\geq \sup_{\varpi\in \Pc(U)}J^{N,i,\left( \pitildebm^N_{-i}, \varpi \right) \otimes_{1}\pitildebm^N}(x, \nu)-\eps \quad \forall (x, \nu)\in \Delta_N, \quad \text{for } 1\leq i\leq N.
	\ee }
When $\eps=0$, we simply say a consistent equilibrium instead of a consistent $0$-equilibrium.
\end{Definition}

\begin{Remark}\label{rm:sophisticatedN}
{\color{black}Consider a consistent $\eps$-equilibrium $\pitildebm^N$ in Definition \ref{def:N.sophisticated}. Take any $t\in \T$, condition \eqref{eq:def.Nsophisticated} implies for all $1\leq i\leq N$ that 
\bee
\begin{aligned}
	&\E\left[ J^{N,i,\bm\pitilde^N_{[t:\infty)}}\left(s^{N,i}_t, \mu^N_t\right) \bigg| s^{N,i}_t =x, \mu^N_t=\nu  \right] \\
	&\geq \E\left[ \sup_{\varpi \in \Pc(U)}J^{N,i,\left(\left(\bm\pitilde^N_{t}\right)_{-i}, \varpi\right)\otimes_{1} \bm\pitilde^N_{[t+1:\infty)}} \left(s^{N,i}_t, \mu^N_t \right) \bigg|  s^{N,i}_t =x , \mu^N_t=\nu\right] -\eps\quad \forall (x,\nu)\in \Delta_N,
\end{aligned}
\eee
which is exactly the criterion proposed in \eqref{eq.equi.strict}. 
This signifies that, for each $1\leq i\leq N$, whenever agent $i$ reconsiders,} deviating from $\pitilde$ always results in no superior payoff within an $\eps$-level, given that each other agent $j$ adhere to $\pitilde^{N,j}$. Therefore, $\pitildebm^N$ is a truly approximate equilibrium from the perspective of a sophisticated agent within the $N$-agent game. 

{\color{black} We also want to mention that a consistent equilibrium in an $N$-agent game guides all agents to react to the current empirical distribution, thus one agent's decision influences the empirical distribution, which in turn affects the decisions of other agents.}
\end{Remark}




Now we upgrade the uniform continuity in Assumption \ref{assum:continuous} to the following uniformly Lipschitz assumption.
\begin{Assumption}\label{assum.lipschitz}
	There exists positive constants $L_1, L_2$ such that
	\begin{align}
		\sup_{t\in \T, (x,\nu)\in \Delta} \left( |f(t,x,\nu,u)-f(t,x,\nu,\bar u)|+|P(x,\nu,u)-P(x,\nu,\bar u)| \right)\leq L_1 |u-\bar u| \label{eq.assum3.0}\\
		\sup_{t\in \T, x\in[d], u\in U} \left( |f(t,x,\nu,u)-f(t,x,\bar \nu,u)|+|P(x,\nu,u)-P(x,\bar \nu,u)| \right)\leq L_2|\nu-\bar\nu|. \label{eq.assum3.1}
	\end{align}
\end{Assumption}

From now on, we focus on the following subset of $\Pitilde$, which contains all policies that are weakly continuous on $\nu$.
\be\label{eq.def.Pic} 
\Pitilde_c:= \left\{ \pitilde\in \Pitilde: \text{$\nu\mapsto \pi(x,\nu): \left(\Pc([d]), |\cdot|\right) \mapsto \left(\Pc(U), \Wc_1\right)$, is continuous, for each $x\in [d]$}  \right\}.
\ee 

Now we provide the main results in this section.  Theorem \ref{thm.MFG.Neps} states that a consistent MFG equilibrium in Definition \ref{def:sophisticated.ne} that belongs to $\Pitilde_c$ is a consistent $\eps$-equilibrium for the $N$-agent game when $N$ is sufficiently large. Theorem \ref{thm.Nequi.converge} concerns the convergence of consistent equilibria in $N$-agent games to a consistent equilibrium in the MFG.
\begin{Theorem}\label{thm.MFG.Neps}
	Suppose Assumptions \ref{assum.bound}, \ref{assum.lipschitz} hold. Let $\pitilde\in \Pitilde_c$ be a consistent MFG equilibrium as defined in Definition \ref{def:sophisticated.ne}. Then for any $\eps>0$, there exists $N_0$ such that for all $N\geq N_0$, the replicating $N$-tuple $\pitildebm^N=(\pitilde,...,\pitilde)$ is a consistent $\eps$-equilibrium in Definition \ref{def:N.sophisticated} for the $N$-agent game.
\end{Theorem}

\begin{Theorem}\label{thm.Nequi.converge}
	Suppose Assumptions \ref{assum.bound}, \ref{assum.lipschitz} hold. Let $\left( \pitilde^N \right)_{N\in \N\cup\{\infty\}}$ be a sequence in $\Pitilde_c$ such that, 
	\be\label{eq.thm.Nequi1} 
	\lim_{N\to\infty}\sup_{(x,\nu)\in \Delta} \Wc_1 \left( \pitilde^N(x,\nu), \pitilde^\infty(x,\nu) \right)=0.
	\ee 
	 If, for each $N\in \N$, the replicating $N$-tuple $\pitildebm^N=(\pitilde^N,\cdots, \pitilde^N)$ is a consistent equilibrium in Definition \ref{def:N.sophisticated} for the $N$-agent game, then $\pitilde^\infty$ a consistent MFG equilibrium in Definition \ref{def:sophisticated.ne}.
\end{Theorem}

A quick corollary of Theorem \ref{thm.Nequi.converge} is the following.

\begin{Corollary}\label{cor.Nequi.converge}
	Suppose Assumptions \ref{assum.bound}, \ref{assum.lipschitz} hold. Let $\left( \pitilde^N \right)_{N\in \N}$ be a sequence in $\Pitilde_c$ such that, for each $x$, $(\pitilde^N(x,\mu))_{N\in \N}$ is uniformly equi-continuous as functionals from $(\Pc([d]), |\cdot|)$  to $(\Pc(U), \Wc_1)$. If, for each $N$, the replicating $N$-tuple $\pitildebm^N=(\pitilde^N,\cdots, \pitilde^N)$ is a consistent equilibrium in Definition \ref{def:N.sophisticated} for the $N$-agent game, then any subsequence limit of $\left( \pitilde^N \right)_{N\in \N}$ is a consistent MFG equilibrium in Definition \ref{def:sophisticated.ne}.
\end{Corollary}

\begin{proof}
	By the equi-continuity of $(\pitilde^N)_{N\in \N}$ and compactness of $(\Pc([d]), |\cdot|)$ and $(\Pc(U), \Wc_1)$, we can apply Arzel\`a–Ascoli theorem to reach that any subsequence of $(\pitilde^N)_{N\in \N}$ has a convergent subsequence $(\pitilde^{N_k})_{k\in \N}$ with a limit $\pitilde^\infty$ in $\Pitilde_c$, and
	$$
	\lim_{k\to\infty}\sup_{(x,\nu)\in \Delta} \Wc_1 \left( \pitilde^{N_k}(x,\nu), \pitilde^\infty(x,\nu) \right)=0.
	$$
	Then Theorem \ref{thm.Nequi.converge} shows that $\pitilde^\infty$ is a consistent MFG equilibrium in Definition \ref{def:sophisticated.ne}.
\end{proof}

\subsection{An example analysis for the existence of continuous consistent MFG equilibria}\label{subsec:example}

We first provide a characterization result for a policy $\pitilde\in \Pitilde$ to be a consistent MFG equilibrium in Definition \ref{def:sophisticated.ne}. To begin with, for $\pitilde\in \Pitilde$,  we define the following auxiliary value function
\be\label{eq.V.relax} 
V^{\pitilde}\left(x, \nu\right):= \E\left[ \sum_{t=0}^\infty   f^{\pitilde \left( s^{\pitilde,\mu^{\pitilde,\nu}}_t, \mu^{\pitilde,\nu}_t \right)}\left(1+t, s_t^{\pitilde, \mu^{\pitilde,\nu}}, \mu^{\pitilde,\nu}_t \right) \Bigg | s^{\pitilde, \mu^{\pitilde,\nu}}_0=x\right].
\ee 
\begin{Proposition}\label{prop.relax.cha}
	Suppose Assumption \ref{assum.bound} hold, then $\pitilde\in \Pitilde$ is a consistent MFG equilibrium in Definition \ref{def:sophisticated.ne} if and only if 
	\be\label{eq.prop.charelax1}  
	\supp (\pitilde (x,\nu))\subset \argmax_{u\in U} \left\{ f(0,x,\nu,u)+P(x,\nu,u)\cdot V^{\pitilde}\left( \cdot, \nu\cdot P^{\pitilde(\nu)}(\nu)\right) \right\} \quad  \forall (x,\nu)\in \Delta.
	\ee
\end{Proposition}

\begin{proof}
	Take $\pitilde\in \Pitilde$, Assumption \ref{assum.bound} ensures that $(x,\nu)\mapsto V^{\pitilde}( x,\nu)$ is well-defined and uniformly bounded on $\Delta$. Again, the continuity of $P,f$ on $u$ implies that the argmax set in \eqref{eq.prop.charelax1} is non-empty.
	Moreover, for any $\varpi\in \Pc(U)$, a direct calculation together with \eqref{eq.V.relax} tells that
	\be\label{eq.mfgV.pivarpi} 
	{\color{black}J^{\varpi\otimes_1 \pitilde}\left(x,\nu,\mu^{\pitilde, \nu}_{[1:\infty)} \right) } =  f^\varpi(0, x, \nu) + P^\varpi(x, \nu)\cdot V^{\pitilde}\left(\cdot, \mu^{\pitilde,\nu}_1 \right) \quad \forall (x,\nu)\in \Delta,
	\ee 
	and
$	\mu^{\pitilde,\nu}_1 =\nu\cdot P^{\pitilde(\nu)}(\nu)$ due to \eqref{eq.mu.population}.
	Then Definition \ref{def:sophisticated.ne} indicates that $\pitilde$ is a consistent MFG equilibrium if and only if 
	\bee
	\pitilde(x,\nu)\in \argmax_{\varpi\in \Pc(U)}\left\{ \int_U \left[ f(0, x, \nu, u) + P(x,\nu, u)\cdot V^{\pitilde}\left( \cdot,\nu\cdot P^{\pitilde(\nu)}(\nu) \right) \right]  \varpi(du)\right\}.
	\eee 
	which is equivalent to \eqref{eq.prop.charelax1}.
\end{proof}



Based on Proposition \ref{prop.relax.cha}, for $\pitilde\in \Pitilde$ we define that
\be\label{eq.Gamma.operator}  
\begin{aligned}
	&\Gamma(\pitilde):=\\
	&\left\{ \pitilde'\in \Pitilde: \supp(\pitilde'(x,\nu)) \subset \argmax_{u\in U} \left\{ f(0,x,\nu,u)+ P(x,\nu,u)\cdot V^{\pitilde}\left(\cdot, \mu^{\pitilde,\nu}_1 \right) \right\}, \forall (x,\nu)\in \Delta \right\}.
\end{aligned}
\ee 
with
	\be\label{eq:mupitilde1}  
\mu^{\pitilde,\nu}_1 = \nu\cdot P^{\pitilde(\nu)}(\nu).
	\ee 
Then $\Gamma$ maps from $\Pitilde$ to $2^{\Pitilde}$.
And Proposition \ref{prop.relax.cha} together with \eqref{eq.Gamma.operator} implies the following corollary, which characterizes a consistent MFG equilibrium as a fixed point of $\Gamma$.
\begin{Corollary}\label{cor:sophisticated.fixedpoint}
	Suppose Assumption \ref{assum.bound} hold. A policy $\pitilde\in \Pitilde$ is a consistent MFG equilibrium in Definition \ref{def:sophisticated.ne} if and only if 
$
	\pitilde\in  \Gamma \left( \pitilde\right).
$
\end{Corollary}

The existence of a consistent MFG equilibrium in $\Pitilde_c$ remains an open question, which we intend to explore in future studies. In fact, $\Gamma$ does not necessarily map from $\Pitilde_c$ to $2^{\Pitilde_c}$, and establishing a fixed point argument as in Section \ref{subsec:proof.nonstationary} within the set $\Pitilde_c$ presents a non-trivial challenge. 

During the rest of this subsection, we present an example study in which sufficient conditions are provided for the existence of a consistent MFG equilibrium within $\Pitilde_c$. 
Take $d=2$. {\color{black}Through this example, we shall use the first component to represent a population distribution in $\Pc([2])$. That is, $\nu\in[0,1]$ represents the population value $(\nu, 1-\nu)$. And we simply take
$\Delta:=\{1,2\}\times[0,1]$. 
Let $U=[0,1]$ and the transition probability function $P(x,\nu,u)$ take form
\bee  
P(x,\nu,u)= \frac{1}{2} (\nu, 1-\nu) + \left( \frac{1}{4}+(-1)^x\left( \frac{1}{4}-\frac{u}{2} \right), \frac{1}{4}-(-1)^x\left( \frac{1}{4}-\frac{u}{2} \right) \right) \quad \forall (x,\nu)\in \Delta, u\in U.
\eee 
That is,
\be\label{eq:eg.new0}  
\begin{aligned}
	P(1,\nu,u)= \frac{1}{2} (\nu, 1-\nu) +\frac{1}{2} (u,1-u), \;
	P(2,\nu,u)=& \frac{1}{2} (\nu, 1-\nu) +\frac{1}{2} (1-u,u).
\end{aligned}
\ee
Take $f(t,x,\nu,u):=\delta(t) g(x,\nu,u)$ where $\delta(t)$ takes the weighted discounting form $\delta(t)= \int_0^1 \rho^{t} dF(\rho)$ for some cumulative distribution function $F$ on interval $[0,1)$ and 
$$
g(x,\nu,u):= -u^2+(\frac{1}{2}+\nu) u+x.
$$
Weighted discounting is a common factor that causes time-inconsistency. Most commonly used discount functions also obey the weighted discounting form; see, e.g., \cite{MR4124420} for a detailed discussion.
In addition, Assumptions \ref{assum.bound}, \ref{assum.lipschitz} hold in this example. 
\begin{Proposition}
	Suppose $F \left( \frac{1}{7} \right)=1$, then a consistent MFG equilibrium as defined in Definition \ref{def:sophisticated.ne} exists within $\Pitilde_c$.
\end{Proposition}

\begin{proof}
	During the proof, for a function $h(x,\nu):\Delta\to \R$, we denote by $h_{\nu}$ the derivative of $h$ on the second variable $\nu$ if it exists. 
	
	For each $(x,\nu)\in \Delta$, $P(x,\nu,u)$ is linear in $u$ and $f(0,x,\nu,u)=g(x,\nu,u)$ is strictly concave in $u$. Then Corollary \ref{cor:sophisticated.fixedpoint} together with \eqref{eq.Gamma.operator} implies that any consistent equilibrium
	takes Dirac form distributions on $U$. Consequently, any consistent equilibrium $\pitilde\in \Pitilde_c$ can be treated as a function $\Delta\to U$ that is continuous in $\nu$, and we denote by $\Pitilde_{c, \operatorname{Dirac}}$ the family of all such functions. 	
	Then, for any consistent equilibrium $\pitilde\in \Pitilde_{c, \operatorname{Dirac}}$,
	\be\label{eq:eg.new1} 
	P^{\tilde\pi(x,\nu)}(x,\nu,y) = P(x,\nu,y,\tilde\pi(x,\nu)), \quad g^{\tilde\pi(x,\nu)}(x,\nu) =  g(x,\nu,\tilde\pi(x,\nu)).
	\ee 
	Moreover, 
	$\Gamma$ is now an operator from $\Pitilde_{c, \operatorname{Dirac}}$ to itself. Next, we will prove that 
	\be\label{eq:eg.fixpoint}  
	\left\| 	\left( \Gamma\left( \pitilde \right) \right)_\nu \right\|_{L^\infty(\Delta)} \leq \frac{3}{2} \quad \text{if} \; \pitilde\in \Pitilde_{c, \operatorname{Dirac}} \text{ and }\|\pitilde_\nu\|_{L^\infty(\Delta)}\leq \frac{3}{2},
	\ee 
	showing that $\Gamma$ maps from a convex and compact subset of $\Pitilde_{c, \operatorname{Dirac}}$ to itself. Then Brouwer's fixed point theorem yields the existence of a fixed point, and thus, a consistent equilibrium exists in $\Pitilde_{c, \operatorname{Dirac}}$.
	
	To verify \eqref{eq:eg.fixpoint}, we take an arbitrary $\pitilde\in \Pitilde_{c, \operatorname{Dirac}}$ that is differentiable in $\nu$ such that
	$$
	\left\| \pitilde_{\nu} \right\|_{L^\infty(\Delta)}=\sup_{(x,\nu)\in \Delta} \left| \pitilde_{\nu}(x,\nu) \right|\leq \frac{3}{2}.
	$$
By \eqref{eq:eg.new0} and \eqref{eq:eg.new1}, 
	$\mu^{\pitilde,\nu}_1$ in \eqref{eq:mupitilde1} now obeys the formula
	\begin{align}
		\mu^{\pitilde,\nu}_1=& \nu P\left( 1,\nu,1,\pitilde(1,\nu)\right)+(1-\nu)P(2, \nu, 1,\pitilde(2,\nu))  \notag\\
		=&\nu \left(\frac{1}{2} \nu+\frac{1}{2} \pitilde(1,\nu) \right)+ (1-\nu) \left(\frac{1}{2}\nu+\frac{1}{2}(1-\pitilde(2,\nu)) \right) \notag\\
		=&\frac 12+\frac{1}{2}\nu \pitilde(1,\nu)-\frac{1}{2}(1-\nu)\pitilde(2,\nu)=:w(\nu). \label{eq:eg.mu1}
	\end{align}
Taking derivative yields
	\bee
	w_{\nu}(\nu) =\frac12 \pitilde (1,\nu) +\frac12 \pitilde(2,\nu) +\frac 12 \nu \pitilde_\nu(1,\nu)-\frac 12 (1-\nu) \pitilde_\nu(2,\nu),
	\eee 
	which tells that 
	\be\label{eq:eg.new}    
 \|w_\nu\|_{L^\infty([0,1])} \leq 1+ \frac 12 \|\pitilde_\nu\|_{L^\infty(\Delta)}.
  \ee 
	For any $\rho\in [0,1)$. We define for each $x\in[2]$ that
	\bee
	J^\pitilde(x,\nu; \rho):= \E\left[ \sum_{t=0}^\infty  \rho^{-t}g^{\pitilde\left( s^{\pitilde, \mu^{\pitilde,\nu}}_t, \mu^{\pitilde,\nu}_t \right)}\left(t, s_t^{\pitilde, \mu^{\pitilde,\nu}}, \mu^{\pitilde,\nu}_t \right) \right].
	\eee 
	Then $J^{\pitilde}(x,\nu;\rho)$ is differentiable on variable $\nu$ and 
	\be\label{eq:eg.derivweight} 
	V^\pitilde(x,\nu)=\int_0^1 \rho J^{\pitilde}(x,\nu; \rho) dF(\rho),\quad V_\nu^\pitilde(x,\nu)=\int_0^1 \rho J_\nu^{\pitilde}(x,\nu; \rho) dF(\rho)\quad  \forall x\in[2].
	\ee 
	A direct calculation together with \eqref{eq:eg.new1} and \eqref{eq:eg.mu1} shows that 
	\be\label{eq:Vrho} 
	\begin{aligned}
		J^{\pitilde}(x,\nu; \rho)=&g(x, \nu, \pitilde(x,\nu))+ \sum_{y\in[2]}\rho P\left( x,\nu, y, \pitilde(x,\nu) \right)J^{\pitilde}\big(y,w(\nu ); \rho\big)\quad  \forall (x,\nu)\in \Delta.
	\end{aligned}
	\ee 
	For $x=1,2$, taking derivative on $\nu$ for both sides of \eqref{eq:Vrho} yields
	\bee  
	\begin{aligned}
		J^{\pitilde}_\nu(1,\nu; \rho)
		=&-2\pitilde(1,\nu)\pitilde_\nu(1,\nu)+\pitilde(1,\nu)+\left(\frac{1}{2}+\nu \right)\pitilde_\nu(1,\nu)\\
		& +\frac{1}{2}\rho \bigg[ \left( 1+\pitilde_\nu (1,\nu) \right)J^{\pitilde} \big(1,  w(\nu); \rho \big) +\left(-1-\pitilde_\nu(1,\nu) \right)J^{\pitilde} \big(2,  w(\nu); \rho \big) \bigg]\\
		& +\frac{1}{2}\rho \bigg[ \left( \nu+\pitilde (1,\nu) \right)J^{\pitilde}_\nu \big(1,  w(\nu); \rho \big) +\left( 2-\nu-\pitilde(1,\nu) \right)J^{\pitilde}_\nu \big(2,  w(\nu); \rho \big) \bigg]w_\nu(\nu),\\
		J^{\pitilde}_\nu(2,\nu ; \rho)=&-2\pitilde(2,\nu)\pitilde_\nu(2,\nu)+\pitilde(2,\nu)+\left( \frac{1}{2}+\nu \right)\pitilde_\nu(2,\nu)\\
		& +\frac{1}{2}\rho \bigg[ (1-\pitilde_\nu (2,\nu))J^{\pitilde} \big(1,  w(\nu); \rho \big) +(-1+\pitilde_\nu(2,\nu))J^{\pitilde} \big(2,  w(2       ,\nu); \rho \big) \bigg]\\
		& +\frac{1}{2}\rho \bigg[ (\nu+1-\pitilde (2,\nu))J^{\pitilde}_\nu \big(1,  w(\nu); \rho \big) +(1-\nu+\pitilde(2,\nu))J^{\pitilde}_\nu \big(2,  w(\nu); \rho \big) \bigg]w_\nu(\nu).
	\end{aligned}
	\eee
	Notice that $\pitilde\in[0,1]$ and $|J^{\pitilde}(x,\nu; \rho)|\leq \frac{1}{1-\rho}\sup_{(x,\nu)\in \Delta, u\in U}|g(x,\nu,u)|=\frac{3}{1-\rho}$. Then the above two equalities together with \eqref{eq:eg.new} imply that
	\bee 
	\begin{aligned}
		\left\|  J^{\pitilde}_\nu(\cdot ; \rho) \right\|_{L^\infty(\Delta)}\leq & 1+ 2 \| \pitilde_\nu \|_{L^\infty(\Delta)}+\frac 12 \rho\left[  2\cdot \frac{3}{1-\rho}\left(1+ \| \pitilde_\nu \|_{L^\infty(\Delta)} \right)\right]\\
		&+\frac 12 \rho\left[ 2 \left\|  J^{\pitilde}_\nu(\cdot ; \rho) \right\|_{L^\infty(\Delta)} \right]\left( 1+ \frac{1}{2}\| \pitilde_\nu \|_{L^\infty(\Delta)} \right).
	\end{aligned}
	\eee  
	Since $\| \pitilde_\nu \|_{L^\infty(\Delta)} \leq \frac 32$, $\rho \left(1+ \frac 12 \| \pitilde_\nu \|_{L^\infty(\Delta)} \right)\leq \frac 14$ for all $0\leq \rho\leq \frac 17$. Then the above inequality yields
	\be\label{eq:eg.Vrhonu} 
	\begin{aligned}
	\left\|  J^{\pitilde}_\nu (\cdot ; \rho)\right\|_{L^\infty(\Delta)}\leq & \frac{1+ 2 \| \pitilde_\nu \|_{L^\infty(\Delta)}+  \frac{3\rho}{1-\rho}\left(1+ \| \pitilde_\nu \|_{L^\infty(\Delta)} \right)}{1-\rho \left( 1+\frac 12 \| \pitilde_\nu \|_{L^\infty(\Delta)} \right) }\\
	\leq & \frac{1+3+\frac{1}{2}(1+\frac 32)}{1-\frac 14} = 7\quad \forall  0< \rho\leq \frac{1}{7}.
	\end{aligned} 
	\ee 
	Meanwhile, for $x\in[2]$, set 
	\be\label{eq:eg.new5} 
	\begin{aligned}
		0=&\frac{\partial}{\partial u} \left( g(x,\nu,u)+P(x,\nu,u)\cdot V^{\pitilde}\left(\cdot, w(\nu) \right) \right)\\ 
		= &-2u +\frac{1}{2}+\nu+(-1)^x \frac{1}{2}\left[   V^{\pitilde}\left(2, w(\nu) \right) -V^{\pitilde}\left(1, w(\nu) \right) \right].
	\end{aligned}
	\ee 
	By \eqref{eq:eg.derivweight}, $F(\frac 17)=1$ and $|J^{\pitilde}(\cdot ; \rho)|\leq \frac{3}{1-\rho}$, $|V^{\pitilde}(\cdot;\rho)|\leq \frac 12$, so the right-hand side of \eqref{eq:eg.new5} satisfies
	$$
	0\leq \frac{1}{2}+\nu+(-1)^x \frac{1}{2}\left[   V^{\pitilde}\left(2, w(\nu) \right) -V^{\pitilde}\left(1, w(\nu) \right) \right]\leq 2,
	$$ 
	which tells the root of \eqref{eq:eg.new5} belongs to $[0,1]$.
	This together with \eqref{eq.Gamma.operator} and the strict concavity of $g$ on $u$ yields that 
	$$
	 \Gamma\left( \pitilde \right) =\frac{1}{2}\left[\frac{1}{2}+\nu+(-1)^x \frac{1}{2}\left(  V^{\pitilde}\left(2,  w(\nu) \right) - V^{\pitilde}\left(1,  w(\nu) \right) \right) \right].$$
	Then by taking derivative of the above equality and combining with \eqref{eq:eg.new}, we have
	\begin{align*} 
		&\left\| 	\left( \Gamma\left( \pitilde \right) \right)_\nu  \right\|_{L^\infty(\Delta)}\leq \frac{1}{2}\left[ 1+ \left\|  V^{\pitilde}_\nu\left(\cdot, \nu \right)   \right\|_{L^\infty(\Delta)} \left( 1+\frac{1}{2}\left\| \pitilde_\nu \right\|_{L^\infty(\Delta)}\right)\right] \\
		&\leq \frac{1}{2}\left[ 1+ \left( \int_0^1 \rho  \left\|  J^{\pitilde}_\nu (\cdot ; \rho)\right\|_{L^\infty(\Delta)} dF(\rho) \right) \left( 1+\frac{1}{2}\left\| \pitilde_\nu \right\|_{L^\infty(\Delta)}\right) \right]\\
		&\leq \frac{1}{2}\left[ 1 +\left(\int_0^\frac{1}{7} \rho 7dF(\rho) \right)\frac{7}{4} \right]\leq \frac{1}{2}\left( 1+\frac{1}{7}\cdot 7\cdot \frac{7}{4} \right)=\frac{11}{8}<\frac{12}{8}= \|\pitilde_\nu\|_{L^\infty(\Delta)},
	\end{align*}
	where the second line above follows from \eqref{eq:eg.derivweight} and the third line above follows from \eqref{eq:eg.Vrhonu}. Hence, \eqref{eq:eg.fixpoint} holds.
\end{proof}
}

\begin{Remark}
	Even though we assume that $\pitilde$ is differential on $\nu$ in the proof, it is sufficient to assume that $\pitilde$ is Lipschitz continuous with a bounded Lipschitz norm less or equal to $3/2$. Then by the calculations in the proof, one can deduce that the same upper bound applies to the bounded Lipschitz norm of $\Gamma\left(\pitilde \right)$. Consequently, the existence result remains valid.
	
\end{Remark}

\section{Proofs for Theorem \ref{thm.MFG.Neps} and Theorem \ref{thm.Nequi.converge}}\label{subsec:proof.sophisticated}

Throughout this section, we consistently assume that Assumptions \ref{assum.bound}, \ref{assum.lipschitz} hold. {\color{black}Notice that all consistent ($\eps$)-equilibria for $N$-agent games in the statement of Theorem \ref{thm.MFG.Neps} and Theorem \ref{thm.Nequi.converge} are replicating $N$-tuples, i.e., $\pitildebm^N=(\pitilde,...,\pitilde)$ for some $\pitilde\in \Pitilde$. Therefore, we always take agent 1 as the representative agent for an $N$-agent game through this section.} To start the proofs, we first provide a useful formula of $J^{N,1, {\color{black}({\pitildebm}^N_{-1}, \varpi)\otimes_1{\pitildebm}^N} }(x, \nu)$ where $\pitildebm^N=(\pitilde,...,\pitilde)$ with $\pitilde\in \Pitilde$,  (see \eqref{eq.N.Jpasting}) in Section \ref{subsubsec:formula.Jpasting}. Then several lemmas are provided in Section \ref{subsubsec:lms}. The proofs for Theorem \ref{thm.MFG.Neps} and Theorem \ref{thm.Nequi.converge} are conducted in Section \ref{subsubsec:proof.Neps} and Section \ref{subsubsec:proof.Nconverge}, respectively.

\subsection{A formula of $J^{N,1,( \vect{{\tilde \pi}}^N_{-1}, \varpi)\otimes_1{\vect{\pitilde}}^N}(x, \nu)$}\label{subsubsec:formula.Jpasting}

Take a replicating $N$-tuple $\pitildebm^N=(\pitilde,...,\pitilde)$ with $\pitilde\in \Pitilde$ and a probability measure $\varpi\in \Pc(U)$. Pick $(x,\nu)\in \Delta_N$. We consider the pay-off $J^{N,1, {\color{black}\left( \pitildebm^N_{-1}, \varpi \right)\otimes_{1}\pitildebm^N} }(x, \nu)$. Similar to the auxiliary function defined in \eqref{eq.V.relax} for MFG, we introduce the following auxiliary function for the replicating $N$-tuple $\pitildebm^N$ as
\be\label{eq.def.VNagentpitilde}
\begin{aligned}
	V^{N,\pitilde}(x, \nu):=& \E\left[ \sum_{t=0}^\infty  f^{\pitilde^{N,1}\left( s^{N,1}_t, \mu_t^{N,\pitildebm^N,\nu} \right) }\left(1+t, s^{N,1}_t, \mu^{N,\pitildebm^N,\nu}_t \right) \Bigg| s^{N,1}_0 = x \right].
\end{aligned}
\ee

At $t= 0$, the $N$-tuple $\left( \pitildebm^N_{-1},\varpi \right)$ is applied, and agent 1 simulates $\alpha^{N,1}_0$ using $\varpi$ independently of other agents. Then a direct calculation yields
\be\label{eq:JN.pitildevarpi} 
\begin{aligned}
&J^{N,1,\left( \pitildebm^N_{-1}, \varpi \right)\otimes_{1}\pitildebm^N}(x, \nu)\\
&=\E\left[  f^\varpi\left(0, s^{N,1}_0,\nu \right)+	V^{N, \pitilde}\left(s^{N,1}_1, \mu^{N,\bm \pitilde^N,\nu}_1 \right) \Big| s^{N,1}_0= x\right]\\
&=\int_U \E\left[  f\left(0, x,\nu, u \right)+\sum_{y\in [d]}P(x,\nu,y,u)V^{N,\pitilde}\left(y, \mu^{N,\bm \pitilde^N,\nu}_1 \right)\Bigg| s^{N,1}_0= x, \alpha^{N,1}_0=u \right] \varpi(du).
\end{aligned}
\ee
 {\color{black}Notice that $\E\left[\mu^{N,\bm \pitilde^N,\nu}_1 \big| s^{N,1}_0= x, \alpha^{N,1}_0=u \right]$, and thus $\E\left[V^{N,\pitilde}\left(y, \mu^{N,\bm \pitilde^N,\nu}_1 \right)\big| s^{N,1}_0= x, \alpha^{N,1}_0=u \right]$ in \eqref{eq:JN.pitildevarpi}, are permutation invariant w.r.t. the order of agents $i=2,\cdots, N$, thus are fully determined by values $(x,\nu)\in \Delta_N$, $y\in[d]$, $u\in U$ and control $\pitilde\in \Pitilde$. Indeed, for $(x,\nu)\in \Delta_N$ chosen above, suppose the initial states $\bm{s}^{N}_0=(s^{N,1}_0,..., s^{N,N}_0)$ with $s^{N,1}_0=x$ and $\frac{1}{N}\left(\sum_{1\leq i\leq N} e_{s^{N,i}_0}\right)=\nu$. Then we can reorder agents 2 to $N$ in the way that the agents with smaller indices have smaller starting states. That is, we define a new vector $\bm{ \bar s}^{N}(x,\nu)=(\bar s^{N, 1}(x,\nu),..., \bar s^{N, N}(x,\nu))$ only dependent on $N$ and $(x,\nu)\in \Delta_N$ as follows: 
 \be\label{eq.def.sNnu} 
 \begin{aligned} 
 	&\bar{s}^{N,1}(x,\nu):=x,\; \bar{s}^{N,2}(x,\nu)=\cdots=\bar{s}^{N, 1+N\nu(1)+(-1)\bm1_{\{ x=1\}} }(x,\nu)=1,\\ &\bar{s}^{N,{2+N\nu(1)}+(-1)\bm\delta_{\{ x=1\}}}(x,\nu)=\cdots=\bar{s}^{N,{1+N\nu(1)}+N\nu(2)+\sum_{y=1}^2(-1)\bm1_{\{ x=y\}}}(x,\nu)=2, \;\text{and so on},
 \end{aligned}
 \ee
 here $\bm1_{\{x=y\}}$ means the indicator function with value equal to 1 when $x=y$ and equal to 0 otherwise.} Then the game starting with $\bm s^N_0$ and the game starting with $\bm{ \bar s}^{N}(x,\nu)$ exhibit the same joint empirical flow and dynamic for agent 1 in law.
Moreover, by considering $\alpha^{N,i,\nu}_0\sim \pitilde\left( \bar{s}^{N,i}(x,\nu), \nu \right)$, $2\leq i\leq N$, as mutually independent random variables, we have for any $u\in U$ that
      	\be\label{eq:YnuxN} 
      	\begin{aligned}
      {\color{black}	Y^{N,\pitilde}(x,\nu,u):=}& {\color{black}\P\left( \mu^{N,\pitildebm^N, \nu}_1 \Big| s^{N,1}_0=x, \alpha^{N,1}_0=u \right)}\\      	    
        {\color{black} =}&  {\color{black}\frac{1}{N}\left[e_{x}\cdot P(\nu, u)+\sum_{i=2}^N e_{\bar s^{N, i}(x,\nu)} \cdot P\left( \nu, \alpha^{N,i,\nu}_0 \right) \right] \; \text{in distribution},}
      	   \end{aligned}  
      	\ee 
the last line above is random, of which the distribution is fully determined by $N\in \N, (x,\nu)\in \Delta_N, u\in U$ and $\pitilde\in \Pitilde$. {\color{black}That is, we represents the conditional distribution of $\mu^{N,\pitildebm^N, \nu}_1$ on $\{s^{N,1}_0=x, \alpha^{N,1}_0=u\}$ as the random variable $Y^{N,\pitilde}(x,\nu,u)$, whose distribution is fully determined by $N\in \N, (x,\nu)\in \Delta_N, u\in U$ and $\pitilde\in \Pitilde$.} Then, 
 \begin{align}
 	&\E\left[V^{N,\pitilde}\left(y, \mu^{N,\bm \pitilde^N,\nu}_1 \right)\Bigg| s^{N,1}_0= x, \alpha^{N,1}_0=u \right]
 	= \E\left[V^{N,\pitilde}\left(y, {\color{black} Y^{N,\pitilde}(x,\nu,u) } \right) \right] \notag \\
 	&=\int_U\cdots \int_U V^{N,\pitilde} \left(y,\frac{1}{N}\left[e_{x}\cdot P(\nu, u)+\sum_{i=2}^N e_{\bar s^{N, i}(x,\nu)} \cdot P\left( \nu, u_i \right) \right]  \right) \prod_{i=2}^N\pitilde\left( \bar s^{N,i}(x,\nu), \nu \right) (du_i) \notag\\
 	&=: W^{N, \pitilde}(x,\nu,y,u) \label{eq.def.WN}.
 \end{align}
The integral in second line above is fully determined by $N\in \N, (x,\nu)\in \Delta_N, y\in[d], u\in U$ and $\pitilde\in \Pitilde$. Thus, for fixed $N\in \N$ and $\pitilde\in \Pitilde$, $W^{N, \pitilde}(x,\nu,y,u)$ defines a new auxiliary function maps from $\Delta_N\times[d]\times  U$ to $\R$.
And by combining \eqref{eq:JN.pitildevarpi} with \eqref{eq.def.WN}, we reach that
\be\label{eq.N.Jpasting}   
\begin{aligned}
	&J^{N,1,\left( \pitildebm^N_{-1}, \varpi \right)\otimes_{1}\pitildebm^N}(x, \nu) \\
	&= \int_U \left[ f(0,x,\nu, u)+\sum_{y\in[d]} P(x,\nu, y, u)
	W^{N, \pitilde}(x,\nu,y,u)\right] \varpi(du). 
\end{aligned} 
\ee

\subsection{Lemmata}\label{subsubsec:lms}

We present several lemmas in this section which serve as key gradients for the proofs in Sections \ref{subsubsec:proof.Neps}, \ref{subsubsec:proof.Nconverge}. To begin with, we further introduce the following notation. For an integer $N\geq 2$ and $(x,\nu)\in \Delta$, define $\nu(N,x)\in \Pc([d])$ as the point  that is closest to $\nu$ under the Euclidean norm and satisfies that $N\nu(N,x)\in \N^d$ with $(\nu(N,x))(x)>0$. When multiple such closest points exist, we simply choose the point that has the larger entries in larger indices. For example, by taking $\nu=(\frac12,\frac16,\frac13)$, $N=3, x=3$, $\nu(3,3)=(\frac{1}{3}, \frac13,\frac13)$ instead of $(\frac23, 0, \frac13)$. Notice that
\be\label{eq.muNx.converge}  
\sup_{(x,\nu)\in \Delta} |\nu(N,x)-\nu|\to 0 \quad \text{as } N\to\infty.
\ee 

Recall \eqref{eq.V.relax}. A quick result for a policy $\pitilde\in \Pitilde_c$ is the following.
\begin{Lemma}\label{lm.pi.cts}
Let $\pitilde\in \Pitilde_c$, then
	$\nu\mapsto V^{\pitilde}(x,\nu)$  is continuous for all $x\in[d]$.
\end{Lemma}

\begin{proof}
	Assumption \ref{assum.lipschitz} and $\pitilde\in \Pitilde_c$ together imply that both $\nu\to f^{\pitilde(x,\nu)}(t,x,\nu)$ and $\nu\to P^{\pitilde(x,\nu)}(x, \nu)$ are continuous for any $x\in[d], t\in \T$. Moreover, an induction argument gives that $\nu \mapsto \mu^{\pitilde,\nu}_t$ is continuous for all $t\in \T$. Take an arbitrary $\eps>0$, Assumption \ref{assum.bound} implies that an $T\in \T$ exists with 
	$$ 
	\sum_{t> T} \sup_{(x,\nu)\in \Delta} |f(1+t, x,\nu, u)|\leq \eps.
	$$
	Now take $\nu^n\to \nu^\infty$ in $\Pc([d])$ and $x\in [d]$. For each $n\in \N\cup\{\infty\}$, define 
	$$Q^n_t:= \left(P^{\pitilde \left(y, \mu^{\pitilde,\nu^n}_t \right)}\left( y,\mu^{\pitilde,\nu^n}_t  \right) \right)_{y\in [d]},\; g^n(t, y):= f^{\pitilde\left( y,\mu^{\pitilde,\nu^n}_t \right)}\left( 1+t,y,\mu^{\pitilde,\nu^n}_t \right) \; \forall y\in [d],\quad \forall t\in \T.$$
	Consider the time-inhomogeneous Markov chain $z^{n}_t$ with transition matrix at time $t$ equal to $Q^n_t$ and $z^n_0=x$.
	Then we can conclude that the law of $(z^n_t)_{0\leq t\leq T}$ converges to the law of $(z^\infty_t)_{0\leq t\leq T}$ as $n\to \infty$, and for each $n\in \N\cup\{\infty\}$ that
	$$\E\left[\sum_{0\leq t\leq T} f^{\pitilde\left( s^{\pitilde,\mu^{\pitilde, \nu^n}}_t, \mu^{\pitilde,\nu^n}_t \right)}\left(1+t,s^{\pitilde,\mu^{\pitilde, \nu^n}}_t,\mu^{\pitilde,\nu^n}_t \right) \bigg | s^{\pitilde,\mu^{\pitilde, \nu^n}}_0=x \right]=\E\left[  \sum_{0\leq t\leq T} g^n(t, z^n_t) \right].$$
	As a consequence, 
	$$
	\lim_{n\to\infty}\left| V^\pitilde(x,\nu^n)-V^\pitilde(x,\nu^\infty) \right|\leq 2\eps,
	$$
	and the proof is complete by the arbitrariness of $\eps$. 
\end{proof}

Through Lemma \ref{prop.lln.Nmu} to Lemma \ref{prop.WN}, we consider a sequence $\left( \pitilde^N \right)_{N\in \N\cup\{\infty\}}\subset \Pitilde_c$ such that \eqref{eq.thm.Nequi1} holds. And for each $N\in \N$ we define the replicating $N$-tuple $\pitildebm^N:=(\pitilde^N,\cdots, \pitilde^N)$. 

\begin{Lemma}\label{prop.lln.Nmu}
	For each $t\in \T$, $\mu^{N,\bm\pitilde^N,\nu(N,x)}_t\to\mu^{\pitilde^\infty,\nu}_t$ in probability uniformly over $\Delta$. That is, 
	\be\label{eq.prop.Nmu}  
	\lim_{N\to \infty}\sup_{(x,\nu)\in \Delta}\P\left(\left| \mu^{N,\bm\pitilde^N,\nu(N,x)}_t-\mu^{\pitilde^\infty,\nu}_t \right|\geq \eps\right)=0\quad \forall \eps>0, \quad  \forall t\in \T.
	\ee 
\end{Lemma}

\begin{proof}
	We prove the result by induction on $t$. 
	The result holds for $t=0$ due to \eqref{eq.muNx.converge} and $\mu^{N,\bm\pitilde^N,\nu(N,x)}_0=\nu(N,x)$.
	Now suppose that 
	\be\label{eq.prop.Nmuinduction}  
	\lim_{N\to \infty}\sup_{(x,\nu)\in \Delta}\P\left( \left| \mu^{N,\bm\pitilde^N,\nu(N,x)}_t-\mu^{\pitilde^\infty,\nu}_t \right|\geq \delta \right)=0\quad \forall \delta>0,
	\ee 
	and we prove the result for $t+1$.
	
	Take an arbitrary $(x,\nu)\in \Delta$, for each $N\in \N$,
	let $\alpha^{N,i}_t$ be mutually independent random action values simulated by $\pitilde^N
		\left( s^{N,i, \nu(N,x)}_t,  \mu^{N,\bm\pitilde^N, \nu(N,x)}_t \right)$. 
		Then
	\be\label{eq.prop.Nmu0}  
	\E\left[  e_{s^{N,i}_t}\cdot P\left(\mu^{\pitilde^\infty,\nu}_t,\alpha^{N,i}_t\right)- e_{s^{N,i}_t}  \cdot P^{\pitilde^N\left( \mu^{N,\bm\pitilde^N,\nu(N,x)}_t \right)}\left( \mu^{\pitilde^\infty,\nu}_t \right)  \bigg| s^{N,i}_t, \mu^{N,\bm\pitilde^N, \nu(N,x)}_t \right]=0.
	\ee 
	Meanwhile, a direct calculation gives that 
	\begin{align}
		&\left| \mu^{N,\bm\pitilde^N,\nu(N,x)}_{t+1}- \mu^{\pitilde^\infty,\nu}_{t+1} \right| \notag\\
		&\leq \left| \frac{1}{N}\sum_{i=1}^N  e_{s^{N,i}_t}\cdot P\left(\mu^{N,\bm\pitilde^N, \nu(N,x)}_t,\alpha^{N,i}_t\right)-  \frac{1}{N}\sum_{i=1}^N  e_{s^{N,i}_t}\cdot P\left(\mu^{\pitilde^\infty,\nu}_t,\alpha^{N,i}_t\right)\right| \notag\\
		&\qquad +\left|  \frac{1}{N}\sum_{i=1}^N  e_{s^{N,i}_t}\cdot P\left(\mu^{\pitilde^\infty,\nu}_t,\alpha^{N,i}_t\right) - \mu^{\pitilde^\infty,\nu}_t\cdot P^{\pitilde^\infty\left( \mu^{\pitilde^\infty,\nu}_t \right)} (\mu^{\pitilde^\infty,\nu}_t) \right|. \notag
	\end{align}
	By Assumption \ref{assum.lipschitz},  the first term on the right-hand side of above inequality is bounded by $L_2\big| \mu^{N,\bm\pitilde^N, \nu(N,x)}_t-\mu^{\pitilde^\infty, \nu}_t \big|$,
	 which converges to 0 in probability uniformly over $\Delta$ due to induction hypothesis \eqref{eq.prop.Nmuinduction}. Hence, to prove the desired result, it suffices to prove that
	\be\label{eq.prop.Nmu1}  
	\lim_{N\to\infty} \sup_{(x,\nu)\in \Delta} \P\left( \left|  \frac{1}{N}\sum_{i=1}^N  e_{s^{N,i}_t}\cdot P\left(\mu^{\pitilde^\infty,\nu}_t,\alpha^{N,i}_t\right) - \mu^{\pitilde^\infty,\nu}_t\cdot P^{\pitilde^\infty\left( \mu^{\pitilde^\infty,\nu}_t \right)}  \left( \mu^{\pitilde^\infty,\nu}_t \right) \right| \geq \delta \right)=0 ,\quad \forall \delta>0.
	\ee 
	Notice that
	\begin{align*}
		& \left|  \frac{1}{N}\sum_{i=1}^N  e_{s^{N,i}_t}\cdot P\left(\mu^{\pitilde^\infty,\nu}_t,\alpha^{N,i}_t\right) - \mu^{\pitilde^\infty,\nu}_t\cdot P^{\pitilde^\infty\left( \mu^{\pitilde^\infty,\nu}_t \right)} \left( \mu^{\pitilde^\infty,\nu}_t \right) \right| \notag \\
		&\leq  \left|  \frac{1}{N}\sum_{i=1}^N  e_{s^{N,i}_t}\cdot P\left(\mu^{\pitilde^\infty, \nu}_t,\alpha^{N,i}_t\right)-\left( \frac{1}{N}\sum_{i=1}^N  e_{s^{N,i}_t} \right) \cdot P^{\pitilde^N\left( \mu^{N,\bm\pitilde^N,\nu(N,x)}_t \right)} \left( \mu^{\pitilde^\infty, \nu}_t \right) \right| \quad \text{(I)}\notag \\
		& \quad + \left| \left( \frac{1}{N}\sum_{i=1}^N  e_{s^{N,i}_t} \right) \cdot P^{\pitilde^N\left( \mu^{N,\bm\pitilde^N,\nu(N,x)}_t \right)} \left( \mu^{\pitilde^\infty, \nu}_t \right)  -  \mu^{\pitilde^\infty,\nu}_t \cdot P^{\pitilde^N\left( \mu^{N,\bm\pitilde^N,\nu(N,x)}_t \right)}\left( \mu^{\pitilde^\infty, \nu}_t \right) \right|  \quad \text{(II)} \notag \\
		&\quad + \left|  \mu^{\pitilde^\infty,\nu}_t \cdot P^{\pitilde^N\left( \mu^{N,\bm\pitilde^N,\nu(N,x)}_t \right)}\left( \mu^{\pitilde^\infty, \nu}_t \right) - \mu^{\pitilde^\infty,\nu}_t\cdot P^{\pitilde^\infty\left( \mu^{N,\bm\pitilde^N,\nu(N,x)}_t \right)}\left( \mu^{\pitilde^\infty, \nu}_t \right)  \right| \quad \text{(III)} \notag \\
		&\quad +  \left|  \mu^{\pitilde^\infty,\nu}_t \cdot P^{\pitilde^\infty\left( \mu^{N,\bm\pitilde^N,\nu(N,x)}_t \right)}\left( \mu^{\pitilde^\infty, \nu}_t \right) - \mu^{\pitilde^\infty,\nu}_t\cdot P^{\pitilde^\infty(\mu^{\pitilde^\infty,\nu})}\left( \mu^{\pitilde^\infty, \nu}_t \right)  \right| \quad \text{(IV)}.
	\end{align*}
	And we prove the uniform convergence on $\Delta$ in probability for each of terms (I)--(IV) above.
	
	As for term (I), by Markov's inequality,
	\begin{align}
		&\P\left(\left|   \frac{1}{N}\sum_{i=1}^N  e_{s^{N,i}_t}\cdot P\left(\mu^{\pitilde^\infty,\nu}_t,\alpha^{N,i}_t\right)-\left( \frac{1}{N}\sum_{i=1}^N  e_{s^{N,i}_t} \right) \cdot P^{\pitilde^N\left( \mu^{N,\bm\pitilde^N,\nu(N,x)}_t \right)}\left( \mu^{\pitilde^\infty,\nu}_t \right)  \right| \geq \delta \right) \notag\\
		&\leq \E\left[\left(  \frac{1}{N}\sum_{i=1}^N   \left( e_{s^{N,i}_t}\cdot P\left(\mu^{\pitilde^\infty}_t,\alpha^{N,i}_t\right)- e_{s^{N,i}_t}  \cdot P^{\pitilde^N\left( \mu^{N,\bm\pitilde^N,\nu(N,x)}_t \right)}\left( \mu^{\pitilde^\infty,\nu}_t \right) \right) \right)^2  \right]\Bigg/\delta^2  \notag \\
		&=\frac{1}{N^2 \delta^2}\sum_{1\leq i\leq N} \E\left[ \left( e_{s^{N,i}_t}\cdot P\left(\mu^{\pitilde^\infty,\nu}_t,\alpha^{N,i}_t\right)- e_{s^{N,i}_t}  \cdot P^{\pitilde^N\left( \mu^{N,\bm\pitilde^N,\nu(N,x)}_t \right)}\left( \mu^{\pitilde^\infty,\nu}_t \right) \right)^2 \right] \label{eq.lm.lln3} \\
		&\leq \frac{4}{N \delta^2}, \notag
	\end{align}
	where the third line is due to the independence of $\alpha^{N,i}_t$ and \eqref{eq.prop.Nmu0}, 
	 and the last line follows from that $\sup_{(x,\mu)\in \Delta, u\in U}|e_x\cdot P(\mu,u)|\leq 1$.
	Notice that the last line in \eqref{eq.lm.lln3} is independent of $(x,\nu)$, so we have, for any $\delta>0$ that,
	$$
	\sup_{(x,\nu)\in \Delta}  \P\left(\left|   \frac{1}{N}\sum_{i=1}^N  e_{s^{N,i}_t}\cdot P\left(\mu^{\pitilde^\infty,\nu}_t,\alpha^{N,i}_t\right)-\left( \frac{1}{N}\sum_{i=1}^N  e_{s^{N,i}_t} \right) \cdot P^{\pitilde^N\left( \mu^{N,\bm\pitilde^N,\nu(N,x)}_t \right)}\left( \mu^{\pitilde^\infty,\nu}_t \right)  \right| \geq \delta \right)  \to 0.
	$$
	Notice that term (II) is bounded by $\sqrt{d}\left(\sup_{(x,\nu)\in \Delta}\left| \mu^{N,\bm\pitilde^N,\nu(N,x)}_t- \mu^{\pitilde^\infty,\nu}\right|\right)$, hence converges in probability to 0 uniformly over $\Delta$. 
	Term (III) is bounded by 
	$$
	\begin{aligned}
		&\sqrt{d}\sup_{(x,\nu)\in \Delta}  \bigg|  \int_U P\left(x, \mu^{\pitilde^\infty, \nu}_t, u \right) \pitilde^N\left( x, \mu^{N,\bm\pitilde^N,\nu(N,x)}_t \right)(du)\\
		&\qquad \quad\;\; -  \int_U P\left(x, \mu^{\pitilde^\infty, \nu}_t, u \right)\pitilde^\infty\left( x, \mu^{N,\bm\pitilde^N,\nu(N,x)}_t \right)(du)   \bigg|\\
		&\leq \sqrt{d}\sup_{(x,\nu)\in \Delta} \sum_{y\in[d]} \bigg|  \int_U P\left(x, \mu^{\pitilde^\infty, \nu}_t, y, u \right) \left[\pitilde^N\left( x, \mu^{N,\bm\pitilde^N,\nu(N,x)}_t \right)(du)-\pitilde^\infty\left( x, \mu^{N,\bm\pitilde^N,\nu(N,x)}_t \right)(du)\right]   \bigg|\\
		&\leq \sqrt{d^3} L_1 \sup_{(x,\nu)\in \Delta} \Wc_1 \left( \pitilde^N\left( x, \mu^{N,\bm\pitilde^N,\nu(N,x)}_t \right), \pitilde^\infty\left( x, \mu^{N,\bm\pitilde^N,\nu(N,x)}_t \right) \right)\\
		&\leq \sqrt{d^3} L_1 \sup_{(y,\mu)\in \Delta}\Wc_1 \left( \pitilde^N(y, \mu), \pitilde^\infty(y,\mu) \right)
		\to 0\quad \text{as }N\to\infty,
	\end{aligned}
	$$
	where the first inequality follows from \eqref{eq.cha.W1} and \eqref{eq.assum3.0}. Therefore, term (III) also converges (in probability) to 0 uniformly over $\Delta$.
	As for term (IV), notice that 
	\begin{align}
		& \sup_{(x,\nu)\in\Delta}\left|  \mu^{\pitilde^\infty,\nu}_t \cdot P^{\pitilde^\infty\left( \mu^{N,\bm\pitilde^N,\nu(N,x)}_t \right)}\left( \mu^{\pitilde^\infty, \nu}_t \right) - \mu^{\pitilde^\infty,\nu}_t\cdot P^{\pitilde^\infty(\mu^{\pitilde^\infty,\nu})}\left( \mu^{\pitilde^\infty, \nu}_t \right)  \right| \label{eq.prop.Nmu5} \\
		&\leq \sqrt{d} \sup_{(x,\nu)\in\Delta}\left|  \int_U P\left(x, \mu^{\pitilde^\infty, \nu}_t, u \right) \pitilde^\infty\left( x, \mu^{N,\bm\pitilde^N,\nu(N,x)}_t \right)(du)-  \int_U P\left(x, \mu^{\pitilde^\infty, \nu}_t, u \right)\pitilde^\infty(x, \mu^{\pitilde^\infty,\nu}) (du)   \right|. \notag
	\end{align}
	Since $\pitilde^\infty\in \Pitilde_c$ and $\Pc([d])$ is compact, for each $x,y\in [d]$, the function
	$$\mu\mapsto \int_U P\left(x, \mu^{\pitilde^\infty, \nu}_t, y, u \right)\pitilde^\infty(x, \mu) (du),$$
	which maps from $\Pc([d])$ to $\R$, is continuous, hence is uniformly continuous. Then for any $\delta>0$, there exists $\delta_0>0$ such that the right-hand side of \eqref{eq.prop.Nmu5} less than $\delta$ whenever $\left| \mu^{N,\bm\pitilde^N,\nu(N,x)}_t- \mu^{\pitilde^\infty,\nu}\right|$ $<$ $\delta_0$. By induction assumption \eqref{eq.prop.Nmuinduction}, for such $\delta_0$ there exists $N_0$ such that 
	$$\sup_{(x,\nu)\in \Delta} \P\left(\left| \mu^{N,\bm\pitilde^N,\nu(N,x)}_t- \mu^{\pitilde^\infty,\nu}\right| \geq \delta_0\right)\leq \eps\quad \forall N\geq N_0.$$
	As a consequence, we have for all $N\geq N_0$ that 
	$$
	\sup_{(x,\nu)\in \Delta} \P\left( \left|  \mu^{\pitilde^\infty,\nu}_t \cdot P^{\pitilde^\infty\left( \mu^{N,\bm\pitilde^N,\nu(N,x)}_t \right)}\left( \mu^{\pitilde^\infty, \nu}_t \right) - \mu^{\pitilde^\infty,\nu}_t\cdot P^{\pitilde^\infty(\mu^{\pitilde^\infty,\nu})}\left( \mu^{\pitilde^\infty, \nu}_t \right)  \right|  \geq \delta \right)\leq \eps.
	$$
	That is, term (IV) converges in probability to 0 uniformly over $\Delta$. In sum, \eqref{eq.prop.Nmu1} is verified.
\end{proof}

\begin{Lemma}\label{prop.VN} The auxiliary functions defined in \eqref{eq.V.relax} and \eqref{eq.def.VNagentpitilde} satisfy that
	\bee
	\lim_{N\to \infty}\sup_{(x,\nu)\in \Delta} \left| V^{N,\pitilde^N} \left(x,\nu(N,x)\right)-V^{{\pitilde^\infty}}(x,\nu) \right|=0.
	\eee 
\end{Lemma}

\begin{proof}
	By Assumption \ref{assum.bound}, for any $\eps$, there exists $T$ such that 
	\be\label{eq.prop.VN0} 
	\sum_{t>T} \sup_{u\in U, (x,\mu)\in \Delta} |f(t,x,\mu,u)|\leq \eps.
	\ee 
	By applying Lemma \ref{prop.lln.Nmu}, we have that 
	\be\label{eq.prop.VNmu} 
	\lim_{N\to\infty} \sup_{0\leq t\leq T, (x,\nu)\in \Delta} \P\left(\left|\mu^{N,\bm\pitilde^N,\nu(N,x)}_t- \mu^{{\pitilde^\infty},\nu}_t\right|>\delta\right) =0\quad \forall \delta>0.
	\ee 
	
	We first prove for each $y\in[d]$ that
	\be\label{eq.prop.VNP}  
	\begin{aligned}
	&\lim_{N\to\infty}\sup_{0\leq t\leq T, (x,\nu)\in \Delta } \E\left[\Bigg| P^{\pitilde^N\left(y,  \mu^{N,\bm\pitilde^N,\nu(N,x)}_t\right)}\left( y,\mu^{N,\bm\pitilde^N,\nu(N,x)}_t \right)
- P^{{\pitilde^\infty} \left(y, \mu^{{\pitilde^\infty},\nu}_t \right)}\left( y,\mu^{{\pitilde^\infty},\nu}_t \right) \Bigg| \right] =  0.
	\end{aligned}
	\ee 
	Notice for each $y\in[d]$ that
	\begin{align*}
		&\sup_{0\leq t\leq T, (x,\nu)\in \Delta }\E\left[\left| P^{\pitilde^N\left(y, \mu^{N,\bm\pitilde^N,\nu(N,x)}_t\right)}\left( y,\mu^{N,\bm\pitilde^N,\nu(N,x)}_t \right)- P^{{\pitilde^\infty} \left( y, \mu^{{\pitilde^\infty},\nu}_t \right) }\left( y ,\mu^{{\pitilde^\infty},\nu}_t \right)\right|\right] \\
		&\leq \sup_{0\leq t\leq T, (x,\nu)\in \Delta}\E\left[\left| P^{\pitilde^N\left( y, \mu^{N,\bm\pitilde^N,\nu(N,x)}_t\right)}\left( y,\mu^{N,\bm\pitilde^N,\nu(N,x)}_t \right)- P^{\pitilde^N\left( y, \mu^{N,\bm\pitilde^N,\nu(N,x)}_t\right)}\left( y, \mu^{{\pitilde^\infty},\nu}_t \right) \right| \right] (\text{I'})\\
		&\qquad +\sup_{0\leq t\leq T, (x,\nu)\in \Delta}\E\left[\left| P^{\pitilde^N\left( y, \mu^{N,\bm\pitilde^N,\nu(N,x)}_t\right)}\left( y, \mu^{{\pitilde^\infty},\nu}_t \right)- P^{\pitilde^\infty\left( y, \mu^{N,\bm\pitilde^N,\nu(N,x)}_t\right)}\left( y, \mu^{{\pitilde^\infty},\nu}_t \right) \right| \right] \; (\text{II'})\\
		&\qquad + \sup_{0\leq t\leq T, (x,\nu)\in \Delta}\E\left[\left| P^{\pitilde^\infty\left( y, \mu^{N,\bm\pitilde^N,\nu(N,x)}_t\right)}\left( y, \mu^{{\pitilde^\infty},\nu}_t \right) - P^{\pitilde^\infty\left( y, \mu^{{\pitilde^\infty},\nu}_t \right)}\left( y, \mu^{{\pitilde^\infty},\nu}_t \right)  \right| \right] \; (\text{III'}).
	\end{align*}
	By \eqref{eq.assum3.1}, the difference in the expectation of term (I') above is bounded by $$L_2\sup_{0\leq t\leq T, (x,\nu)\in \Delta} \bigg|\mu^{N,\bm\pitilde^N,\nu(N,x)}_t - \mu^{{\pitilde^\infty},\nu}_t \bigg|,$$ 
	then \eqref{eq.prop.VNmu} implies that term (I') converges to 0 as $N\to\infty$. As for term (II') (resp. term (III')) above, a similar argument as for term (III) (resp. term (IV)) in the proof of Lemma \ref{prop.lln.Nmu} shows the difference inside the expectation converges in probability to 0 uniformly on $\Delta$ and $0\leq t\leq T$ as $N\to \infty$. Therefore, both term (II') and term (III') converge to 0 as $N\to\infty$, and \eqref{eq.prop.VNP} holds.

{\color{black}Next, 
for each $(x,\nu)\in \Delta$ and $N\geq 2$,
we consider the dynamics of $\left( s^{N,1}_t \right)_{0\leq t\leq T}$ under policy $\pitilde^N$ and flow $\mu^{\bm \pitilde^N, \nu(N,x)}$ and the dynamic of an auxiliary Markov chain $\left( z^{x,\nu}_t\right)_{0\leq t\leq T}$ with initial values $s^{N,1}_0=x=z^{x,\nu}_0$, where $\left( z^{x,\nu}_t\right)_{0\leq t\leq T}$ is guided by the deterministic transition matrices 
$$
Q_t:= \bigg(P^{\pitilde^\infty\left(y,\mu^{\pitilde^\infty,\nu}_t  \right)}\left( y,\mu^{\pitilde^\infty,\nu}_t,z \right)\bigg)_{y,z\in[d]} \text{ for $0\leq t\leq T-1$}. 
$$
We show by induction that
\be\label{eq:prop.VNnew}  
\left( e_{s^{N,1}_{t}}, \mu^{\bm \pitilde^N, \nu(N,x)}_t \right) \to \left(e_{z^{x,\nu}_{t}},  \mu^{\pitilde^\infty,\nu}_t\right) \text{ in distribution uniformly over $(x,\nu)\in \Delta$}, \quad \forall 0\leq t\leq T.
\ee 
By Lemma \ref{prop.lln.Nmu} with $t=0$ and $e_{s^{N,1}_0}=e_x= e_{z^{x,\nu}_0}$, \eqref{eq:prop.VNnew} holds for $t=0$. Now suppose \eqref{eq:prop.VNnew} holds for $t$ and we prove the case for $t+1$. By Lemma \ref{prop.lln.Nmu} again, $\mu^{\bm \pitilde^N, \nu(N,x)}_{t+1}$ converges to $\mu^{\pitilde^\infty,\nu}_{t+1}$ in probability uniformly over $(x,\nu)\in \Delta$ and $\mu^{\pitilde^\infty,\nu}_{t+1}$ is deterministic. Then to prove the joint distribution of $\left( e_{s^{N,1}_{t}}, \mu^{\bm \pitilde^N, \nu(N,x)}_t \right)$ converges to the joint distribution of $\left(e_{z^{x,\nu}_{t}},  \mu^{\pitilde^\infty,\nu}_t\right)$ uniformly over $(x,\nu)\in \Delta$, it suffices to show that 
\be\label{eq:prop.VN0new}   
e_{s^{N,1}_{t+1}} \to e_{z^{x,\nu}_{t+1}} \text{ in distribution uniformly over $(x,\nu)\in \Delta$}.
\ee 
Notice that 
\bee 
e_{s^{N,1}_{t+1}}= e_{s^{N,1}_{t}}\cdot P^{\pitilde^{N}\left( \mu^{N,\pitildebm^N, \nu(N,x)}_t \right)}\left( \mu^{N,\pitildebm^N, \nu(N,x)}_t \right)\text{ in distribution for all $(x,\nu)\in \Delta$}. 
\eee
By Lemma \ref{prop.lln.Nmu}, $\mu^{\bm \pitilde^N, \nu(N,x)}_{t}$ converges to $\mu^{\pitilde^\infty,\nu}_{t}$ in probability uniformly over $(x,\nu)\in \Delta$. This together with the finiteness of $[d]$, Assumption \ref{assum.lipschitz} and the induction hypothesis implies that 
\begin{align*}
&e_{s^{N,1}_{t}}\cdot P^{\pitilde^{N}\left( \mu^{N,\pitildebm^N, \nu(N,x)}_t \right)}\left( \mu^{N,\pitildebm^N, \nu(N,x)}_t \right) \\
&\to e_{z^{x,\nu}_{t}}\cdot P^{\pitilde^\infty\left( \mu^{\pitilde^\infty,\nu}_t \right)}( \mu^{\pitilde^\infty,\nu}_t) = e_{z^{x,\nu}_{t+1}}\text{ in distribution uniformly over $(x,\nu)\in \Delta$},
\end{align*} 
which gives \eqref{eq:prop.VN0new}. In sum, \eqref{eq:prop.VNnew} holds.

Now we prove the desired result. We have for each $0\leq t\leq T$ that
\be\label{eq.prop.VNf0} 
\E\left[f^{{\pitilde^\infty}\left(s^{{\pitilde^\infty}, \nu}_{t}, \mu^{{\pitilde^\infty},\nu}_t \right)}\left(1+t, s^{{\pitilde^\infty}, \nu}_{t},\mu^{{\pitilde^\infty},\nu}_t \right) \right] = \E\left[  f^{{\pitilde^\infty}\left(z^{x,\nu}_{t}, \mu^{{\pitilde^\infty},\nu}_t \right)}\left(1+t, z^{x,\nu}_{t},\mu^{{\pitilde^\infty},\nu}_t \right)  \right]
\ee 
By combining \eqref{eq.thm.Nequi1}, \eqref{eq:prop.VNnew} and Assumption \ref{assum.lipschitz}, we have for each $0\leq t\leq T$ that
			\be\label{eq.prop.VNf}  
		\begin{aligned}
			\lim_{N\to\infty}\sup_{(x,\nu)\in \Delta}	\Bigg| &\E\left[ f^{\pitilde^N\left(s^{N,1}_t,  \mu^{N,\bm\pitilde^N,\nu(N,x)}_t\right)}\left(1+t, s^{N,1}_t,\mu^{N,\bm\pitilde^N,\nu(N,x)}_t \right)\right]\\
			& \; - \E\left[f^{{\pitilde^\infty}\left(z^{x,\nu}_{t}, \mu^{{\pitilde^\infty},\nu}_t \right)}\left(1+t, z^{x,\nu}_{t},\mu^{{\pitilde^\infty},\nu}_t \right) \right] \Bigg| = 0.
		\end{aligned}
		\ee 
	Then \eqref{eq.prop.VN0}, \eqref{eq.prop.VNf0} and \eqref{eq.prop.VNf} together imply that
	$$
	\begin{aligned}
		&\limsup_{N\to \infty}\sup_{(x,\nu)\in \Delta} \left| V^{N,\pitilde^N} \left(x,\nu(N,x)\right)-V^{{\pitilde^\infty}}(x,\nu) \right|\leq 2\eps,
	\end{aligned}
	$$
	and the proof is completed by the arbitrariness of $\eps$.}
\end{proof}

{\color{black}
\begin{Remark}\label{rm:new}
By comparing \eqref{eq.def.VNagentpitilde} with \eqref{eq.def.JNagentpitilde} and comparing \eqref{eq.V.relax} with the notation $J^\pitilde(x,\nu)$ introduced right after \eqref{eq.mu.population}, 
we see that following the same argument in the proof of Lemma \ref{prop.VN}, but replacing $1+t$ with $t$ in \eqref{eq.prop.VNf0} and \eqref{eq.prop.VNf}, yields
\be\label{eq:rm.new}  
	\lim_{N\to \infty}\sup_{(x,\nu)\in \Delta} \left| J^{N,1,\pitildebm^N} \left(x,\nu(N,x)\right)-J^{{\pitilde^\infty}}(x,\nu) \right|=0.
\ee 
Notice that $\pitildebm^N$ here still denotes $(\pitilde^N,\cdots, \pitilde^N)$.
\end{Remark}
}
	
\begin{Lemma}\label{prop.WN} The auxiliary functions defined in \eqref{eq.V.relax} and \eqref{eq.def.WN} satisfies that
	\bee
	\lim_{N\to \infty}\sup_{(x,\nu)\in \Delta, y\in [d], u\in U} \left| W^{N, \pitilde^N} \left(x,\nu(N,x),y,u \right)- V^{\pitilde^\infty}\left(y,\nu\cdot P^{\pitilde^\infty(\nu)}(\nu) \right) \right| =0.
	\eee
\end{Lemma}

\begin{proof}
	Since $[d]$ is finite, it suffices to prove for each fixed $y\in[d]$ that
	$$
	\lim_{N\to \infty}\sup_{(x,\nu)\in \Delta, u\in U} \left| W^{N, \pitilde^N}(x,\nu(N,x),y,u)- V^{\pi^\infty}\left(y,\nu\cdot P^{\pitilde^\infty(\nu)}(\nu) \right) \right| =0.
	$$ 
	Fix $y\in[d]$. 
{\color{black}	Recall \eqref{eq.def.sNnu}, the random variable $Y$ defined in \eqref{eq:YnuxN} and that $\pitildebm^N=(\pitilde^N,\cdots, \pitilde^N)$. Then for any $N\geq 2$, $(x,\nu)\in \Delta$ and $u\in U$, we have that 
$$
{\color{black} Y^{N,\pitilde^N}(x, \nu(N,x),u) =}	{\color{black} \P\left( \mu^{N,\pitildebm^N,\nu(N,x)} \Big| s^{N,1}_0=x, \alpha^{N,1}_0=u \right),} 
$$
which is in distribution equal to
	\begin{align*}
	& \frac{1}{N}\left[ e_x\cdot P(\nu(N,x), u) + \sum_{i=2}^N  e_{\bar{s}^{N,i}(x, \nu(N,x))}\cdot P\left( \nu(N,x),\alpha^{N,i, \nu(N,x)} \right) \right]\\
	&=\frac{1}{N}\left[  \sum_{i=1}^N  e_{\bar{s}^{N,i}(x, \nu(N,x))}\cdot P\left( \nu(N,x),\alpha^{N,i, \nu(N,x)} \right) \right]\\
	&\quad +\frac{1}{N}\left[ e_x\cdot P(\nu(N,x), u) - e_{\bar{s}^{N,1}(x, \nu(N,x))}\cdot P\left( \nu(N,x),\alpha^{N,1, \nu(N,x)}\right) \right],
	\end{align*}
	 where $(\alpha^{N,i, \nu(N,x)})_{1\leq i\leq N}$ are mutually independent random variables such that
	 $$\alpha^{N,i, \nu(N,x)}\sim \pitilde^N\left( \bar{s}^{N,i}\left(x, \nu(N,x) \right), \nu(N,x) \right) \text{ for } 1\leq i\leq N.$$
	By Dominated Convergence Theorem, 
	$$\sup_{u\in U, (x,\nu)\in \Delta} \frac{1}{N}\left| e_x\cdot P(\nu(N,x), u) - e_{\bar{s}^{N,1}_0}\cdot P\left( \nu(N,x),\alpha^{N,1, \nu(N,x)}\right) \right|\to 0, \; \text{as $N\to\infty$}.$$ 
	By applying Lemma \ref{prop.lln.Nmu} with $t=1$ in \eqref{eq.prop.Nmu}, we reach, uniformly for $(x,\nu)\in \Delta$ that 
	\begin{align*}
		\left|\frac{1}{N}\sum_{i=1}^N  e_{\bar{s}^{N,i}_0}\cdot P\left(\nu(N,x),\alpha^{N,i,\nu(N,x)} \right) -\nu\cdot P^{{\pitilde^\infty}(\nu)}(\nu)\right|
		 \to 0\quad \text{in probability,  as $N\to\infty$}.
	\end{align*}}
	As a consequence, 
	\be\label{eq.YN.convergence}  
\lim_{N\to\infty} \sup_{ (x,\nu)\in\Delta, u\in U} \P\left( \left| {\color{black} Y^{N,\pitilde^N}(x, \nu(N,x),u)} -\nu\cdot P^{{\pitilde^\infty}(\nu)}(\nu)\right|\geq \delta\right) =0 \quad \forall \delta>0.
	\ee 
Meanwhile, \eqref{eq.def.WN} tells that 
	\begin{align}
		&W^{N,\pitilde^N}\left( x,\nu(N,x),y,u\right) \notag
		 = \E\left[ V^{N,\pitilde^N}\left(y,{\color{black} Y^{N,\pitilde^N}(x, \nu(N,x),u)} \right)  \right].
	\end{align}
	Notice that
	\begin{align}
		&\sup_{(x,\nu)\in \Delta, u\in U}\left|\E\left[V^{N,\pitilde^N}\left( y, {\color{black} Y^{N,\pitilde^N}(x, \nu(N,x),u)} \right)  \right]- V^{\pitilde^\infty} \left( y,\nu\cdot P^{{\pitilde^\infty}(\nu)}(\nu) \right) \right| \notag\\
		& \leq \sup_{(x,\nu)\in \Delta, u\in U}\left|\E\left[V^{N,\pitilde^N}\left( y, {\color{black} Y^{N,\pitilde^N}(x, \nu(N,x),u)} \right)   \right]- \E\left[V^{{\pitilde^\infty}}\left( y,{\color{black} Y^{N,\pitilde^N}(x, \nu(N,x),u)} \right) \right]\right| \label{eq.prop.WN1}\\
		& \qquad + \sup_{(x,\nu)\in \Delta, u\in U}\left|\E\left[V^{{\pitilde^\infty}}\left( y,{\color{black} Y^{N,\pitilde^N}(x, \nu(N,x),u)} \right) \right]-V^{\pitilde^\infty} \left( y,\nu\cdot P^{{\pitilde^\infty}(\nu)}(\nu) \right) \right|. \notag
	\end{align}

{\color{black}Since $\Delta_N\subset \Delta$ for all $N\geq 2$, Lemma \ref{prop.VN} implies that
\be\label{eq:lm5.4.new}
\lim_{N\to \infty}\sup_{(\hat x, \hat \nu)\in \Delta_N} \left| V^{N,\pitilde^N} \left(\hat x, \hat \nu \right)-V^{{\pitilde^\infty}}(\hat x, \hat \nu) \right|=0.
\ee 	
Notice that the pair $( y,$ $Y^{N,\pitilde}(x,\nu,u) )$ belongs to $\Delta_N$ for all $N\geq 2$. Then
	\begin{align*}
		&\sup_{(x,\nu)\in \Delta, u\in U}\left|\E\left[V^{N,\pitilde^N}\left( y, {\color{black} Y^{N,\pitilde^N}(x, \nu(N,x),u)} \right)   \right]- \E\left[V^{{\pitilde^\infty}}\left( y,{\color{black} Y^{N,\pitilde^N}(x, \nu(N,x),u)} \right) \right]\right| \\
		&\leq  \sup_{(x,\nu)\in \Delta, u\in U}\E\left[\left|V^{N,\pitilde^N}\left( y, {\color{black} Y^{N,\pitilde^N}(x, \nu(N,x),u)} \right) - V^{{\pitilde^\infty}}\left( y,{\color{black} Y^{N,\pitilde^N}(x, \nu(N,x),u)} \right) \right|   \right]\\
		&\leq \sup_{(\hat x, \hat \nu)\in \Delta_N} \left| V^{N,\pitilde^N} \left( \hat x, \hat \nu \right)-V^{{\pitilde^\infty}}( \hat x, \hat \nu) \right|\to 0, \quad \text{as $N\to\infty$},
	\end{align*}
	where the last line follows from \eqref{eq:lm5.4.new}. That is, the first term on the right-hand side of \eqref{eq.prop.WN1} tends to zero, as $N$ goes to infinity.}
	
	As for the second term on the right-hand side of \eqref{eq.prop.WN1}, Lemma \ref{lm.pi.cts} together with the compactness of $\Pc([d])$ implies that $\nu\to V(x,\nu)$ is uniformly continuous for each $x\in[d]$. Thus, for any $\eps$, we can find $\delta>0$ such that 
	$$
	\sup_{x\in [d]}\left| V^{\pitilde^\infty}(x, \mu)-V^{\pitilde^\infty}(x,\bar\mu)\right| <\eps \quad \text{whenever $|\mu-\bar\mu|<\delta$}.
	$$
	Then by \eqref{eq.YN.convergence}, there exists $N_0$ such that 
	$$ \sup_{ \nu\in \Pc([d]), u\in U} \P\left( \left| {\color{black} Y^{N,\pitilde^N}(x, \nu(N,x),u)} -\nu\cdot P^{{\pitilde^\infty}(\nu)}(\nu)\right|\geq \delta\right) <\eps \quad \forall N\geq N_0.$$ 
	As a consequence,
	$$
	\begin{aligned}
		& \sup_{(x,\nu)\in \Delta, u\in U}\left|\E\left[V^{\pitilde^\infty}\left( y,{\color{black} Y^{N,\pitilde^N}(x, \nu(N,x),u)} \right) \right]-V^{\pitilde^\infty} \left( y,\nu\cdot P^{{\pitilde^\infty}(\nu)}(\nu) \right) \right|\\
		& \leq \sup_{(x,\nu)\in \Delta, u\in U}\eps\P\left( \left| {\color{black} Y^{N,\pitilde^N}(x, \nu(N,x),u)} -\nu\cdot P^{{\pitilde^\infty}(\nu)}(\nu)\right|<\delta  \right)\\
		&\qquad + 2\sup_{(x,\nu)\in \Delta, u\in U} |V^{\pi^\infty}(y,\mu)|\cdot \P\left( \left| {\color{black} Y^{N,\pitilde^N}(x, \nu(N,x),u)}-\nu\cdot P^{{\pitilde^\infty}(\nu)}(\nu)\right|\geq \delta  \right)\\
		&<\left( 1+ 2\sup_{(x,\nu)\in \Delta, u\in U} |V^{\pi^\infty}(y,\mu)|\right)\eps\quad \forall N\geq N_0,
	\end{aligned}
	$$
	which shows the convergence to zero of the second term on the right-hand side of \eqref{eq.prop.WN1}, and the desired result is proved.
\end{proof}

Since $(x,\nu)=(x,\nu(N,x))$ for $(x,\nu)\in\Delta_N$ and $\Delta_N\subset \Delta$ for each $N\geq 2$, 
Lemma \ref{prop.WN} implies the following.
\begin{Lemma}\label{cor.WN}
Take an arbitrary $\pitilde \in \Pitilde_c$, 
then 
	\bee\label{eq.prop.WN'}   
	\lim_{N\to \infty}\sup_{(x,\nu)\in \Delta_N, y\in [d], u\in U} \left| W^{N, \pitilde} \left(x,\nu,y,u \right)- V^\pitilde\left(y,\nu\cdot P^{\pitilde(\nu)}(\nu) \right) \right| =0.
	\eee
\end{Lemma}

\subsection{Proof of Theorem \ref{thm.MFG.Neps}}\label{subsubsec:proof.Neps}
\begin{proof}
	{\color{black}It suffices to prove the result for the representative agent 1.} 
	Suppose $\pitilde\in \Pitilde$ is a consistent MFG equilibrium in Definition \ref{def:sophisticated.ne}. For each $N\geq 2$, define the $N$-tuple $\pitildebm^N:= (\pitilde,\cdots,\pitilde)$. 
	Now fix an arbitrary $\eps>0$. By Lemma \ref{cor.WN}, there exits $N_0$ such that for each $N\geq N_0$,
	\begin{align}\label{eq.prop.1}
		\sup_{\varpi\in \Pc(U), (x,\nu)\in \Delta_N}&\int_U \Bigg| \sum_{y\in[d]} P(x,\nu, y, u) W^{N,\pitilde}(x,\nu,y, u) \notag\\
		&\qquad - \sum_{y\in[d]}P(x,\nu,y,u) V^{\pitilde}(y, \nu\cdot P^{\pitilde}(\nu)) \Bigg|\varpi(du) \leq \frac\eps 2.
	\end{align}
	By taking difference between \eqref{eq.mfgV.pivarpi} and \eqref{eq.N.Jpasting}, the first terms for step 0 in these two formulas are canceled. This together with \eqref{eq.prop.1} yields, for all $N\geq N_0$ that
	\begin{align}\label{eq.prop.3}
		&\sup_{\varpi\in \Pc(U), (x,\nu)\in \Delta_N}\left| J^{N,1,(\pitildebm^N_{-1}, \varpi)\otimes_{1}\pitildebm^N}(x, \nu) -J^{\varpi\otimes_1 {\pitilde}} \left(x,\nu, \mu^{\pitilde, \nu}_{[1:\infty)} \right) \right| 
		\leq \frac\eps 2.
	\end{align}
	Similarly, by taking the difference between $J^{N,1,\pitildebm^N}(x, \nu)$ and $J^{{\pitilde}} \left(x,\nu, \mu^{\pitilde,\nu}_{[1:\infty)} \right)$, the first terms for step 0 are also canceled and thus
	
	\begin{align}\label{eq.prop.3'}
		&\sup_{(x,\nu)\in \Delta_N}\left| J^{N,1,\pitildebm^N}(x, \nu) -J^{{\pitilde}} \left(x,\nu, \mu^{\pitilde,\nu}_{[1:\infty)}  \right)\right| 
		\notag \\
		= &\sup_{(x,\nu)\in \Delta_N}\int_U \Bigg| \sum_{y\in[d]} P(x,\nu, y, u) W^{N,\pitilde}(x,\nu,y, u) \\
		&\qquad\qquad \qquad \qquad \qquad - \sum_{y\in[d]}P(x,\nu, y, u) V^{\pitilde}(y, \nu\cdot P^{\pitilde}(\nu)) \Bigg|\pitilde(x,\nu)(du) 
		\leq \frac\eps 2\quad \forall N\geq N_0. \notag 
	\end{align}
	Since $\pitilde$ is a consistent equilibrium in the MFG, we have 
	\begin{align}\label{eq.prop.5}
		\sup_{\varpi\in \Pc(U)}J^{\varpi\otimes_1 {\pitilde}} \left(x,\nu, \mu^{\pitilde, \nu}_{[1:\infty)}  \right) =J^{{\pitilde}} \left(x,\nu, \mu^{\pitilde, \nu}_{[1:\infty)}  \right)=J^{{\pitilde}} \left(x,\nu \right).
	\end{align}
	Hence, for each $N\geq N_0$, we have uniformly on $(x,\nu)\in \Delta_N$ that,
	\begin{align*}
		J^{N,1,\pitildebm^N}(x, \nu) \geq &J^{\pitilde}\left(x,\nu \right)-\frac\eps2\\
		\geq & \sup_{\varpi\in \Pc(U)}J^{\varpi\otimes_1 {\pitilde }}\left(x,\nu,  \mu^{\pitilde,\nu}_{[1:\infty)}  \right)-\frac\eps2
		\geq 
		\sup_{\varpi\in \Pc(U)}J^{N,1,(\pitildebm^N_{-1}, \varpi)\otimes_{1}\pitildebm^N}(x, \nu)  -\eps,
	\end{align*}
	where the first inequality follows from \eqref{eq.prop.3'}, the second inequality follows from \eqref{eq.prop.5}, the third inequality follows from \eqref{eq.prop.3}.

\end{proof}

\subsection{Proof of Theorem \ref{thm.Nequi.converge}}\label{subsubsec:proof.Nconverge}
\begin{proof}
	Now fix an arbitrary $(x,\nu)\in \Delta$. Since for each $N\in \N$, $\pitildebm^{N}=(\pitilde^{N},\cdots \pitilde^{N})$ is a consistent equilibrium in Definition \ref{def:N.sophisticated} for the $N$-agent game, we have that
	\be\label{eq.thm.converge}
	\begin{aligned}
		&J^{{N},1,(\pitildebm^{{N}}_{-1}, \varpi)\otimes_{1}\pitildebm^{{N}}}(x,\nu(N,x)) 
		 \leq  J^{{N},1,\pitildebm^{N}}(x, \nu(N,x)),\quad \forall \varpi\in \Pc(U).
	\end{aligned} 
	\ee
	By Lemma \ref{prop.WN}, 
	\be\label{eq.thm.converge1}
	\lim_{N\to\infty}\sup_{(x,\nu)\in \Delta, y\in[d], u\in U}\left| W^{{N}, \pitilde^{{N}}}\left( x,\nu(N,x), y,u \right)-V^{\pitilde^\infty}\left(y,\nu\cdot P^{\pitilde^\infty(\nu)}(\nu) \right)\right| =0.
	\ee
{\color{black}By \eqref{eq.N.Jpasting}, 
	\be\label{eq:thm1}
	\begin{aligned}
		&J^{N,1,\left( \pitildebm^N_{-1}, \varpi \right)\otimes_{1}\pitildebm^N}(x, \nu(N,x)) \\
		&= \int_U \left[ f(0,x,\nu(N,x), u)+\sum_{y\in[d]} P(x,\nu(N,x), y, u)
		W^{N, \pitilde}(x,\nu(N,x),y,u)\right] \varpi(du). 
	\end{aligned} 
	\ee
	Then by taking $N\to\infty$,  \eqref{eq.mfgV.pivarpi}, \eqref{eq.muNx.converge}, \eqref{eq.thm.converge1}, \eqref{eq:thm1} and Assumption \ref{assum.lipschitz} together yield that
		\be\label{eq.thm.converge3}
	\lim_{N\to\infty}J^{{N},1,(\pitildebm^{{N}}_{-1}, \varpi){\color{black}\otimes_{1}}\pitildebm^{{N}}}(x,\nu(N,x))=J^{\varpi\otimes_1 \pitilde^\infty} \left( x,\nu, \mu^{\pitilde^\infty}_{[1:\infty)} \right)\quad \forall \varpi\in \Pc(U).
	\ee
Also, \eqref{eq:rm.new} in Remark \ref{rm:new} states that
\be\label{eq.thm.converge5}
\lim_{N\to\infty}J^{{N},1,\pitildebm^{N}}(x, \nu(N,x))
=J^{\pitilde^\infty}(x,\nu).
\ee}
Then combining \eqref{eq.thm.converge}, \eqref{eq.thm.converge3} and \eqref{eq.thm.converge5} results in
	$$
	J^{\pitilde^\infty}(x,\nu)=\sup_{\varpi\in \Pc(U)} J^{\varpi\otimes_1 \pitilde^\infty}\left(x,\nu, \mu^{\pitilde^\infty}_{[1:\infty)} \right).
	$$
	Then the desired result follows from the arbitrariness of $(x,\nu)$.
\end{proof}

{\color{black}
\section{Conclusion}\label{sec:conclusion}

In this paper, we study time-inconsistent MFGs with non-exponential discounting in discrete time. 
	We first introduce the classic equilibrium concept for the time-inconsistent MFG and provide a general existence result. 
This equilibrium concept coincides with the classic equilibrium in the MFG literature when the context is time-consistent.	
We then demonstrate that such a classic equilibrium functions as an approximate equilibrium for the corresponding time-inconsistent $N$-agent games, but only in a precommitment sense, and thus is not a true equilibrium from the perspective of a sophisticated agent. 
To address this limitation, we introduce the concept of consistent equilibrium for both the MFG and the $N$-agent game, and we prove that a consistent MFG equilibrium serves as a consistent $\eps$-equilibrium for the $N$-agent game when $N$ is sufficiently large. We also establish the convergence of consistent $N$-agent equilibria to a consistent MFG equilibrium as $N$ tends to infinity. In addition, we present a case study, in which the existence of consistent MFG equilibria is provided.
}

\bibliographystyle{plain}
\bibliography{reference}

\begin{thebibliography}{10}

\bibitem{adlakha2015equilibria}
Sachin Adlakha, Ramesh Johari, and Gabriel~Y. Weintraub.
\newblock Equilibria of dynamic games with many players: existence,
  approximation, and market structure.
\newblock {\em J. Econom. Theory}, 156:269--316, 2015.

\bibitem{guide2006infinite}
Charalambos~D. Aliprantis and Kim~C. Border.
\newblock {\em Infinite dimensional analysis}.
\newblock Springer, Berlin, third edition, 2006.
\newblock A hitchhiker's guide.

\bibitem{anahtarci2023q}
Berkay Anahtarci, Can~Deha Kariksiz, and Naci Saldi.
\newblock Q-learning in regularized mean-field games.
\newblock {\em Dynamic Games and Applications}, 13(1):89--117, 2023.

\bibitem{bayraktar2023mean}
Erhan Bayraktar, Alekos Cecchin, and Prakash Chakraborty.
\newblock Mean field control and finite agent approximation for
  regime-switching jump diffusions.
\newblock {\em Applied Mathematics \& Optimization}, 88(2):36, 2023.

\bibitem{bayraktar2022finite}
Erhan Bayraktar, Alekos Cecchin, Asaf Cohen, and Fran{\c{c}}ois Delarue.
\newblock Finite state mean field games with wright--fisher common noise as
  limits of n-player weighted games.
\newblock {\em Mathematics of Operations Research}, 47(4):2840--2890, 2022.

\bibitem{bayraktar2018analysis}
Erhan Bayraktar and Asaf Cohen.
\newblock Analysis of a finite state many player game using its master
  equation.
\newblock {\em SIAM J. Control Optim.}, 56(5):3538--3568, 2018.

\bibitem{bayraktar2023existence}
Erhan Bayraktar and Bingyan Han.
\newblock Existence of markov equilibrium control in discrete time.
\newblock {\em SIAM Journal on Financial Mathematics}, 14(4):SC60--SC71, 2023.

\bibitem{bayraktar2023relaxed}
Erhan Bayraktar, Yu-Jui Huang, Zhenhua Wang, and Zhou Zhou.
\newblock Relaxed equilibria for time-inconsistent markov decision processes.
\newblock {\em To appear in Math. Oper. Res., arXiv preprint arXiv:2307.04227},
  2023.

\bibitem{bjork2021time}
Tomas Bj\"{o}rk, Mariana Khapko, and Agatha Murgoci.
\newblock {\em Time-inconsistent control theory with finance applications}.
\newblock Springer Finance. Springer, Cham, [2021] \copyright 2021.

\bibitem{bogachev2007measure}
V.~I. Bogachev.
\newblock {\em Measure theory. {V}ol. {I}, {II}}.
\newblock Springer-Verlag, Berlin, 2007.

\bibitem{budhiraja2015long}
Amarjit Budhiraja and Abhishek~Pal Majumder.
\newblock Long time results for a weakly interacting particle system in
  discrete time.
\newblock {\em Stoch. Anal. Appl.}, 33(3):429--463, 2015.

\bibitem{cardaliaguet2019master}
Pierre Cardaliaguet, Fran\c{c}ois Delarue, Jean-Michel Lasry, and Pierre-Louis
  Lions.
\newblock {\em The master equation and the convergence problem in mean field
  games}, volume 201 of {\em Annals of Mathematics Studies}.
\newblock Princeton University Press, Princeton, NJ, 2019.

\bibitem{carmona2013probabilistic}
Ren\'{e} Carmona and Fran\c{c}ois Delarue.
\newblock Probabilistic analysis of mean-field games.
\newblock {\em SIAM J. Control Optim.}, 51(4):2705--2734, 2013.

\bibitem{djehiche2023time}
Boualem Djehiche and Mattia Martini.
\newblock Time-inconsistent mean-field optimal stopping: a limit approach.
\newblock {\em J. Math. Anal. Appl.}, 528(1):Paper No. 127582, 26, 2023.

\bibitem{MR4124420}
Sebastian Ebert, Wei Wei, and Xun~Yu Zhou.
\newblock Weighted discounting---on group diversity, time-inconsistency, and
  consequences for investment.
\newblock {\em J. Econom. Theory}, 189:105089, 40, 2020.

\bibitem{ekeland2012time}
Ivar Ekeland, Oumar Mbodji, and Traian~A. Pirvu.
\newblock Time-consistent portfolio management.
\newblock {\em SIAM J. Financial Math.}, 3(1):1--32, 2012.

\bibitem{fischer2017connection}
Markus Fischer.
\newblock On the connection between symmetric {$N$}-player games and mean field
  games.
\newblock {\em Ann. Appl. Probab.}, 27(2):757--810, 2017.

\bibitem{gomes2013continuous}
Diogo~A. Gomes, Joana Mohr, and Rafael Rig\~{a}o Souza.
\newblock Continuous time finite state mean field games.
\newblock {\em Appl. Math. Optim.}, 68(1):99--143, 2013.

\bibitem{guo2019learning}
Xin Guo, Anran Hu, Renyuan Xu, and Junzi Zhang.
\newblock Learning mean-field games.
\newblock {\em Advances in neural information processing systems}, 32, 2019.

\bibitem{guo2023general}
Xin Guo, Anran Hu, Renyuan Xu, and Junzi Zhang.
\newblock A general framework for learning mean-field games.
\newblock {\em Math. Oper. Res.}, 48(2):656--686, 2023.

\bibitem{MR4288523}
Yu-Jui Huang and Zhou Zhou.
\newblock Strong and weak equilibria for time-inconsistent stochastic control
  in continuous time.
\newblock {\em Math. Oper. Res.}, 46(2):428--451, 2021.

\bibitem{lacker2016general}
Daniel Lacker.
\newblock A general characterization of the mean field limit for stochastic
  differential games.
\newblock {\em Probab. Theory Related Fields}, 165(3-4):581--648, 2016.

\bibitem{lacker2017limit}
Daniel Lacker.
\newblock Limit theory for controlled {M}c{K}ean-{V}lasov dynamics.
\newblock {\em SIAM J. Control Optim.}, 55(3):1641--1672, 2017.

\bibitem{moon2020linear}
Jun Moon and Hyun~Jong Yang.
\newblock Linear-quadratic time-inconsistent mean-field type {S}tackelberg
  differential games: time-consistent open-loop solutions.
\newblock {\em IEEE Trans. Automat. Control}, 66(1):375--382, 2021.

\bibitem{ni2017time}
Yuan-Hua Ni, Ji-Feng Zhang, and Miroslav Krstic.
\newblock Time-inconsistent mean-field stochastic {LQ} problem: open-loop
  time-consistent control.
\newblock {\em IEEE Trans. Automat. Control}, 63(9):2771--2786, 2018.

\bibitem{saldi2018markov}
Naci Saldi, Tamer Ba\c{s}ar, and Maxim Raginsky.
\newblock Markov-{N}ash equilibria in mean-field games with discounted cost.
\newblock {\em SIAM J. Control Optim.}, 56(6):4256--4287, 2018.

\bibitem{strotz1955myopia}
R.~H. Strotz.
\newblock Myopia and inconsistency in dynamic utility maximization.
\newblock {\em The Review of Economic Studies}, 23(3):165--180, 1955.

\bibitem{wang2023time}
Haiyang Wang and Ruimin Xu.
\newblock Time-inconsistent {LQ} games for large-population systems and
  applications.
\newblock {\em J. Optim. Theory Appl.}, 197(3):1249--1268, 2023.

\bibitem{wang2024sharp}
Ziyuan Wang and Zhou Zhou.
\newblock Sharp equilibria for time-inconsistent mean-field stopping games.
\newblock {\em SIAM Journal on Control and Optimization}, 62(4):2319--2345,
  2024.

\bibitem{yardim2023policy}
Batuhan Yardim, Semih Cayci, Matthieu Geist, and Niao He.
\newblock Policy mirror ascent for efficient and independent learning in mean
  field games.
\newblock In {\em International Conference on Machine Learning}, pages
  39722--39754. PMLR, 2023.

\bibitem{yong2012time}
Jiongmin Yong.
\newblock Time-inconsistent optimal control problems and the equilibrium {HJB}
  equation.
\newblock {\em Math. Control Relat. Fields}, 2(3):271--329, 2012.

\bibitem{yu2023time}
Xiang Yu and Fengyi Yuan.
\newblock Time-inconsistent mean-field stopping problems: A regularized
  equilibrium approach.
\newblock {\em arXiv preprint arXiv:2311.00381}, 2023.

\end{thebibliography}

\end{document}